\documentclass[12pt,reqno]{amsart}
\usepackage[margin=1in]{geometry}
\usepackage{amsmath,amssymb,amsthm,graphicx,amsxtra, setspace}
\usepackage[utf8]{inputenc}
\usepackage{mathrsfs}
\usepackage{hyperref}
\usepackage{upgreek}
\usepackage{mathtools}
\usepackage[mathcal]{euscript}
\usepackage{xcolor}
\allowdisplaybreaks

\usepackage[pagewise]{lineno}

\newtheorem{theorem}{Theorem}[section]
\newtheorem{lemma}[theorem]{Lemma}

\newtheorem{corollary}[theorem]{Corollary}
\newtheorem{definition}[theorem]{Definition}

\newtheorem{remark}[theorem]{Remark}

\let\originalleft\left
\let\originalright\right
\renewcommand{\left}{\mathopen{}\mathclose\bgroup\originalleft}
\renewcommand{\right}{\aftergroup\egroup\originalright}

\renewcommand{\d}{\/\mathrm{d}\/}

\def\w{\textbf{W}^{\varepsilon}_{{\theta}^{\varepsilon}}}

\def\e{\varepsilon}

\def\L{\mathbb{L}}
\def\A{\mathrm{A}}

\def\C{\mathrm{C}}
\def\f{\boldsymbol{f}}

\def\B{\mathrm{B}}
\def\D{\mathrm{D}}
\def\y{\boldsymbol{y}}

\def\X{\mathbb{X}}
\def\x{\boldsymbol{x}}
\def\k{\boldsymbol{k}}

\def\z{\boldsymbol{z} }
\def\v{\boldsymbol{v}}
\def\w{\boldsymbol{w}}
\def\W{\mathrm{W}}

\def\V{\mathbb{V}}
\def\wi{\widetilde}

\def\u{\mathrm{U}}

\def\u{\boldsymbol{u}}
\def\H{\mathbb{H}}

\newcommand{\R}{\mathbb{R}}

\renewcommand{\d}{\/\mathrm{d}\/}

\newcommand{\Addresses}{{
		\footnote{
			\noindent \textsuperscript{1,2}Department of Mathematics, Indian Institute of Technology Roorkee-IIT Roorkee,
			Haridwar Highway, Roorkee, Uttarakhand 247667, INDIA.\par\nopagebreak
			\noindent  \textit{e-mail:} \texttt{Manil T. Mohan: maniltmohan@ma.iitr.ac.in, maniltmohan@gmail.com.}
			
			\textit{e-mail:} \texttt{Kush Kinra: kkinra@ma.iitr.ac.in.}
			
			\noindent \textsuperscript{*}Corresponding author.
			
			\textit{Key words:} Deterministic and stochastic convective Brinkman-Forchheimer equations, small forcing intensity, singleton attractor, upper semicontinuity, lower semicontinuity.
			
			Mathematics Subject Classification (2020): Primary 35B41, 35Q35; Secondary 37L55, 37N10, 35R60.

}}}

\begin{document}
	
	\title[Convergence of random attractor towards singleton attractor]{Convergence of random attractors towards deterministic singleton attractor for 2D and 3D convective Brinkman-Forchheimer equations
		\Addresses}
	
	\author[K. Kinra and M. T. Mohan]
	{Kush Kinra\textsuperscript{1} and Manil T. Mohan\textsuperscript{2*}}

	\maketitle
	
	\begin{abstract}
		This work deals with the asymptotic behavior of the two  as well as three dimensional convective Brinkman-Forchheimer (CBF) equations in periodic domains: $$\frac{\partial\boldsymbol{u}}{\partial t}-\mu \Delta\boldsymbol{u}+(\boldsymbol{u}\cdot\nabla)\boldsymbol{u}+\alpha\boldsymbol{u}+\beta|\boldsymbol{u}|^{r-1}\boldsymbol{u}+\nabla p=\boldsymbol{f},\ \nabla\cdot\boldsymbol{u}=0,$$ where $r\geq1$. We prove that the global attractor of the above system is a singleton under small forcing intensity ($r\geq 1$ for $n=2$ and $r\geq 3$ for $n=3$ with $2\beta\mu\geq 1$ for $r=n=3$). After perturbing the above system with additive or multiplicative white noise,  the random attractor does not have a singleton structure. But we obtain that the random attractor for 2D stochastic CBF equations with additive and multiplicative white noise converges towards the deterministic singleton attractor for $1\leq r\leq 2$ and $1\leq r<\infty$, respectively, when the coefficient of random perturbation converges to zero (upper and lower semicontinuity). Interestingly in the case of 3D stochastic CBF equations perturbed by multiplicative white noise, we are able to establish that the random attractor converges towards the deterministic singleton attractor for $3\leq r\leq 5$ ($2\beta\mu\geq 1$ for $r=3$), when the coefficient of random perturbation converges to zero.
	\end{abstract}

	\section{Introduction} \label{sec1}\setcounter{equation}{0}
	Attractors have great importance in mathematical physics, especially in the study of the long-term behavior of dynamical systems generated by fluid flow models. For the deterministic infinite dimensional dynamical systems, a good number of works are available on the  attractors, see \cite{ICh,JCR,R.Temam}, etc for a comprehensive study on the subject. An extensive study on perturbations of evolution equations by stochastic noise and the generation of random dynamical systems as well as random attractors for stochastic evolution equations has been carried out in \cite{Arnold}. Using the past information of the system, one can consider the random attractor as pathwise pullback attracting, but it is also forward attracting in probability (\cite{HCJA}). The theory of pathwise pullback random attractors was first introduced in \cite{CF,FS}, and thereafter the existence of random attractors for several SPDEs is proved in \cite{BLL,BLW,BGT,Crauel1,CDF,FY,GLS,You} etc and the references therein. The work \cite{CCLR} gave  attention to the very different effects that various types of noises can have on the asymptotic behavior of deterministic systems.
	
	An important property of attractors is known as the \emph{stability of the attractor under stochastic perturbations}, that is, the convergence of random attractor	$\mathcal{A}_{\varepsilon}(\omega)$ towards the deterministic global attractor $\mathcal{A}$, when the coefficient of random perturbation $\varepsilon$ converges to zero. For a metric space $(\mathbb{X},d)$, we know that each attractor is expected to be a compact set that are invariant and attract certain kinds of sets like bounded subsets of the state space, and is not necessary singleton. For this attractor, there are two types of convergences, which are defined as follows:
	\begin{itemize}
		\item [(i)] Upper semicontinuity, that is,
		\begin{align*}
		\lim_{\varepsilon\to0^{+}} \text{dist}\left(\mathcal{A}_{\varepsilon}(\omega),\mathcal{A}\right)=0.
		\end{align*}
		\item [(ii)] Lower semicontinuity, that is,
		\begin{align*}
		\lim_{\varepsilon\to0^{+}} \text{dist}\left(\mathcal{A},\mathcal{A}_{\varepsilon}(\omega)\right)=0,
		\end{align*}
	\end{itemize}
	where $\text{dist}(\cdot,\cdot)$ denotes the Hausdorff semi-distance between two non-empty subsets of $\mathbb{X}$, that is, for non-empty sets $A,B\subset \mathbb{X}$ $$\text{dist}(A,B)=\sup_{a\in A}\inf_{b\in B} d(a,b).$$ The upper semicontinuity of random attractors is easy to establish in comparison with lower semicontinuity of random attractors. For upper semicontinuity, we need only general study on the dynamical systems and attractors, and the conditions under which upper semicontinuity are established are not too restrictive, see \cite{CLR,KM1,KM3,Wang}, etc. But for lower semicontinuity, we require more detailed study either on the structure of the deterministic attractor or on the equi-attraction of the family of the random attractors of perturbed systems, see \cite{BV,CLR1,HR1,LK}, etc for more details.

	Let us now discuss the mathematical model concerning convective Brinkman-Forchheimer (CBF) equations in periodic domains.	Let $\mathbb{T}^n=[0,L]^n$ ($n=2,3$) be a periodic domain ($n$ dimensional torus). We consider CBF equations, which describe the motion of Incompressible fluid flows in a saturated porous medium, as (\cite{PAM}):
	\begin{equation}\label{1}
	\left\{
	\begin{aligned}
	\frac{\partial \u}{\partial t}-\mu \Delta\u+(\u\cdot\nabla)\u+\alpha\u+\beta|\u|^{r-1}\u+\nabla p&=\boldsymbol{f}, \ \text{ in } \ \mathbb{T}^n\times(0,\infty), \\ \nabla\cdot\u&=0, \ \text{ in } \ \mathbb{T}^n\times(0,\infty), \\
	\u(0)&=\x, \ \text{ in } \ \mathbb{T}^n,
	\end{aligned}
	\right.
	\end{equation}
	where $\u(x,t) :
	\mathbb{T}^n\times(0,\infty)\to \R^n$ denotes the velocity field, $p(x,t):
	\mathbb{T}^n\times(0,\infty)\to\R$ represents the pressure field and $\f(x,t):
	\mathbb{T}^n\times(0,\infty)\to \R^n$ is an external forcing. Also $\u(x,t)$ satisfies the periodic boundary conditions, that is,
	\begin{align}\label{2}
	\u(x+\mathrm{L}e_i)=\u(x),
	\end{align}
	for every $x\in\R^n$ and $i=1,\ldots,n$, where
	$\{e_1,\ldots,e_n\}$ is the canonical basis
	of $\R^n$. The positive constants $\mu,\alpha$ and $\beta$ represent  the Brinkman coefficient (effective viscosity), the Darcy coefficient (permeability of porous medium) and Forchheimer coefficient, respectively. The absorption exponent $1\leq r<\infty$ and  $r=3$ is known as the critical exponent.  It has been proved in Proposition 1.1, \cite{HR} that the critical homogeneous CBF equations \eqref{1} have the same scaling as Navier-Stokes equations (NSE) only when $\alpha=0$, which is sometimes referred to as NSE  modified by an absorption term (\cite{SNA}) or the tamed NSE (\cite{MRXZ}), and no scale invariance property for other values of $\alpha$ and $r$.  For $\alpha=\beta=0$, we obtain the classical NSE and one can consider the system \eqref{1} as damped NSE with the nonlinear damping $\alpha\u+\beta|\u|^{r-1}\u$.  The applicability of CBF equations \eqref{1} is limited to flows when the velocities are sufficiently high and porosities are not too small, that is, when the Darcy law for a porous medium no longer applies and we use the term \emph{non-Darcy models} for these types of flows (see \cite{PAM}). In order to have better understanding of the effect of nonlinear damping and a comparison with  NSE, we put $\alpha=0$ in our further analysis.
	
	Let us now discuss the solvability results concerning the system \eqref{1}. For deterministic CBF equations, the result on the existence of unique weak as well as strong solutions is established in \cite{SNA,FHR,HR,PAM,MTM}, etc. For stochastic CBF equations, the existence of strong solutions (in the probabilistic sense) and martingale solutions are established in \cite{MTM1} and \cite{LHGH1}, respectively. The existence of a unique pathwise strong solution for 3D stochastic NSE is a well known open problem, similarly it is open for 3D CBF equations for $1\leq r<3$. Therefore, for $n=3$, we restrict ourselves to $r\geq 3$ ($r>3$, for any $\mu, \beta>0$, and  $r=3,$  for $2\beta\mu\geq1$) (see \cite{MTM1} for more details).
	
	Next, we discuss the literature on different types of attractor for  the model under our consideration. The existence of global attractors for 2D CBF equations in bounded and unbounded domains, and 3D CBF equations in periodic domains is established in \cite{MTM2,MTM3} and \cite{KM3}, respectively.  In the case of stochastic CBF equations, the existence of random attractors as well as upper semicontinuity results in 2D bounded or unbounded like Poincar\'e domains is proved in \cite{KM,KM1,KM2}, etc. The existence of random attractors in 3D periodic domains is established in \cite{KM3}. Moreover, the existence of weak pullback mean random attractors for 2D as well as 3D stochastic CBF equations is discussed in \cite{KM4}. 

	The main reason for considering the CBF equations \eqref{1} in periodic domains is as follows. 	In periodic domains, as the Helmholtz-Hodge (or Leray) projection $\mathcal{P}$ and $-\Delta$ commutes, the equality  
	\begin{align}\label{3}
	&	\int_{\mathbb{T}^n}(-\Delta\u(x))\cdot|\u(x)|^{r-1}\u(x)d x\nonumber\\&=\int_{\mathbb{T}^n}|\nabla\u(x)|^2|\u(x)|^{r-1}\d x+4\left[\frac{r-1}{(r+1)^2}\right]\int_{\mathbb{T}^n}|\nabla|\u(x)|^{\frac{r+1}{2}}|^2\d x.
	\end{align} is quite useful in obtaining regularity results. It is also noticed in the literature that the above equality may not be useful except in whole domain or periodic domains (see \cite{KT2,MTM}, etc for a detailed discussion). Moreover, we also point out that the equality \eqref{3} is not used in the works \cite{KM,KM1,KM2,MTM2,MTM3} etc, where the problem \eqref{1} under consideration is bounded or unbounded domains (Poincar\'e domain). 
	
	There objectives of this work are two folded, the first one is to prove the existence of deterministic singleton attractor for 2D as well as 3D CBF equations, and, the second one is to prove the convergence of random attractors towards the deterministic singleton attractor when coefficient of random perturbation converges to zero. For sufficiently small forcing intensity (see \eqref{C_1} for 2D case and, \eqref{3D-C_1} and \eqref{3D-C_3} for 3D case), it is proved that the deterministic global attractor is singleton. Furthermore, with additive and multiplicative white noise perturbation, the random attractors $\mathcal{A}_{\varepsilon}(\omega)$ converges towards deterministic singleton attractor $\mathcal{A}$ both upper and lower continuously, that is, $$\text{dist}_{\H}\left(\mathcal{A}_{\varepsilon}(\omega),\mathcal{A}\right)\leq\updelta_{\varepsilon}(\omega)$$where $\text{dist}_{\H}(A,B)=\max\{\text{dist}(A,B),\text{dist}(B,A)\},$ and each $\updelta_{\varepsilon}$ is a random variable such that $\updelta_{\varepsilon}(\omega)\sim\varepsilon^{\delta}$ as $\varepsilon\to 0^{+}$ for some $\delta>0$. Even though the upper semicontinuity of random attractor for 2D and 3D stochastic CBF equations is proved in \cite{KM1} and \cite{KM3}, respectively, as we are  interested in the convergence rate, more careful study is needed. The existence of a singleton attractor for 2D deterministic Navier-Stokes equations (NSE) with small forcing intensity and convergence of random attractors for 2D stochastic NSE towards deterministic singleton attractor is discussed in \cite{HCPEK}.
	
	As we have discussed earlier, the effect of noise on attractors is remarkable (see \cite{CCLR}) and hence we need to restrict ourselves to some special cases in different kinds of noises. Also the results obtained for additive noise are not same as the case of multiplicative noise with respect to absorption exponent $ r\in[1,\infty)$. The major outcomes of this work are as follows:
	\begin{itemize}
		\item The existence of singleton attractor for 2D CBF equations with small intensity (cf. \eqref{C_1}), for every $1\leq r <\infty$.
		\item The existence of singleton attractor for 3D CBF equations with small intensity for every $3\leq r <\infty$ (see \eqref{3D-C_1} for $r>3$  and \eqref{3D-C_3} for $r=3$).
		\item Under the small forcing intensity condition \eqref{C_1},  convergence of random attractors of 2D stochastic CBF equations perturbed by additive white noise towards the deterministic singleton attractor for $1\leq r\leq 2$ with the order of convergence $\e^{\frac{r+1}{2r}}$.
		\item Under the condition \eqref{C_1}, convergence of random attractors of 2D stochastic CBF equations subjected to multiplicative white noise towards the deterministic singleton attractor for every $1\leq r<\infty$ with the order of convergence $\e$, which is same order as the perturbation intensity.  
		\item Convergence of random attractors of 3D stochastic CBF equations with multiplicative white noise towards the deterministic singleton attractor for $3\leq r\leq5$ with the order of convergence $\e$ under the conditions \eqref{3D-C_1} for $3<r\leq 5$  and \eqref{3D-C_3} for $r=3$. 
	\end{itemize}
	The organization of rest of the paper is as follows. In the next section, we first define some function spaces on periodic domains, which are needed to prove the results of this work. Secondly, we define some operators in order  to transform \eqref{1} into an abstract form and discuss some estimates of these operators which we need in our further analysis. Abstract formulation of \eqref{1} and solvability results are also discussed in the same section. In section \ref{sec3}, we first consider an assumption on the so-called Grashof number (or small forcing intensity), see \eqref{C_1}, \eqref{3D-C_1} and \eqref{3D-C_3}. Under this assumption, the existence of singleton attractors for 2D and 3D CBF equations is proved in Theorems \ref{D-SA} and \ref{3D-SA}, respectively. Two dimensional stochastic CBF equations with additive noise is considered in section \ref{sec4}.  We prove the convergence of random attractors towards the deterministic singleton attractor with the order of convergence $\e^{\frac{r+1}{2r}}$  in Theorem \ref{Conver-add}. We also point out that the order of convergence of the linear perturbation (that is, $r=1$) of stochastic NSE with additive white noise is same as that of  stochastic NSE discussed in \cite{HCPEK}. But random attractors for the nonlinear perturbation (that is, $1<r\leq 2$) of stochastic NSE with additive white noise will convergence towards deterministic singleton attractor slowly than the linear perturbation. Section \ref{sec5} is devoted for establishing  the convergence of random attractors for 2D stochastic CBF equations with multiplicative noise  towards deterministic singleton attractor (Theorem \ref{Conver-multi}).   In the final section, we consider 3D stochastic CBF equations with multiplicative noise and prove the convergence of random attractors towards the deterministic singleton attractor with the order $\e$ in Theorem \ref{3D-Conver-multi}.
	
	\section{Mathematical formulation} \label{sec2}\setcounter{equation}{0}
	In this section, we provide the necessary functional setting needed to obtain the results of this work. We consider the problem \eqref{1} on a $n$-dimensional torus $\mathbb{T}^n=[0,L]^n \ (n=2,3)$,  with the periodic boundary conditions and zero-mean value constraint for the functions, that is, $\int_{\mathbb{T}^n}\u(x)\d x=\mathbf{0}$. 
	\subsection{Function spaces} Let \ $\dot{\C}_p^{\infty}(\mathbb{T}^n;\R^n)$ denote the space of all infinitely differentiable  functions ($\mathbb{R}^n$-valued) such that $\int_{\mathbb{T}^n}\u(x)\d x=\mathbf{0}$ and satisfy periodic boundary conditions \eqref{2}. The Sobolev space  $\dot{\H}_p^k(\mathbb{T}^n):=\dot{\mathrm{H}}_p^k(\mathbb{T}^n;\mathbb{R}^n)$ is the completion of $\dot{\C}_p^{\infty}(\mathbb{T}^n;\R^n)$  with respect to the $\H^s$ norm $$\|\u\|_{\dot{\H}^s_p}:=\left(\sum_{0\leq|\alpha|\leq s}\|\D^{\alpha}\u\|_{\mathbb{L}^2(\mathbb{T}^n)}^2\right)^{1/2}.$$ The Sobolev space of periodic functions with zero mean $\dot{\H}_p^k(\mathbb{T}^n)$ is the same as (Proposition 5.39, \cite{Robinson1}) $$\left\{\u:\u=\sum_{\k\in\mathbb{Z}^n}\u_{\k} e^{2\pi \k\cdot\x /  L},\u_0=\mathbf{0},\ \bar{\u}_{\k}=\u_{-\k},\ \|\u\|_{\dot{\H}^s_f}:=\sum_{k\in\mathbb{Z}^n}|\k|^{2s}|\u_{\k}|^2<\infty\right\}.$$ From Proposition 5.38, \cite{Robinson1}, we infer that the norms $\|\cdot\|_{\dot{\H}^s_p}$ and $\|\cdot\|_{\dot{\H}^s_f}$ are equivalent. Let us define 
	\begin{align*} 
	\mathcal{V}&:=\{\u\in\dot{\C}_p^{\infty}(\mathbb{T}^n;\R^n):\nabla\cdot\u=0\}.
	\end{align*}
	The spaces $\H$ and $\widetilde{\L}^{p}$ are the closure of $\mathcal{V}$ in the Lebesgue spaces $\mathrm{L}^2(\mathbb{T}^n;\R^n)$ and $\mathrm{L}^p(\mathbb{T}^n;\R^n)$ for $p\in(2,\infty)$, respectively. The space $\V$ is the closure of $\mathcal{V}$ in the Sobolev space $\mathrm{H}^1(\mathbb{T}^n;\R^n)$. The zero mean condition provides the well-known \emph{Poincar\'{e} inequality}, \begin{align}\label{poin}
	\lambda_1\|\u\|_{\mathbb{H}}^2\leq\|\u\|^2_{\V},
	\end{align} where $\lambda_1=\frac{4\pi^2}{L^2}$ (Lemma 5.40, \cite{Robinson1}). Then, we characterize the spaces $\H$, $\widetilde{\L}^p$ and $\V$ with the norms  $$\|\u\|_{\H}^2:=\int_{\mathbb{T}^n}|\u(x)|^2\d x,\quad \|\u\|_{\widetilde{\L}^p}^p=\int_{\mathbb{T}^n}|\u(x)|^p\d x\ \text{ and }\ \|\u\|_{\V}^2:=\int_{\mathbb{T}^n}|\nabla\u(x)|^2\d x,$$ respectively. 
	Let $(\cdot,\cdot)$ denote the inner product in the Hilbert space $\H$ and $\langle \cdot,\cdot\rangle $ represent the induced duality between the spaces $\V$  and its dual $\V'$ as well as $\widetilde{\L}^p$ and its dual $\widetilde{\L}^{p'}$, where $\frac{1}{p}+\frac{1}{p'}=1$. Note that $\H$ can be identified with its own dual $\H'$. We endow the space $\V\cap\widetilde{\L}^{p}$ with the norm $\|\u\|_{\V}+\|\u\|_{\widetilde{\L}^{p}},$ for $\u\in\V\cap\widetilde{\L}^p$ and its dual $\V'+\widetilde{\L}^{p'}$ with the norm $$\inf\left\{\max\left(\|\v_1\|_{\V'},\|\v_1\|_{\widetilde{\L}^{p'}}\right):\v=\v_1+\v_2, \ \v_1\in\V', \ \v_2\in\widetilde{\L}^{p'}\right\}.$$ Moreover, we have the continuous embedding $\V\cap\widetilde{\L}^p\hookrightarrow\H\hookrightarrow\V'+\widetilde{\L}^{p'}$. 
	
	\subsection{Linear operator}
	Let $\mathcal{P}_p: \L^p(\mathbb{T}^n) \to\wi\L^p,$ $p\in[1,\infty)$ be the Helmholtz-Hodge (or Leray) projection  (cf.  \cite{JBPCK,DFHM}, etc).	Note that $\mathcal{P}_p$ is a bounded linear operator and for $p=2$,  $\mathcal{P}_2$ is an orthogonal projection  (section 2.1, \cite{RRS}). In the rest of the paper, we use the notation $\mathcal{P}$ for the Helmholtz-Hodge projection.  We define the Stokes operator 
	\begin{equation*}
	\A\u:=-\mathcal{P}\Delta\u,\;\u\in\D(\A):=\V\cap\dot{\H}^{2}_p(\mathbb{T}^n).
	\end{equation*}
	Note that $\D(\A)$ can also be written as $\D(\A)=\big\{\u\in\dot{\H}^{2}_p(\mathbb{T}^n):\nabla\cdot\u=0\big\}$.  It should be noted that $\mathcal{P}$ and $\Delta$ commutes in periodic domains (Lemma 2.9, \cite{RRS}). For the Fourier expansion $\u=\sum\limits_{\k\in\mathbb{Z}^n} e^{2\pi \k\cdot\x /  L}\u_{\k},$ one obtains $$-\Delta\u=\frac{4\pi^2}{L^2}\sum_{\k\in\mathbb{Z}^n} e^{2\pi \k\cdot\x /  L}|\k|^2\u_{\k}.$$ It is easy to observe that $\D(\A^{s/2})=\big\{\u\in \dot{\H}^{s}_p(\mathbb{T}^n):\nabla\cdot\u=0\big\}$ and $\|\A^{s/2}\u\|_{\H}=C\|\u\|_{\dot{\H}^{s}_p},$ for all $\u\in\D(\A^{s/2})$, $s\geq 0$ (see \cite{Robinson1}). Note that the operator $\A$ is a non-negative self-adjoint operator in $\H$ with a compact resolvent and   \begin{align}\label{2.7a}\langle \A\u,\u\rangle =\|\u\|_{\V}^2,\ \textrm{ for all }\ \u\in\V, \ \text{ so that }\ \|\A\u\|_{\V'}\leq \|\u\|_{\V}.\end{align}
	Since $\A^{-1}$ is a compact self-adjoint operator in $\H$, we obtain a complete family of orthonormal eigenfunctions  $\{w_i\}_{i=1}^{\infty}\subset\dot{\C}_p^{\infty}(\mathbb{T}^n;\R^n)$ such that $\A w_i=\lambda_i w_i$, for $i=1,2,\ldots,$ and  $0<\lambda_1\leq \lambda_2\leq \ldots\to\infty$ are the eigenvalues of $\A$. Note that $\lambda_1=\frac{4\pi^2}{L^2}$ is the smallest eigenvalue of $\A$ appearing in the Poincar\'e inequality \eqref{poin}. We also deduce that 
	\begin{align}\label{poin_1}
	\|\A\u\|^2_{\H}=(\A\u, \A\u)&=\sum_{k=1}^{\infty}\lambda_k^2|( \u,w_k)|^2\geq\lambda_1\sum_{k=1}^{\infty}\lambda_k|(\u,w_k)|^2 =\lambda_1\|\nabla\u\|_{\mathbb{H}}^2,
	\end{align}
	for all $\u\in\D(\A)$.

	\subsection{Bilinear operator}
	Let us define the \emph{trilinear form} $b(\cdot,\cdot,\cdot):\V\times\V\times\V\to\R$ by $$b(\u,\v,\w)=\int_{\mathbb{T}^n}(\u(x)\cdot\nabla)\v(x)\cdot\w(x)\d x=\sum_{i,j=1}^n\int_{\mathbb{T}^n}\u_i(x)\frac{\partial \v_j(x)}{\partial x_i}\w_j(x)\d x.$$ If $\u, \v$ are such that the linear map $b(\u, \v, \cdot) $ is continuous on $\V$, the corresponding element of $\V'$ is denoted by $\B(\u, \v)$. We also denote $\B(\u) = \B(\u, \u)=\mathcal{P}(\u\cdot\nabla)\u$.
	An integration by parts gives 
	\begin{equation}\label{b0}
	\left\{
	\begin{aligned}
	b(\u,\v,\w) &=  -b(\u,\w,\v),\ \text{ for all }\ \u,\v,\w\in \V,\\
	b(\u,\v,\v) &= 0,\ \text{ for all }\ \u,\v \in\V.
	\end{aligned}
	\right.\end{equation}
	
	\begin{remark}
		We need the following estimates on the trilinear form $b(\cdot,\cdot,\cdot)$ in the sequel (see Chapter 2, section 2.3, \cite{Temam1}):
		\begin{itemize}
			\item [(i)]
			For $n=2$, 
			\begin{align}
			|b(\u_1,\u_2,\u_3)|&\leq
			c_1\|\u_1\|^{1/2}_{\H}\|\u_1\|^{1/2}_{\V}\|\u_2\|_{\V}\|\u_3\|^{1/2}_{\H}\|\u_3\|^{1/2}_{\V}, \ \text{ for all }\  \u_1, \u_2, \u_3\in \V,\label{b1}\\
			|b(\u_1,\u_2,\u_3)|&\leq 
			c_2\|\u_1\|^{1/2}_{\H}\|\u_1\|^{1/2}_{\V}\|\u_2\|^{1/2}_{\V}\|\A\u_2\|^{1/2}_{\H}\|\u_3\|_{\H}, \label{b2}
			\end{align}
			for all $ \u_1\in \V, \u_2\in \D(\A), \u_3\in \H.$
			\item[(ii)]
			For $n=3$,
			\begin{align}
			|b(\u_1,\u_2,\u_3)|&\leq
			c_3\|\u_1\|^{1/4}_{\H}\|\u_1\|^{3/4}_{\V}\|\u_2\|_{\V}\|\u_3\|^{1/4}_{\H}\|\u_3\|^{3/4}_{\V}, \ \text{ for all }\  \u_1, \u_2, \u_3\in \V.\label{b3}
			\end{align}
		\end{itemize}
		The explicit values of the constants $c_1$ and $c_2$ are obtained in Appendix A, \cite{CFOP}. 
	\end{remark}
	\subsection{Nonlinear operator}
	Let us now consider the operator $\mathcal{C}(\u):=\mathcal{P}(|\u|^{r-1}\u)$. It is immediate that $\langle\mathcal{C}(\u),\u\rangle =\|\u\|_{\widetilde{\L}^{r+1}}^{r+1}$. It can be easily seen that 
	the map $\mathcal{C}(\cdot):\widetilde{\L}^{r+1}\to\widetilde{\L}^{\frac{r+1}{r}}$.  For any $r\in [1, \infty)$ and $\u_1, \u_2 \in \V\cap\widetilde{\L}^{r+1}$, we have (see subsection 2.4, \cite{MTM1})
	\begin{align}\label{MO_c}
	\langle\mathcal{C}(\u_1)-\mathcal{C}(\u_2),\u_1-\u_2\rangle \geq\frac{1}{2}\||\u_1|^{\frac{r-1}{2}}(\u_1-\u_2)\|_{\H}^2+\frac{1}{2}\||\u_2|^{\frac{r-1}{2}}(\u_1-\u_2)\|_{\H}^2\geq 0,
	\end{align}
	and 
	\begin{align}\label{a215}
	\|\u-\v\|_{\wi\L^{r+1}}^{r+1}\leq 2^{r-2}\||\u|^{\frac{r-1}{2}}(\u-\v)\|_{\H}^2+2^{r-2}\||\v|^{\frac{r-1}{2}}(\u-\v)\|_{\H}^2,
	\end{align}
	for $r\geq 1$ (replace $2^{r-2}$ with $1,$ for $1\leq r\leq 2$).
	\subsection{Abstract formulation}
	On taking orthogonal projection $\mathcal{P}$ onto the first equation in \eqref{1}, we obtain 
	\begin{equation}\label{D-CBF}
	\left\{
	\begin{aligned}
	\frac{\d\u(t)}{\d t}+\mu \A\u(t)+\B(\u(t))+\beta \mathcal{C}(\u(t))&=\f , \ \ \ t\geq 0, \\ 
	\u(0)&=\x,
	\end{aligned}
	\right.
	\end{equation}
	where $\x\in \H \text{ and } \f\in \H$. Next, we shall give the definition of weak solution of the system \eqref{D-CBF} and discuss global solvability results for both 2D and 3D cases. 
	\begin{definition}\label{def3.1}
		Let us assume that $\x\in \H$ and $\f\in \H$. Let $\tau>0$ be any fixed time. Then, the function $\u(\cdot)$ is called a \emph{Leray-Hopf weak solution} of the problem \eqref{D-CBF} on time interval $[0,\tau]$, if $$\u\in \mathrm{L}^{\infty}(0,\tau;\H)\cap\mathrm{L}^2(0, \tau;\V)\cap\mathrm{L}^{r+1}(0,\tau;\widetilde{\L}^{r+1}),$$ with $ \partial_t\u\in \mathrm{L}^{2}(0,\tau;\V')+\mathrm{L}^{\frac{r+1}{r}}(0,\tau;\widetilde{\L}^{\frac{r+1}{r}})$ satisfying: 
		\begin{enumerate}
			\item [(i)]  	for any $\psi\in \V\cap\widetilde{\L}^{r+1},$ 	\begin{align*}
			\bigg\langle\frac{\d\u(t)}{\d t}, \psi\bigg\rangle =  - \big\langle \mu \A\u(t)+\B(\u(t))+\beta \mathcal{C}(\u(s)) , \psi\big\rangle +  ( \f, \psi),
			\end{align*}
			for a.e. $t\in[0,\tau]$,
			\item [(ii)] the initial data is satisfied in the following sense: 
			$$\lim\limits_{t\downarrow 0}\int_{\mathcal{O}}\u(t,\x)\psi(x)\d x=\int_{\mathcal{O}}\x(x)\psi(x)\d x, $$ for all $\psi\in\H$.  
		\end{enumerate}
		
	\end{definition}
	The following global solvability result is proved in Theorems 3.4 and 3.5, \cite{MTM} by using monotonicity of linear and nonlinear operators and Minty-Browder technique, (see \cite{SNA,HR} also). The following theorem is proved for $1\leq r<\infty$ in the  2D case and $3\leq r<\infty$ ($r=3$ with $2\beta\mu\geq1$) for the 3D case. 
	\begin{theorem}[\cite{MTM}]\label{D-Sol}
		Let $\x \in \H$ and $\f\in \H$ be given. Then there exists a unique Leray-Hopf weak solution $\u(\cdot)$ to the system \eqref{D-CBF} in the sense of Definition \ref{def3.1}. 
		
		Moreover, the solution $\u\in\C([0,\tau];\H)$ and satisfies  the following energy equality: 
		\begin{align*}
		\|\u(t)\|_{\H}^2+2\mu\int_0^t\|\u(s)\|_{\V}^2\d s+2\beta\int_0^t\|\u(s)\|_{\wi\L^{r+1}}^{r+1}\d s=\|\x\|_{\H}^2+2\int_0^t(\f,\u(s))\d s,
		\end{align*}
		for all $t\in[0,\tau]$. 
	\end{theorem}

	\section{Singleton attractor for 2D and 3D CBF equations with small forcing intensity} \label{sec3}\setcounter{equation}{0}
	This section is devoted for establishing the singleton attractor for deterministic CBF equations. We need to make an assumption on the Grashof number, which plays a major role to prove the existence of singleton attractor for the deterministic CBF equations \eqref{D-CBF}.
	\subsection{\textbf{Small Grashof number or small forcing intensity}}\label{subsec3.1}
	Here, we consider the non-dimensional Grashof number $G$ given by 
	\begin{align*}
	G=\frac{\|\f\|_{\H}}{\mu^2\lambda_1},
	\end{align*}
	which measures the intensity of the force $\f$ against the viscosity $\mu$. This number is closely related to the Reynolds number (see \cite{R.Temam}), 
	\begin{align*}
	Re=\frac{\|\f\|^{1/2}_{\H}}{\mu\lambda_1^{1/2}}.
	\end{align*}
	The assumptions on the Grashof number are as follows:
	\begin{itemize}
		\item [(i)] \textbf{For $n=2$ and $r\geq 1$ with any $\beta, \mu>0$}. We assume the small forcing intensity, that is, we assume that the Grashof number satisfies the following:
		\begin{align}\label{C_1}
		G=\frac{\|\f\|_{\H}}{\mu^2\lambda_1}<\frac{1}{c_1}\left[\frac{\mu\lambda_1}{1+\mu\lambda_1+\mu^2\lambda_1^2}\right]^{1/2},
		\end{align}
		or, equivalently, 
		\begin{align}\label{C_2}
		\varrho:=\mu\lambda_1-\frac{c_1^2}{\mu^2}\left[1+\frac{1}{\mu\lambda_1}+\frac{1}{\mu^2\lambda_1^2}\right]\|\f\|^2_{\H}>0,
		\end{align}
		where $c_1$ is the constant appearing in the estimate \eqref{b1}.
		\item [(ii)] \textbf{For $n=3$ and $r>3$ with any $\beta, \mu>0$}. The Grashof number satisfies the following:
		\begin{align}\label{3D-C_1}
		G=\frac{\|\f\|_{\H}}{\mu^2\lambda_1}<\frac{1}{c_3}\left[\frac{4\mu\sqrt{\lambda_1}}{3\sqrt{3}\{2\eta_3+1+(2\eta_3+1)\mu\lambda_1+2\mu^2\lambda_1^2\}}\right]^{1/2},
		\end{align}
		or, equivalently, 
		\begin{align}\label{3D-C_2}
		\varrho_1:=\mu\lambda_1-\frac{27c_3^4}{16\mu^5}\left[2+\frac{2\eta_3+1}{\mu\lambda_1}+\frac{2\eta_3+1}{\mu^2\lambda_1^2}\right]^2\|\f\|^4_{\H}>0,
		\end{align}
		where \begin{align}\label{215}\eta_3=\frac{r-3}{\mu(r-1)}\left[\frac{4}{\beta\mu (r-1)}\right]^{\frac{2}{r-3}},\end{align} 
		and $c_3$ is the constant appearing in the estimate \eqref{b3}. 
		\item [(iii)] \textbf{For $n=3$ and $r=3$ with $2\beta\mu\geq1$}. The Grashof number satisfies the following:
		\begin{align}\label{3D-C_3}
		G=\frac{\|\f\|_{\H}}{\mu^2\lambda_1}<\frac{1}{c_3}\left[\frac{4\mu\sqrt{\lambda_1}}{3\sqrt{3}\{1+\mu\lambda_1+\mu^2\lambda_1^2\}}\right]^{1/2},
		\end{align}
		or, equivalently, 
		\begin{align}\label{3D-C_4}
		\varrho_2:=\mu\lambda_1-\frac{27c_3^4}{16\mu^5}\left[1+\frac{1}{\mu\lambda_1}+\frac{1}{\mu^2\lambda_1^2}\right]^2\|\f\|^4_{\H}>0,
		\end{align}
		where $c_3$ is the constant appearing in the estimate \eqref{b3}.
	\end{itemize}
	
	\subsection{\textbf{Deterministic singleton attractor for 2D CBF equations}}
	Our aim is to show that under the condition \eqref{C_1}, the global attractor $\mathcal{A}$ of the 2D CBF equations \eqref{D-CBF} is a singleton. We start our discussion from the following useful lemma.
	\begin{lemma}\label{d-ab}
		Let $n=2 ,  r\geq 1$ and $\f\in\H$. For any bounded set $B_{\H}\subset\H$ and $\varepsilon>0,$ there exists a time $T_{B_{\H},\varepsilon}>0$ such that any solution $\u(\cdot)$ with initial data in $B_{\H}$ satisfies
		\begin{align}
		\|\u(t)\|^2_{\V}\leq \frac{1}{\mu}\left[1+\frac{1}{\mu\lambda_1}+\frac{1}{\mu^2\lambda_1^2}\right]\|\f\|^2_{\H}+\varepsilon,\ \text{ for all }\ t\geq T_{B_{\H},\varepsilon}.
		\end{align} 
	\end{lemma}
	\begin{proof}
		Taking the inner product with $\u(\cdot)$ to first equation of \eqref{D-CBF} and using \eqref{b0}, we obtain
		\begin{align*}
		\frac{1}{2}\frac{\d}{\d t}\|\u(t)\|^2_{\H} +\mu\|\u(t)\|^2_{\V}+\beta\|\u(t)\|^{r+1}_{\widetilde{\L}^{r+1}}&=(\f,\u(t))\leq \frac{\mu}{2}\|\u\|_{\V}^2+\frac{1}{2\mu\lambda_1}\|\f\|_{\H}^2 ,
		\end{align*}
		for a.e. $t\in[0,\tau]$. 
		Thus, it is immediate that 
		\begin{align}\label{d-ab1}
		\frac{\d}{\d t}\|\u(t)\|^2_{\H} +\mu\lambda_1\|\u(t)\|^2_{\H}\leq\frac{1}{\mu\lambda_1}\|\f\|^2_{\H}.
		\end{align}
		Making use of the Gronwall inequality, we deduce that 
		\begin{align}\label{d-ab2}
		\|\u(t)\|^2_{\H} \leq \|\x\|^2_{\H}e^{-\mu\lambda_1 t} + \frac{1}{\mu^2\lambda_1^2}\|\f\|^2_{\H}(1-e^{-\mu\lambda_1t})\leq \|\x\|^2_{\H}e^{-\mu\lambda_1 t} + \frac{1}{\mu^2\lambda_1^2}\|\f\|^2_{\H},
		\end{align}
		for all $t\geq0$.
		Furthermore, from \eqref{d-ab1}, we infer that
		\begin{align}\label{d-ab3}
		\mu\int_{t}^{t+1}\|\u(s)\|^2_{\V} \d s+2\beta\int_{t}^{t+1}\|\u(s)\|^{r+1}_{\widetilde{\L}^{r+1}}\d s\leq\|\u(t)\|_{\H}^2+\frac{\|\f\|^2_{\H}}{\mu\lambda_1},
		\end{align}
		for all $t\geq0$. From \eqref{d-ab2} and \eqref{d-ab3}, we get
		\begin{align}\label{d-ab4}
		\int_{t}^{t+1}\|\u(s)\|^2_{\V} \d s\leq\frac{1}{\mu}\|\x\|^2_{\H}e^{-\mu\lambda_1 t} + \frac{1}{\mu^3\lambda_1^2}\|\f\|^2_{\H}+\frac{1}{\mu^2\lambda_1}\|\f\|^2_{\H},
		\end{align}
		for all $t\geq0$.	
		
		Taking the inner product with $\A\u(\cdot)$ to first equation of \eqref{D-CBF}, we  get (\cite{FHR,PAM})
		\begin{align}\label{d-ab5}
		&\frac{1}{2}\frac{\d}{\d t}\|\u(t)\|^2_{\V} +\mu\|\A\u(t)\|^2_{\H}+(\B(\u(t)),\A\u(t))+\beta(\mathcal{C}(\u(t)),\A\u(t))\nonumber\\&=(\f,\A\u(t))\leq\frac{\mu}{2}\|\A\u\|^2_{\H} + \frac{1}{2\mu}\|\f\|^2_{\H},
		\end{align}
		for a.e. $t\in[0,\tau]$. 	In 2D periodic domains, we have (see \cite{R.Temam}, Lemma 3.1, page no. 404)
		\begin{align}\label{d-ab6}
		(\B(\u),\A\u)=b(\u, \u,\A\u)=0.
		\end{align}
		From \eqref{3}, we infer that 
		\begin{align}
		(\mathcal{C}(\u),\A\u)&=\||\nabla\u||\u|^{\frac{r-1}{2}}\|^2_{\H}+4\left[\frac{r-1}{(r+1)^2}\right]\|\nabla|\u|^{\frac{r+1}{2}}\|^2_{\H}.\label{d-ab7} 
		\end{align}
		Making use of \eqref{d-ab6}-\eqref{d-ab7} in \eqref{d-ab5}, we obtain
		\begin{align}\label{d-ab9}
		&	\frac{\d}{\d t}\|\u(t)\|^2_{\V}+\mu\|\A\u(t)\|_{\H}^2+2\beta\||\nabla\u(t)||\u(t)|^{\frac{r-1}{2}}\|^2_{\H}+8\beta\left[\frac{r-1}{(r+1)^2}\right]\|\nabla|\u(t)|^{\frac{r+1}{2}}\|^2_{\H}\nonumber\\&\leq \frac{1}{\mu}\|\f\|^2_{\H},
		\end{align}
		for a.e. $t\in[0,\tau]$.	Let us use the double integration trick used in \cite{Robinson} to obtain an absorbing ball in $\V$. Integrating the inequality \eqref{d-ab9} from  $s$ to $t+1$ with $t\leq s<t+1$,  we find 
		\begin{align*}
		\|\u(t+1)\|_{\V}^2&\leq\|\u(s)\|_{\V}^2+ \frac{1}{\mu}\|\f\|^2_{\H}.
		\end{align*}
		Let us now integrate both sides of the above inequality with respect to $s$ between $t$  and $t+1$ to obtain 
		\begin{align*}
		\|\u(t+1)\|_{\V}^2&\leq\int_t^{t+1}\|\u(s)\|_{\V}^2\d s+\frac{1}{\mu}\|\f\|^2_{\H}\leq\frac{1}{\mu}\|\x\|^2_{\H}e^{-\mu\lambda_1 t} +\frac{1}{\mu}\left[1+\frac{1}{\mu\lambda_1}+\frac{1}{\mu^2\lambda_1^2}\right]\|\f\|^2_{\H},
		\end{align*}
		for all $t\geq0$, from which the lemma follows.
	\end{proof}
	\begin{lemma}\label{Uniqueness}
		For $n=2$, let the condition \eqref{C_1} be satisfied. For any bounded set $B_{\H}\subset \H$, there exists a $T>0$ such that for any two solutions $\u_1(\cdot)$ and $\u_2(\cdot)$ of \eqref{D-CBF} with initial data in $B_{\H}$ satisfies
		\begin{align}\label{Unique}
		\|\u_1(t)-\u_2(t)\|^2_{\H} \leq e^{\frac{\varrho}{2}(t-T)} \|\u_1(T)-\u_2(T)\|^2_{\H}, \ \text{ for all }\  t\geq T,
		\end{align}
		where $\varrho$ is the constant given by \eqref{C_2}. 
	\end{lemma}
	\begin{proof}
		Since $\u_1(\cdot)$ and $\u_2(\cdot)$ are the solutions of \eqref{D-CBF}, therefore $\y(\cdot):=\u_1(\cdot)-\u_2(\cdot)$ satisfies the following:
		\begin{align}\label{uni1}
		\frac{\d \y(t)}{\d t} + \mu\A\y (t)&= 
		-\B(\y(t),\u_1(t))+\B(\u_2(t),\y(t)) -\beta\mathcal{C}(\u_1(t))+\beta\mathcal{C}(\u_2(t)),
		\end{align}
		for a.e. $t\in[0,\tau]$. 	Taking the inner product of \eqref{uni1} with $\y(\cdot)$, we get
		\begin{align}\label{uni2}
		&	\frac{1}{2}\frac{\d}{\d t} \|\y(t)\|^2_{\H} + \mu\|\y(t)\|^2_{\V}\nonumber\\&= -b(\y(t),\u_1(t),\y(t))-b(\u_2(t),\y(t),\y(t))-\beta\left\langle\mathcal{C}(\u_1(t))-\mathcal{C}(\u_2(t)),\u_1(t)-\u_2(t)\right\rangle \nonumber\\&\leq -b(\y(t),\u_1(t),\y(t))\leq c_1\|\y(t)\|_{\H}\|\y(t)\|_{\V}\|\u_1(t)\|_{\V}\nonumber\\&\leq \frac{\mu}{2}\|\y(t)\|_{\V}^2+\frac{c_1^2}{2\mu} \|\y(t)\|^2_{\H}\|\u_1(t)\|^2_{\V} , \ \mbox{ for a.e.} \ t\in[0,\tau],
		\end{align}
		where we used \eqref{b0}, \eqref{MO_c} and \eqref{b1}. 
		Using \eqref{poin}, the inequality \eqref{uni2} can be written as
		\begin{align*}
		\frac{\d}{\d t}\|\y(t)\|^2_{\H} + \left(\mu\lambda_1-\frac{c_1^2}{\mu}\|\u_1(t)\|^2_{\V}\right)\|\y(t)\|^2_{\H}\leq 0.
		\end{align*}
		Invoking Lemma \ref{d-ab} for $\varepsilon=\frac{\varrho\mu}{c_1^2}$ (where $\varrho=\mu\lambda_1-\frac{c_1^2}{\mu^2}\left[1+\frac{1}{\mu\lambda_1}+\frac{1}{\mu^2\lambda_1^2}\right]\|\f\|^2_{\H}>0$), one can obtain the existence of  a time $T=T_{B_{\H},\frac{\varrho\mu}{2c_1^2}}>0$ such that
		\begin{align}\label{uni5}
		\mu\lambda_1-\frac{c_1^2}{\mu}\|\u_1(t)\|^2_{\V}\geq\mu\lambda_1-\frac{c_1^2}{\mu}\left(\frac{1}{\mu}\left[1+\frac{1}{\mu\lambda_1}+\frac{1}{\mu^2\lambda_1^2}\right]\|\f\|^2_{\H}+\frac{\varrho\mu}{2c_1^2}\right)=\frac{\varrho}{2},\ \mbox{ for all }\ t\geq T.
		\end{align}
		Hence, we obtain 
		\begin{align*}
		\frac{\d}{\d t}\|\y(t)\|^2_{\H} +\frac{\varrho}{2}\|\y(t)\|^2_{\H}\leq 0,  \ \mbox{ for a.e. } \ t\geq T,
		\end{align*}
		and by applying Gronwall's inequality, we immediate obtain \eqref{Unique}.
	\end{proof}
	\begin{theorem}\label{D-SA}
		Let the condition \eqref{C_1} be satisfied. Then the global attractor $\mathcal{A}$ of the 2D CBF equations \eqref{D-CBF} is a singleton $\mathcal{A}=\{\textbf{a}_{*}\}$ with $\textbf{a}_{*}\in\V$.
	\end{theorem}
	\begin{proof}
		Let us take two arbitrary points $\textbf{a}_1, \textbf{a}_2\in \mathcal{A}$. Since the global attractor $\mathcal{A}$ is a bounded set in $\H$ and consists of bounded complete trajectories in $\H$, there exists two complete trajectories $\upeta_1(t)$ and $\upeta_2(t)$ with initial conditions $\upeta_1(0)=\textbf{a}_1$ and $\upeta_2(0)=\textbf{a}_2$, respectively. Hence, by Lemma \ref{Uniqueness}, we see that there exists a $T=T_{\mathcal{A}}>0$ such that
		\begin{align*}
		\|\textbf{a}_1-\textbf{a}_2\|^2_{\H} &= \|\upeta_1(0)-\upeta_2(0)\|^2_{\H}\\
		&=\|\u(t,\upeta_1(-t))-\u(t,\upeta_2(-t))\|^2_{\H}\\
		&\leq e^{-\frac{\varrho}{2}(t-T)}\|\u(T,\upeta_1(-t))-\u(T,\upeta_2(-t))\|^2_{\H}, \ \text{ for all }\  t\geq T,\\
		&\leq e^{-\frac{\varrho}{2}(t-T)} \left[\text{Diameter} (\mathcal{A})\right]^2, \text{ for all } t\geq T,\\
		& \to 0 \ \text{ as } \ t\to \infty,
		\end{align*}
		which completes the proof.
	\end{proof}
	\begin{remark}
		If we consider initial data in $\V$ and assume that the Grashof number satisfies the following:
		\begin{align}\label{C_3}
		G=\frac{\|\f\|_{\H}}{\mu^2\lambda}<\frac{1}{c_1},
		\end{align}
		or, equivalently, 
		\begin{align}\label{C_4}
		\varrho_{*}:=\mu\lambda_1-\frac{c_1^2 \|\f\|^2_{\H}}{\mu^3\lambda_1}>0,
		\end{align}
		where $c_1$ is the constant appearing in the estimate \eqref{b1}. Then, Lemma \ref{d-ab}, Lemma \ref{Uniqueness} and Theorem \ref{D-SA} can be stated as follows.
		\begin{lemma}\label{d-ab_1}
			Let $n=2, r\geq 1$ and $\f\in\H$. For any bounded set $B_{\V}\subset\V$ and $\varepsilon>0,$ there exists a time $T_{B_{\V},\varepsilon}>0$ such that any solution $\u(\cdot)$ with initial data in $B_{\V}$ satisfies
			\begin{align*}
			\|\u(t)\|^2_{\V}\leq \frac{\|\f\|^2_{\H}}{\mu^2\lambda_1}+\varepsilon, \text{ for all } t\geq T_{B_{\V},\varepsilon}.
			\end{align*} 
		\end{lemma}
		\begin{proof}
			By applying Gronwall's inequality in \eqref{d-ab9}, we get
			\begin{align*}
			\|\u(t)\|^2_{\V}\leq e^{-\mu\lambda_1 t}\|\u(0)\|^2_{\V} +\int_{0}^{t} e^{\mu\lambda_1(\xi-t)} \frac{\|\f\|^2_{\H}}{\mu}\d \xi\leq e^{-\mu\lambda_1 t}\|\u(0)\|^2_{\V} + \frac{\|\f\|^2_{\H}}{\mu^2\lambda_1},
			\end{align*}
			from which the lemma follows.
		\end{proof}
		\begin{lemma}\label{Uniqueness_1}
			For $n=2$, let the condition \eqref{C_3} be satisfied. For any bounded set $B_{\V}\subset \V$, there exists a $T>0$ such that for any two solutions $\u_1(\cdot)$ and $\u_2(\cdot)$ of \eqref{D-CBF} with initial data in $B_{\V}$ satisfies
			\begin{align*}
			\|\u_1(t)-\u_2(t)\|^2_{\H} \leq e^{\frac{\varrho_{*}}{2}(t-T)} \|\u_1(T)-\u_2(T)\|^2_{\H}, \text{ for all } t\geq T,
			\end{align*}
			where $\varrho_{*}$ is the constant given by \eqref{C_4}. 
		\end{lemma}
		\begin{theorem}\label{D-SA_1}
			Let the condition \eqref{C_3} be satisfied. Then the global attractor $\mathcal{A}$ of 2D CBF equations \eqref{D-CBF} is a singleton $\mathcal{A}=\{\textbf{a}_{*}\}$ with $\textbf{a}_{*}\in\D(\A)$.
		\end{theorem}
	\end{remark}
	\begin{corollary}
		The global attractor $\mathcal{A}$ of 2D CBF equations \eqref{D-CBF} with $\f=\bf{0}$ is a singleton $\mathcal{A}=\{\bf{0}\}$.
	\end{corollary}
	\subsection{\textbf{Deterministic singleton attractor for 3D CBF equations}}
	In this subsection, we prove that under the condition \eqref{3D-C_1} (for $n=3$ and $r>3$) and \eqref{3D-C_3} (for $n=r=3$ with $2\beta\mu\geq 1$) the global attractor $\mathcal{A}$ of the 3D CBF equations \eqref{D-CBF} is a singleton. 
	\begin{lemma}\label{3d-ab}
		Let $n=3 \text{ and } r\geq 3$ ($r>3$ for any $\beta,\mu>0$ and $r=3$ for $2\beta\mu\geq1$) and $\f\in\H$. For any bounded set $B_{\H}\subset\H$ and $\varepsilon>0,$ there exists a time $T_{B_{\H},\varepsilon}>0$ such that any solution $\u(\cdot)$ with initial data in $B_{\H}$ satisfies
		\begin{itemize}
			\item [(i)] For $r>3$,
			\begin{align}\label{3D_ab1}
			\|\u(t)\|_{\V}^4&\leq\frac{1}{\mu^2}\left[2+\frac{2\eta_3+1}{\mu\lambda_1}+ \frac{2\eta_3+1}{\mu^2\lambda_1^2}\right]^2\|\f\|^4_{\H}+\varepsilon, \ \text{ for all }\ t\geq T_{B_{\H},\varepsilon}.
			\end{align} 
			\item [(ii)] For $r=3$,
			\begin{align}\label{3D_ab2}
			\|\u(t)\|_{\V}^4&\leq\frac{1}{\mu^2}\left[1+\frac{1}{\mu\lambda_1}+\frac{1}{\mu^2\lambda_1^2}\right]^2\|\f\|^4_{\H}+\varepsilon, \ \text{ for all } \ t\geq T_{B_{\H},\varepsilon}.
			\end{align} 
		\end{itemize}
	\end{lemma}
	\begin{proof}
		Calculation similar to \eqref{d-ab4} gives 
		\begin{align}\label{3d-ab4}
		\int_{t}^{t+1}\|\u(s)\|^2_{\V} \d s\leq\frac{1}{\mu}\|\x\|^2_{\H}e^{-\mu\lambda_1 t} + \frac{1}{\mu^3\lambda_1^2}\|\f\|^2_{\H}+\frac{1}{\mu^2\lambda_1}\|\f\|^2_{\H},
		\end{align}
		for all $t\geq0$. Let us now take the inner product with $\A\u(\cdot)$ to first equation of \eqref{D-CBF} to  get 
		\begin{align}\label{3d-ab5}
		\frac{1}{2}\frac{\d}{\d t}\|\u(t)\|^2_{\V} +\mu\|\A\u(t)\|^2_{\H}+(\B(\u(t)),\A\u(t))+\beta(\mathcal{C}(\u(t)),\A\u(t))=(\f,\A\u(t)),
		\end{align}
		for a.e. $t\in[0,\tau]$. First we consider the case $r>3$ for any $\beta,\mu>0$ and then $r=3$ with $2\beta\mu\geq1$.  
		\vskip 0.2 cm
		\noindent
		\textbf{Case I:}  $r>3$.
		From \eqref{3}, we have
		\begin{align}
		(\mathcal{C}(\u),\A\u)&=\||\nabla\u||\u|^{\frac{r-1}{2}}\|^2_{\H}+4\left[\frac{r-1}{(r+1)^2}\right]\|\nabla|\u|^{\frac{r+1}{2}}\|^2_{\H}.\label{3d-ab9} 
		\end{align}
		Once again the	Cauchy-Schwarz and Young inequalities yield 
		\begin{align}
		|(\f, \A\u)|&\leq \|\f\|_{\H}\|\A\u\|_{\H}\leq\frac{\mu}{4}\|\A\u\|^2_{\H} + \frac{1}{\mu}\|\f\|^2_{\H}.\label{3d-ab10}
		\end{align}
		We estimate $|(\B(\u),\A\u)|$ using H\"older's and Young's inequalities as 
		\begin{align}\label{3d-ab12}
		|(\B(\u),\A\u)|&\leq\||\u||\nabla\u|\|_{\H}\|\A\u\|_{\H}\leq\frac{\mu}{4}\|\A\u\|_{\H}^2+\frac{1}{\mu }\||\u||\nabla\u|\|_{\H}^2. 
		\end{align}
		We  estimate the final term from \eqref{3d-ab12} using H\"older's and Young's inequalities as (similarly as in \cite{MTM1})
		\begin{align}\label{3d-ab13}
		&	\int_{\mathbb{T}^3}|\u(x)|^2|\nabla\u(x)|^2\d x\nonumber\\&=\int_{\mathbb{T}^3}|\u(x)|^2|\nabla\u(x)|^{\frac{4}{r-1}}|\nabla\u(x)|^{\frac{2(r-3)}{r-1}}\d x\nonumber\\&\leq\left(\int_{\mathbb{T}^3}|\u(x)|^{r-1}|\nabla\u(x)|^2\d x\right)^{\frac{2}{r-1}}\left(\int_{\mathbb{T}^3}|\nabla\u(x)|^2\d x\right)^{\frac{r-3}{r-1}}\nonumber\\&\leq{\frac{\beta\mu}{2} }\left(\int_{\mathbb{T}^3}|\u(x)|^{r-1}|\nabla\u(x)|^2\d x\right)+\frac{r-3}{r-1}\left[\frac{4}{\beta\mu (r-1)}\right]^{\frac{2}{r-3}}\left(\int_{\mathbb{T}^3}|\nabla\u(x)|^2\d x\right).
		\end{align}
		Making use of the estimate \eqref{3d-ab13} in \eqref{3d-ab12}, we find
		\begin{align}\label{3d-ab14}
		|(\B(\u),\A\u)|&\leq\frac{\mu}{4}\|\A\u\|_{\H}^2+\frac{\beta}{2}\||\nabla\u||\u|^{\frac{r-1}{2}}\|^2_{\H}+\eta_3\|\u\|^2_{\V},
		\end{align}
		where $\eta_3$ is defined by \eqref{215}.
		Combining \eqref{3d-ab9}-\eqref{3d-ab10} and \eqref{3d-ab14}, and inserting in \eqref{3d-ab5}, we obtain
		\begin{align}\label{3d-ab15}
		&\frac{\d}{\d t}\|\u(t)\|^2_{\V}+\mu\|\A\u(t)\|_{\H}^2+\beta\||\nabla\u(t)||\u(t)|^{\frac{r-1}{2}}\|^2_{\H}+8\beta\left[\frac{r-1}{(r+1)^2}\right]\|\nabla|\u(t)|^{\frac{r+1}{2}}\|^2_{\H}\nonumber\\&\leq 2\eta_3\|\u(t)\|^2_{\V}+\frac{2}{\mu}\|\f\|^2_{\H}.
		\end{align}
		We use the double integration trick used in \cite{Robinson} to obtain an absorbing ball in $\V$. Integrating the inequality \eqref{3d-ab15} from  $s$ to $t+1$, with $t\leq s<t+1$,  we find 
		\begin{align*}
		\|\u(t+1)\|_{\V}^2&\leq\|\u(s)\|_{\V}^2+ 2\eta_3\int_{s}^{t+1}\|\u(\zeta)\|^2_{\V}\d \zeta+\frac{2}{\mu}\|\f\|^2_{\H}\nonumber\\&\leq \|\u(s)\|_{\V}^2+ 2\eta_3\int_{t}^{t+1}\|\u(\zeta)\|^2_{\V}\d \zeta+\frac{2}{\mu}\|\f\|^2_{\H}.
		\end{align*}
		Let us now integrate both sides of the above inequality with respect to $s$ between $t$  and $t+1$ to obtain 
		\begin{align*}
		\|\u(t+1)\|_{\V}^2&\leq(2\eta_3+1)\int_t^{t+1}\|\u(s)\|_{\V}^2\d s+\frac{2}{\mu}\|\f\|^2_{\H}\nonumber\\&\leq(2\eta_3+1)\left[\frac{1}{\mu}\|\x\|^2_{\H}e^{-\mu\lambda_1 t} + \frac{1}{\mu^3\lambda_1^2}\|\f\|^2_{\H}+\frac{1}{\mu^2\lambda_1}\|\f\|^2_{\H}\right]+\frac{2}{\mu}\|\f\|^2_{\H}\nonumber\\&\leq\frac{2\eta_3+1}{\mu}\|\x\|^2_{\H}e^{-\mu\lambda_1 t} +\frac{1}{\mu}\left[2+\frac{2\eta_3+1}{\mu\lambda_1}+ \frac{2\eta_3+1}{\mu^2\lambda_1^2}\right]\|\f\|^2_{\H},\nonumber\\ \|\u(t+1)\|_{\V}^4&\leq\frac{(2\eta_3+1)^2}{\mu^2}\|\x\|^4_{\H}e^{-2\mu\lambda_1 t} +\frac{2(2\eta_3+1)}{\mu^2}\left[2+\frac{2\eta_3+1}{\mu\lambda_1}+ \frac{2\eta_3+1}{\mu^2\lambda_1^2}\right]\|\f\|^2_{\H}\|\x\|^2_{\H}e^{-\mu\lambda_1 t}\nonumber\\&\quad+\frac{1}{\mu^2}\left[2+\frac{2\eta_3+1}{\mu\lambda_1}+ \frac{2\eta_3+1}{\mu^2\lambda_1^2}\right]^2\|\f\|^4_{\H},
		\end{align*}
		for all $t\geq0$, from which the estimate \eqref{3D_ab1} follows.
		\vskip 0.2 cm
		\noindent 
		\textbf{Case II:} $r=3$ and $2\beta\mu\geq1$. Using \eqref{3}, the Cauchy-Schwarz and Young inequalities, we find 
		\begin{align}
		|(\B(\u),\A\u)|&\leq\||\u||\nabla\u|\|_{\H}\|\A\u\|_{\H}\leq\frac{1}{4\beta}\|\A\u\|_{\H}^2+\beta\||\u||\nabla\u|\|_{\H}^2,\label{3d-ab18}\\
		(\mathcal{C}(\u),\A\u)&=\||\nabla\u||\u|\|^2_{\H}+\frac{1}{2}\|\nabla|\u|^2\|^2_{\H},\label{3d-ab19}\\
		|(\f, \A\u)|&\leq \|\f\|_{\H}\|\A\u\|_{\H}\leq\frac{\mu}{2}\|\A\u\|^2_{\H} + \frac{1}{2\mu}\|\f\|^2_{\H}.\label{3d-ab20}
		\end{align}
		Using \eqref{3d-ab18}-\eqref{3d-ab20} in \eqref{3d-ab5}, we obtain
		\begin{align*}
		\frac{\d}{\d t}\|\u(t)\|^2_{\V}+\bigg(\mu-\frac{1}{2\beta}\bigg)\|\A\u(t)\|^2_{\H}+\beta\|\nabla|\u|^2\|^2_{\H} &\leq\frac{1}{\mu}\|\f\|^2_{\H}.
		\end{align*}
		For $2\beta\mu\geq 1$, performing calculations similar to the case of $r>3$, we deduce that 
		\begin{align*}
		\|\u(t+1)\|_{\V}^2&\leq\int_t^{t+1}\|\u(s)\|_{\V}^2\d s+\frac{1}{\mu}\|\f\|^2_{\H}\leq\frac{1}{\mu}\|\x\|^2_{\H}e^{-\mu\lambda_1 t} +\frac{1}{\mu}\left[1+\frac{1}{\mu\lambda_1}+\frac{1}{\mu^2\lambda_1^2}\right]\|\f\|^2_{\H},\nonumber\\ \|\u(t+1)\|_{\V}^4&\leq\frac{1}{\mu^2}\|\x\|^4_{\H}e^{-2\mu\lambda_1 t} +\frac{2}{\mu^2} \left[1+\frac{1}{\mu\lambda_1}+\frac{1}{\mu^2\lambda_1^2}\right]\|\f\|^2_{\H}\|\x\|^2_{\H}e^{-\mu\lambda_1 t}\nonumber\\&\quad+\frac{1}{\mu^2}\left[1+\frac{1}{\mu\lambda_1}+\frac{1}{\mu^2\lambda_1^2}\right]^2\|\f\|^4_{\H},
		\end{align*}
		for all $t\geq 0$, and hence the estimate \eqref{3D_ab2} follows.  
	\end{proof}
	\begin{lemma}\label{3DUniqueness}
		For $n=3$, let the conditions \eqref{3D-C_1} (for $r>3$ with any $\beta,\mu>0$) and \eqref{3D-C_3} (for $r=3$ with $2\beta\mu\geq1$) be satisfied. For any bounded set $B_{\H}\subset \H$, there exist $T_1>0$ (for $r>3$ with any $\beta,\mu>0$) and $T_2>0$ (for $r=3$ with $2\beta\mu\geq1$) such that for any two solutions $\u_1(\cdot)$ and $\u_2(\cdot)$ of \eqref{D-CBF} with initial data in $B_{\H}$ satisfies:
		\begin{itemize}
			\item [(i)] For $r>3$ with any $\beta,\mu>0$, 	\begin{align}\label{3DUnique1}
			\|\u_1(t)-\u_2(t)\|^2_{\H} \leq e^{\frac{\varrho_1}{2}(t-T)} \|\u_1(T)-\u_2(T)\|^2_{\H}, \ \text{ for all }\  t\geq T_1,
			\end{align}
			where $\varrho_1$ is the constant given by \eqref{3D-C_2}. 
			\item [(ii)] For $r=3$ with $2\beta\mu\geq1$, 	\begin{align}\label{3DUnique2}
			\|\u_1(t)-\u_2(t)\|^2_{\H} \leq e^{\frac{\varrho_2}{2}(t-T)} \|\u_1(T)-\u_2(T)\|^2_{\H}, \ \text{ for all }\  t\geq T_2,
			\end{align}
			where $\varrho_2$ is the constant given by \eqref{3D-C_4}.
		\end{itemize}
	\end{lemma}
	\begin{proof}
		A calculation similar to \eqref{uni2} and  	use of \eqref{b3} yields 
		\begin{align}\label{3d-uni3}
		\frac{1}{2}\frac{\d}{\d t} \|\y(t)\|^2_{\H} + \mu\|\y(t)\|^2_{\V}&\leq -b(\y(t),\u_1(t),\y(t))\leq  c_3\|\y(t)\|^{1/2}_{\H}\|\y(t)\|^{3/2}_{\V}\|\u_1(t)\|_{\V}\nonumber\\&\leq \frac{\mu}{2}\|\y(t)\|_{\V}^2+\frac{27c_3^4}{32\mu^3} \|\y(t)\|^2_{\H}\|\u_1(t)\|^4_{\V},
		\end{align}
		for a.e. $t\in[0,\tau]$.	Also, by Poincar\'e's inequality, the  inequality \eqref{3d-uni3} can be rewritten as
		\begin{align*}
		\frac{\d}{\d t}\|\y(t)\|^2_{\H} + \left(\mu\lambda_1-\frac{27c_3^4}{16\mu^3}\|\u_1(t)\|^4_{\V}\right)\|\y(t)\|^2_{\H}\leq0.
		\end{align*}
		\vskip 0.2 cm
		\noindent
		\textbf{Case I:}  $r>3$ with any $\beta,\mu>0$.
		From Lemma \ref{3d-ab} for $\varepsilon=\frac{8\varrho_1\mu^3}{27c_3^4}$, where $$\varrho_1=\mu\lambda_1-\frac{27c_3^4}{16\mu^5}\left[2+\frac{2\eta_3+1}{\mu\lambda_1}+\frac{2\eta_3+1}{\mu^2\lambda_1^2}\right]^2\|\f\|^4_{\H}>0,$$ one can obtain the existence of a time $T_1=T_{B_{\H},\frac{8\varrho_1\mu^3}{27c_3^4}}>0$ such that
		\begin{align}\label{3d-uni5}
		\mu\lambda_1-\frac{27c_3^4}{16\mu^3}\|\u_1(t)\|^4_{\V}\geq\mu\lambda_1-\frac{27c_3^4}{16\mu^3}\left(\frac{1}{\mu^2}\left[2+\frac{2\eta_3+1}{\mu\lambda_1}+\frac{2\eta_3+1}{\mu^2\lambda_1^2}\right]^2\|\f\|^4_{\H}+\frac{8\varrho_1\mu^3}{27c_3^4}\right)=\frac{\varrho_1}{2},
		\end{align}
		for all $t\geq T_1$.	Hence, we obtain 
		\begin{align*}
		\frac{\d}{\d t}\|\y(t)\|^2_{\H} +\frac{\varrho_1}{2}\|\y(t)\|^2_{\H}\leq 0,  \ \text{ for a.e. }\  t\geq T_1,
		\end{align*}
		and by applying Gronwall's inequality, we arrive at \eqref{3DUnique1}.
		\vskip 0.2 cm
		\noindent
		\textbf{Case II:}  $r=3$ with $2\beta\mu\geq1$.
		From Lemma \ref{3d-ab} for $\varepsilon=\frac{8\varrho_2\mu^3}{27c_3^4}$, where $$\varrho_2=\mu\lambda_1-\frac{27c_3^4}{16\mu^5}\left[1+\frac{1}{\mu\lambda_1}+\frac{1}{\mu^2\lambda_1^2}\right]^2\|\f\|^4_{\H}>0,$$ we infer the existence of  a time $T_2=T_{B_{\H},\frac{8\varrho_2\mu^3}{27c_3^4}}>0$ such that
		\begin{align}\label{3d-uni7}
		\mu\lambda_1-\frac{27c_3^4}{16\mu^3}\|\u_1(t)\|^4_{\V}\geq\mu\lambda_1-\frac{27c_3^4}{16\mu^3}\left(\frac{1}{\mu^2}\left[1+\frac{1}{\mu\lambda_1}+\frac{1}{\mu^2\lambda_1^2}\right]^2\|\f\|^4_{\H}+\frac{8\varrho_2\mu^3}{27c_3^4}\right)=\frac{\varrho_2}{2},
		\end{align}
		for all $t\geq T_2$.	Hence, we obtain 
		\begin{align*}
		\frac{\d}{\d t}\|\y(t)\|^2_{\H} +\frac{\varrho_2}{2}\|\y(t)\|^2_{\H}\leq 0, \  \text{ for a.e. } \ t\geq T_2,
		\end{align*}
		and by applying the Gronwall inequality, \eqref{3DUnique2} follows.
	\end{proof}
	The following theorem establishes the existence of a singleton attractor for the 3D CBF equations \eqref{D-CBF} and a proof of this theorem can be obtained in a similar way as in the proof of Theorem \ref{D-SA}, by using Theorem \ref{3DUniqueness}.
	\begin{theorem}\label{3D-SA}
		Let the conditions \eqref{3D-C_1} (for $r>3$ with any $\beta,\mu>0$) and \eqref{3D-C_3} (for $r=3$ with $2\beta\mu\geq1$) be satisfied. Then the global attractor $\mathcal{A}$ of the 3D CBF equations \eqref{D-CBF} is a singleton $\mathcal{A}=\{\mathbf{a}_{*}\}$ with $\mathbf{a}_{*}\in\V$.
	\end{theorem}

	\section{2D Stochastic CBF Equations with Additive Noise} \label{sec4}\setcounter{equation}{0}
	This section is devoted for discussing the existence of random attractors for 2D stochastic CBF equations perturbed by additive white noise and its convergence rate towards the deterministic singleton attractor as the noise intensity approaches zero.
	\subsection{Preliminaries}
	We consider the 2D stochastic convective Brinkman-Forchheimer equations perturbed by additive white noise as follows:
	\begin{equation}\label{SCBF-add}
	\left\{
	\begin{aligned}
	\d\u_{\varepsilon}(t)+\left[\mu \A\u_{\varepsilon}(t)+\B(\u_{\varepsilon}(t)) +\beta\mathcal{C}(\u_{\varepsilon}(t))\right]\d t&=\f \d t+\varepsilon \Phi(x)\d \W(t) , \quad t\geq 0, \\ 
	\u_{\varepsilon}(0)&=\x.
	\end{aligned}
	\right.
	\end{equation}
	The term $\varepsilon \Phi(x)$ denotes the perturbation intensity with $\varepsilon\in(0,1]$ and $\Phi(x)\in\D(\A)$, and $\W(t,\omega)$ is the standard scalar Wiener process on probability space $(\Omega, \mathscr{F}, \mathbb{P}),$ where $$\Omega=\{\omega\in C(\R;\R):\omega(0)=0\},$$ endowed with the compact-open topology given by the complete metric 
	\begin{align*}
	d_{\Omega}(\omega,\omega'):=\sum_{m=1}^{\infty} \frac{1}{2^m}\frac{\|\omega-\omega'\|_{m}}{1+\|\omega-\omega'\|_{m}},\text{ where } \|\omega-\omega'\|_{m}:=\sup_{-m\leq t\leq m} |\omega(t)-\omega'(t)|,
	\end{align*}
	and $\mathscr{F}$ the Borel sigma-algebra induced by the compact-open topology of $\Omega,$ $\mathbb{P}$ the two-sided Wiener measure on $(\Omega,\mathscr{F})$. From \cite{FS}, it is clear that  the measure $\mathbb{P}$ is ergodic and invariant under the translation-operator group $\{\vartheta_t\}_{t\in\R}$ on $\Omega$ defined by 
	\begin{align*}
	\vartheta_t \omega(\cdot) = \omega(\cdot+t)-\omega(t), \ \text{ for all }\ t\in\R, \ \omega\in \Omega.
	\end{align*}  The operator is known as Wiener shift operator. 
	
	Consider for some $\alpha>0$ (which will be specified later) 
	\begin{align}\label{OU1}
	\z(\vartheta_{t}\omega) =  \int_{-\infty}^{t} e^{-\alpha(t-\tau)}\d \W(\tau), \ \ \omega\in \Omega,
	\end{align} which is a stationary solution of the one dimensional Ornstein-Uhlenbeck equation
	\begin{align}\label{OU2}
	\d\z(\vartheta_t\omega) + \alpha\z(\vartheta_t\omega)\d t =\d\W(t).
	\end{align}
	It is known from \cite{FAN} that there exists a $\vartheta$-invariant subset $\widetilde{\Omega}\subset\Omega$ of full measure such that $\z(\vartheta_t\omega)$ is continuous in $t$ for every $\omega\in \widetilde{\Omega},$ and
	\begin{align}
	\mathbb{E}\left(e^{\delta\int_{s}^{s+t}|\z(\vartheta_{\xi}\omega)|\d \xi}\right)&\leq e^{\frac{\delta t}{\sqrt{\alpha}}}, \ \text{ for all } s\in \R, \ \alpha^3\geq\delta^2\geq 0, t\geq 0,\label{Z1}\\
	\mathbb{E}\left(|\z(\vartheta_s\omega)|^{\xi}\right)&=\frac{\Gamma\left(\frac{1+\xi}{2}\right)}{\sqrt{\pi\alpha^{\xi}}}, \ \text{ for all } \xi>0, s\in \R,\label{Z2}\\
	\lim_{t\to \pm \infty} \frac{|\z(\vartheta_t\omega)|}{|t|}&=0,\label{Z3}\\
	\lim_{t\to \pm \infty} \frac{1}{t} \int_{0}^{t} \z(\vartheta_{\xi}\omega)\d\xi &=0,\label{Z4}\\
	\lim_{t\to \infty} e^{-\delta t}|\z(\vartheta_{-t}\omega)| &=0, \ \text{ for all } \ \delta>0,\label{Z5}
	\end{align} where $\Gamma$ is the Gamma function. From now onward, we will not distinguish between $\widetilde{\Omega}$ and $\Omega$.
	\begin{definition}
		A random set $D(\omega)$ in Banach space $\mathbb{X}$ is called \emph{tempered} if for any $\delta>0$ 
		\begin{align*}
		\lim_{t\to \infty} e^{-\delta t}\|D(\vartheta_{-t}\omega)\|_{\X} = \lim_{t\to \infty} e^{-\delta t}\left(\sup_{x\in D(\vartheta_{-t}\omega)}\|x\|_{\X}\right)=0.
		\end{align*}
		Let us define a class $\mathfrak{D}$ of all tempered random sets in $\H$, that is, 
		\begin{align*}
		\mathfrak{D}= \{D(\omega): D(\omega) \emph{ is a tempered random set in }\H\}.
		\end{align*}
	\end{definition}
	Next, we provide a result which is proved in \cite{HCPEK} and useful in our further analysis.
	\begin{lemma}[Lemma 10.6, \cite{HCPEK}]\label{z_esti}
		For any $T, k>0$,
		\begin{align*}
		\lim_{t\to \infty} \frac{\int_{-t}^{T-t}|\z(\vartheta_s\omega)|^k\d s}{t-T} =0, \ \text{ for all }\ \omega\in \Omega.
		\end{align*}
	\end{lemma}
	\subsection{Uniform estimates of solutions}
	Let us define $$\v_{\varepsilon}(t, \omega, \v_{\varepsilon}(0))=\u_{\varepsilon}(t, \omega, \v_{\varepsilon}(0))-\varepsilon\z(\vartheta_t\omega)\Phi(x),$$ with $\v_{\varepsilon}(0)=\u_{\varepsilon}(0)-\varepsilon\z(\omega)\Phi(x),$ where $\u_{\varepsilon}(\cdot)$ is the solution of \eqref{SCBF-add}. Then $\v_{\varepsilon}(\cdot)$ satisfies the following:
	\begin{equation}\label{CCBF-add}
	\left\{
	\begin{aligned}
	\frac{\d\v_{\varepsilon}(t)}{\d t}+\mu \A\v_{\varepsilon}(t)&+\B\big(\v_{\varepsilon}(t)+\varepsilon\z(\vartheta_t\omega)\Phi\big) +\beta\mathcal{C}\big(\v_{\varepsilon}(t)+\varepsilon\z(\vartheta_t\omega)\Phi\big)\\&=\f + \varepsilon\alpha\z(\vartheta_t\omega)\Phi-\varepsilon\mu\z(\vartheta_t\omega)\A \Phi , \quad t\geq 0, \\ 
	\v_{\varepsilon}(0)&=\x-\varepsilon\z(\omega)\Phi(x).
	\end{aligned}
	\right.
	\end{equation}
	\begin{lemma}
		Let $n=2$ and $\f, \v_{\varepsilon}(0)\in\H.$ Then for any $t\geq T>0$, the solution $\v_{\varepsilon}(\cdot)$ of \eqref{CCBF-add} satisfies
		\begin{align}\label{SH_ab}
		&	\|\v_{\varepsilon}(T,\vartheta_{-t}\omega, \v_{\varepsilon}(0))\|^2_{\H} +\int_{0}^{T} e^{\int_{T}^{\tau}\left(\frac{\varrho}{8}-\varepsilon^2K_1|\z(\vartheta_{s-t}\omega)|^2\right)\d s} \|\v_{\varepsilon}(\tau,\vartheta_{-t}\omega,\v_{\varepsilon}(0))\|^2_{\V}\ \d \tau \nonumber\\& \leq  e^{-\frac{\varrho T}{8}+\varepsilon^2K_1\int_{-t}^{T-t}|\z(\vartheta_{s}\omega)|^2\d s}\|\v_{\varepsilon}(0)\|^2_{\H}\nonumber\\&\quad+e^{\frac{\varrho}{8}(t-T)}\int_{-t}^{T-t}C e^{\frac{\varrho\tau}{8}+\varepsilon^2K_1\int_{\tau}^{T-t}|\z(\vartheta_{s}\omega)|^2\d s}\ \big(\|\f\|_{\H}^2 +\varepsilon^{j} |\z(\vartheta_{\tau}\omega)|^{j}+1\big)\ \d\tau,
		\end{align}
		where 
		\begin{align}\label{K1}
		K_1:=\max\left\{\frac{16\lambda_1 c_1^2 \|\Phi\|^2_{\V}}{\varrho},\frac{8c_2^2\|\A\Phi\|^2_{\H}}{\mu\lambda_1^2}\right\} \ \text{ and } \ j=\max\{4,r+1\}.
		\end{align}
	\end{lemma}
	\begin{proof}
		Taking the inner product with $\v_{\varepsilon}(\cdot)$ to the first equation in \eqref{CCBF-add}, we have
		\begin{align}\label{H1}
		\frac{1}{2}\frac{\d}{\d t} \|\v_{\varepsilon}(t)\|^2_{\H}=&-\mu \|\v_{\varepsilon}(t)\|^2_{\V} - b\big(\v_{\varepsilon}(t)+\varepsilon\z(\vartheta_t\omega)\Phi, \v_{\varepsilon}(t)+\varepsilon\z(\vartheta_t\omega)\Phi, \v_{\varepsilon}(t)\big)\nonumber\\&-\beta\left\langle \mathcal{C}(\v_{\varepsilon}(t)+\varepsilon\z(\vartheta_t\omega)\Phi),\v_{\varepsilon}(t)\right\rangle+\big(\f,\v_{\varepsilon}(t)\big)+\big(\varepsilon\alpha\z(\vartheta_t\omega)\Phi, \v_{\varepsilon}(t)\big) \nonumber\\&-\big(\varepsilon\mu\z(\vartheta_t\omega)\A \Phi,\v_{\varepsilon}(t)\big),\nonumber\\
		=&-\mu \|\v_{\varepsilon}(t)\|^2_{\V} - b\big(\v_{\varepsilon}(t)+\varepsilon\z(\vartheta_t\omega)\Phi,\varepsilon\z(\vartheta_t\omega)\Phi, \v_{\varepsilon}(t)\big)\nonumber\\&-\beta\|\v_{\varepsilon}(t)+\varepsilon\z(\vartheta_t\omega)\Phi\|^{r+1}_{\wi \L^{r+1}}+\beta\left\langle \mathcal{C}(\v_{\varepsilon}(t)+\varepsilon\z(\vartheta_t\omega)\Phi),\varepsilon\z(\vartheta_t\omega)\Phi\right\rangle\nonumber\\&+\big(\f,\v_{\varepsilon}(t)\big)+\big(\varepsilon\alpha\z(\vartheta_t\omega)\Phi, \v_{\varepsilon}(t)\big) -\big(\varepsilon\mu\z(\vartheta_t\omega)\A\Phi,\v_{\varepsilon}(t)\big),
		\end{align}
		for a.e. $t\in[0,\tau]$,	where we have used the fact that $b\big(\v_{\varepsilon}+\varepsilon\z(\vartheta_t\omega)\Phi,\v_{\varepsilon}, \v_{\varepsilon}\big)=0$. Applying H\"older's and Young's inequalities, we obtain
		\begin{align}\label{H2}
		\big|\big(\varepsilon\alpha\z(\vartheta_t\omega)\Phi, \v_{\varepsilon}\big)\big|+\big|\big(\f,\v_{\varepsilon}\big)\big|+\big|\big(\varepsilon\mu\z(\vartheta_t\omega)\A\Phi,\v_{\varepsilon}\big)\big|\leq C\|\f\|_{\H}^2+ \varepsilon^2C |\z(\vartheta_t\omega)|^2 + \frac{\varrho}{64} \|\v_{\varepsilon}\|_{\H}^2.
		\end{align}
		By \eqref{poin}, \eqref{b1} and Young's inequality, we estimate 
		\begin{align}\label{H3}
		&\big|b\big(\v_{\varepsilon}+\varepsilon\z(\vartheta_t\omega)\Phi, \varepsilon\z(\vartheta_t\omega)\Phi, \v_{\varepsilon}\big)\big|\nonumber\\&\leq\big|b\big(\v_{\varepsilon}, \varepsilon\z(\vartheta_t\omega)\Phi, \v_{\varepsilon}\big)\big|+|\varepsilon\z(\vartheta_t\omega)|^2\big|b(\Phi,\Phi, \v_{\varepsilon})\big|\nonumber\\&\leq \varepsilon c_1|\z(\vartheta_{t}\omega)|\|\Phi\|_{\V}\|\v_{\varepsilon}\|_{\H}\|\v_{\varepsilon}\|_{\V}+\varepsilon^2c_2|\z(\vartheta_{t}\omega)|^2\|\Phi\|_{\H}^{1/2}\|\Phi\|_{\V}\|\A\Phi\|^{1/2}_{\H}\|\v_{\varepsilon}\|_{\H}\nonumber\\&\leq \frac{8\varepsilon^2\lambda_1c_1^2|\z(\vartheta_{t}\omega)|^2 \|\Phi\|_{\V}^2 }{\varrho}\|\v_{\varepsilon}\|^2_{\H}+\frac{\varrho}{32\lambda_1}\|\v_{\varepsilon}\|^2_{\V}+ \varepsilon^4C|\z(\vartheta_{t}\omega)|^4 +\frac{\varrho}{64}\|\v_{\varepsilon}\|^2_{\H}\nonumber\\&\leq \frac{\varepsilon^2K_1|\z(\vartheta_{t}\omega)|^2}{2}\|\v_{\varepsilon}\|^2_{\H}+\frac{\varrho}{32\lambda_1}\|\v_{\varepsilon}\|^2_{\V}+ \varepsilon^4C|\z(\vartheta_{t}\omega)|^4 +\frac{\varrho}{64}\|\v_{\varepsilon}\|^2_{\H},
		\end{align}
		where $K_1$ is given by \eqref{K1}. Making use of H\"older's and Young's inequalities, we obtain 
		\begin{align}\label{H4}
		\beta\left\langle \mathcal{C}(\v_{\varepsilon}+\varepsilon\z(\vartheta_t\omega)\Phi),\varepsilon\z(\vartheta_t\omega)\Phi\right\rangle&\leq \beta|\varepsilon\z(\vartheta_t\omega)|\|\v_{\varepsilon}+\varepsilon\z(\vartheta_t\omega)\Phi\|^{r}_{\wi \L^{r+1}}\|\Phi\|_{\wi \L^{r+1}}\nonumber\\&\leq\frac{\beta}{2}\|\v_{\varepsilon}+\varepsilon\z(\vartheta_t\omega)\Phi\|^{r+1}_{\wi \L^{r+1}}+\varepsilon^{r+1} C |\z(\vartheta_t\omega)|^{r+1} .
		\end{align}
		Combining \eqref{H2}-\eqref{H4} and using it in \eqref{H1}, we find
		\begin{align*}
		&\frac{1}{2}\frac{\d}{\d t} \|\v_{\varepsilon}(t)\|^2_{\H}+\mu \|\v_{\varepsilon}(t)\|^2_{\V}+\frac{\beta}{2}\|\v_{\varepsilon}(t)+\varepsilon\z(\vartheta_t\omega)\Phi\|^{r+1}_{\wi \L^{r+1}}\nonumber\\&\leq \frac{\varepsilon^2K_1|\z(\vartheta_{t}\omega)|^2}{2}\|\v_{\varepsilon}(t)\|^2_{\H}+\frac{\varrho}{32\lambda_1}\|\v_{\varepsilon}(t)\|^2_{\V}+\frac{\varrho}{32} \|\v_{\varepsilon}(t)\|_{\H}^2+C\|\f\|_{\H}^2 +  \varepsilon^2C |\z(\vartheta_t\omega)|^2\nonumber\\&\quad+ \varepsilon^4C|\z(\vartheta_{t}\omega)|^4+\varepsilon^{r+1} C |\z(\vartheta_t\omega)|^{r+1},
		\end{align*}
		for a.e. $t\in[0,\tau]$. Arranging constants carefully (in order to prove further results), we deduce that 
		\begin{align*}
		&	\frac{1}{2} \frac{\d}{\d t} \|\v_{\varepsilon}(t)\|^2_{\H} + \left(\mu-\frac{\varrho}{8\lambda_1}+\frac{\varrho}{8\lambda_1}-\frac{\varrho}{32\lambda_1}\right)\|\v_{\varepsilon}(t)\|^2_{\V} -\left(\frac{\varepsilon^2K_1|\z(\vartheta_{t}\omega)|^2}{2}+\frac{\varrho}{32}\right)\|\v_{\varepsilon}(t)\|^2_{\H} \nonumber\\&\leq C\|\f\|_{\H}^2 +  \varepsilon^2C |\z(\vartheta_t\omega)|^2+ \varepsilon^4C|\z(\vartheta_{t}\omega)|^4+\varepsilon^{r+1} C |\z(\vartheta_t\omega)|^{r+1},
		\end{align*}
		for a.e. $t\in[0,\tau]$. By \eqref{poin} and Young's inequality, we have
		\begin{align}\label{H5}
		&	 \frac{\d}{\d t} \|\v_{\varepsilon}(t)\|^2_{\H} + \left(2\mu-\frac{\varrho}{4\lambda_1}\right)\|\v_{\varepsilon}(t)\|^2_{\V} +\left(\frac{\varrho}{8}-\varepsilon^2K_1|\z(\vartheta_{t}\omega)|^2\right)\|\v_{\varepsilon}(t)\|^2_{\H}\nonumber\\&\leq C\|\f\|_{\H}^2 +  \varepsilon^{j}C |\z(\vartheta_t\omega)|^{j}+C,
		\end{align}
		where $j=\max\{4,r+1\}$, for a.e. $t\in[0,\tau]$. By means of variation of constants formula, we obtain 
		\begin{align*}
		&	\|\v_{\varepsilon}(t)\|^2_{\H} +\left(2\mu-\frac{\varrho}{4\lambda_1}\right)\int_{0}^{t} e^{\int_{t}^{\tau}\left(\frac{\varrho}{8}-\varepsilon^2K_1|\z(\vartheta_{s}\omega)|^2\right)\d s} \|\v_{\varepsilon}(\tau)\|^2_{\V}\ \d \tau \nonumber\\& \leq e^{-\int_{0}^{t}\left(\frac{\varrho}{8}-\varepsilon^2K_1|\z(\vartheta_{s}\omega)|^2\right)\d s}\|\v_{\varepsilon}(0)\|^2_{\H}+\int_{0}^{t}C e^{\int_{t}^{\tau}\left(\frac{\varrho}{8}-\varepsilon^2K_1|\z(\vartheta_{s}\omega)|^2\right)\d s}\ \big(\|\f\|_{\H}^2 + \varepsilon^{j} |\z(\vartheta_t\omega)|^{j}+1\big)\ \d\tau.
		\end{align*}
		Therefore, for any $t\geq T>0$, we get
		\begin{align*}
		&	\|\v_{\varepsilon}(T,\vartheta_{-t}\omega, \v_{\varepsilon}(0))\|^2_{\H} +\left(2\mu-\frac{\varrho}{4\lambda_1}\right)\int_{0}^{T} e^{\int_{T}^{\tau}\left(\frac{\varrho}{8}-\varepsilon^2K_1|\z(\vartheta_{s-t}\omega)|^2\right)\d s} \|\v_{\varepsilon}(\tau,\vartheta_{-t}\omega,\v_{\varepsilon}(0))\|^2_{\V}\ \d \tau \nonumber\\& \leq e^{-\int_{0}^{T}\left(\frac{\varrho}{8}-\varepsilon^2K_1|\z(\vartheta_{s-t}\omega)|^2\right)\d s}\|\v_{\varepsilon}(0)\|^2_{\H}+\int_{0}^{T}C e^{\int_{T}^{\tau}\left(\frac{\varrho}{8}-\varepsilon^2K_1|\z(\vartheta_{s-t}\omega)|^2\right)\d s}\ \big(\|\f\|_{\H}^2 \nonumber\\&\quad+ \varepsilon^{j} |\z(\vartheta_t\omega)|^{j}+1\big)\ \d\tau\nonumber\\& = e^{-\frac{\varrho T}{8}+\varepsilon^2K_1\int_{-t}^{T-t}|\z(\vartheta_{s}\omega)|^2\d s}\|\v_{\varepsilon}(0)\|^2_{\H}+e^{\frac{\varrho}{8}(t-T)}\int_{-t}^{T-t}C e^{\frac{\varrho\tau}{8}+\varepsilon^2K_1\int_{\tau}^{T-t}|\z(\vartheta_{s}\omega)|^2\d s}\ \big(\|\f\|_{\H}^2 \nonumber\\&\quad+  \varepsilon^{j} |\z(\vartheta_{\tau}\omega)|^{j}+1\big)\ \d\tau.
		\end{align*}
		Since $2\mu-\frac{\varrho}{4\lambda_1}>0$, the lemma follows.
	\end{proof}
	\begin{lemma}
		For $n=2$, let $\f, \v_{\varepsilon}(0)\in\H.$ Then for any $t\geq \upeta>0$, the solution $\v_{\varepsilon}$ of \eqref{CCBF-add} satisfies
		\begin{align}\label{SV_ab}
		&\|\v_{\varepsilon}(\upeta,\vartheta_{-t}\omega, \v_{\varepsilon}(0))\|^2_{\V} \nonumber\\& \leq e^{-\frac{\varrho \upeta}{8}+\varepsilon^2K_1\int_{-t}^{\upeta-t}|\z(\vartheta_{s}\omega)|^2\d s}\|\v_{\varepsilon}(0)\|^2_{\H}\nonumber\\&\quad+e^{\frac{\varrho}{8}(t-\upeta)}\int_{-t}^{\upeta-t}C e^{\frac{\varrho\xi}{8}+\varepsilon^2K_1\int_{\xi}^{\upeta-t}|\z(\vartheta_{s}\omega)|^2\d s}\ \big(\|\f\|_{\H}^2+ \varepsilon^{j}|\z(\vartheta_{\xi }\omega)|^{j}+ 1\big)\d \xi,
		\end{align}
		where $K_1$ is the constant given by \eqref{K1} and $j=\max\{4,r+1\}$.
	\end{lemma}
	\begin{proof}
		Taking the inner product with $\A\v_{\varepsilon}(\cdot)$ to the first equation in \eqref{CCBF-add}, we have
		\begin{align}\label{V1}
		\frac{1}{2}\frac{\d}{\d t}\|\v_{\varepsilon}(t)\|^2_{\V} = & -\mu\|\A\v_{\varepsilon}(t)\|^2_{\V} -b(\v_{\varepsilon}(t)+\varepsilon\z(\vartheta_t\omega)\Phi,\v_{\varepsilon}(t)+\varepsilon\z(\vartheta_t\omega)\Phi,\A\v_{\varepsilon}(t)) \nonumber\\&-\beta\big\langle\mathcal{C}(\v_{\varepsilon}(t)+\varepsilon\z(\vartheta_t\omega)\Phi),\A\v_{\varepsilon}(t)\big\rangle+\big(\alpha\varepsilon\z(\vartheta_t\omega)\Phi, \A\v_{\varepsilon}(t)\big) +\big(\f,\A\v_{\varepsilon}(t)\big)\nonumber\\&-\big(\varepsilon\mu\z(\vartheta_t\omega)\A\Phi,\A\v_{\varepsilon}(t)\big)\nonumber\\
		=&-\mu\|\A\v_{\varepsilon}(t)\|^2_{\V} +b(\v_{\varepsilon}(t)+\varepsilon\z(\vartheta_t\omega)\Phi,\v_{\varepsilon}(t)+\varepsilon\z(\vartheta_t\omega)\Phi,\varepsilon\z(\vartheta_t\omega)\A\Phi)\nonumber\\&-b(\u_{\varepsilon}(t),\u_{\varepsilon}(t),\A\u_{\varepsilon}(t))-\beta\big\langle\mathcal{C}\big(\u_{\varepsilon}(t)\big),\A\u_{\varepsilon}(t)\big\rangle+\beta\big\langle\mathcal{C}\big(\u_{\varepsilon}(t)\big),\varepsilon\z(\vartheta_t\omega)\A\Phi\big\rangle \nonumber\\&+\big(\alpha\varepsilon\z(\vartheta_t\omega)\Phi, \A\v_{\varepsilon}(t)\big) +\big(\f,\A\v_{\varepsilon}(t)\big)-\big(\varepsilon\mu\z(\vartheta_t\omega)\A\Phi,\A\v_{\varepsilon}(t)\big),
		\end{align}
		for a.e. $t\in[0,\tau]$. Applying H\"older's and Young's inequalities, we get 
		\begin{align}\label{V2}
		&	\big|\big(\varepsilon\alpha\z(\vartheta_t\omega)\Phi, \A\v_{\varepsilon}\big)\big|+\big|\big(\f,\A\v_{\varepsilon}\big)\big|+\big|\big(\varepsilon\mu\z(\vartheta_t\omega)\A\Phi,\A\v_{\varepsilon}\big)\big|\nonumber\\&\leq C\|\f\|_{\H}^2+ \varepsilon^2C |\z(\vartheta_t\omega)|^2 + \frac{\mu}{8} \|\A\v_{\varepsilon}\|_{\H}^2.
		\end{align}
		In 2D periodic domains, we infer that 
		\begin{align}\label{V3}
		b(\u_{\varepsilon},\u_{\varepsilon},\A\u_{\varepsilon})=0,
		\end{align}
		and
		\begin{align}\label{V4}
		\big\langle\mathcal{C}\big(\u_{\varepsilon}\big),\A\u_{\varepsilon}\big\rangle&=\int_{\mathbb{T}^2}|\nabla\u_{\varepsilon}(x)|^2|\u_{\varepsilon}(x)|^{r-1}\d x+4\left[\frac{r-1}{(r+1)^2}\right]\int_{\mathbb{T}^2}|\nabla|\u_{\varepsilon}(x)|^{\frac{r+1}{2}}|^2\d x.
		\end{align}
		Making use of \eqref{poin}, \eqref{poin_1} and \eqref{b1}, we estimate  
		\begin{align}\label{V5}
		&	|b(\v_{\varepsilon}+\varepsilon\z(\vartheta_t\omega)\Phi,\v_{\varepsilon}+\varepsilon\z(\vartheta_t\omega)\Phi,\varepsilon\z(\vartheta_t\omega)\A\Phi)|\nonumber\\&\leq|b(\v_{\varepsilon},\v_{\varepsilon},\varepsilon\z(\vartheta_t\omega)\A\Phi)|+|b(\varepsilon\z(\vartheta_t\omega)\Phi,\v_{\varepsilon},\varepsilon\z(\vartheta_t\omega)\A\Phi)|+|b(\v_{\varepsilon},\varepsilon\z(\vartheta_t\omega)\Phi,\varepsilon\z(\vartheta_t\omega)\A\Phi)|\nonumber\\&\quad+|b(\varepsilon\z(\vartheta_t\omega)\Phi,\varepsilon\z(\vartheta_t\omega)\Phi,\varepsilon\z(\vartheta_t\omega)\A\Phi)|\nonumber\\&\leq \frac{c_2}{\lambda_1}\|\v_{\varepsilon}\|_{\V}\|\A\v_{\varepsilon}\|_{\H}\|\varepsilon\z(\vartheta_{t}\omega)\A\Phi\|_{\H}+\varepsilon^2C|\z(\vartheta_{t}\omega)|^2\|\A\v_{\varepsilon}\|_{\H}+\varepsilon^2C|\z(\vartheta_{t}\omega)|^2\|\v_{\varepsilon}\|_{\V}\nonumber\\&\quad+\varepsilon^3C|\z(\vartheta_{t}\omega)|^3\nonumber\\&\leq\frac{\mu}{8}\|\A\v_{\varepsilon}\|^2_{\H}+\frac{4\varepsilon^2 c_2^2\|\A\Phi\|^2_{\H}}{\mu\lambda_1^2}|\z(\vartheta_{t}\omega)|^2\|\v_{\varepsilon}\|^2_{\V} +\frac{K_2}{2}\|\v_{\varepsilon}\|^2_{\V}+\varepsilon^4C|\z(\vartheta_{t}\omega)|^4+\varepsilon^3C|\z(\vartheta_{t}\omega)|^3+C\nonumber\\&\leq\frac{\mu}{8}\|\A\v_{\varepsilon}\|^2_{\H}+\frac{\varepsilon^2 K_1}{2}|\z(\vartheta_{t}\omega)|^2\|\v_{\varepsilon}\|^2_{\V} +\frac{K_2}{2}\|\v_{\varepsilon}\|^2_{\V}+\varepsilon^4C|\z(\vartheta_{t}\omega)|^4+\varepsilon^3C|\z(\vartheta_{t}\omega)|^3+C,
		\end{align}
		where $K_1$ is given by \eqref{K1} and $K_2$ is taken as
		\begin{align}\label{K2}
		K_2:= \mu\lambda_1-\frac{\varrho}{8} \qquad \text{ or }\qquad \frac{\mu\lambda_1}{2}-\frac{K_2}{2}= \frac{\varrho}{16}.
		\end{align}
		Applying H\"older's inequality, Sobolev embedding and Young's inequality, we have
		\begin{align}
		\big\langle\mathcal{C}\big(\u_{\varepsilon}\big),\varepsilon\z(\vartheta_t\omega)\A\Phi\big\rangle&\leq\|\u_{\varepsilon}\|^r_{\L^{2r}}|\varepsilon\z(\vartheta_t\omega)|\|\A\Phi\|_{\H}\nonumber\\&\leq C\|\u_{\varepsilon}\|^r_{\L^{2(r+1)}}|\varepsilon\z(\vartheta_t\omega)|\|\A\Phi\|_{\H}\nonumber\\
		&=C\bigg(\int_{\mathbb{T}^2}|\u_{\varepsilon}(x)|^{2(r+1)}\d x\bigg)^{\frac{r}{2(r+1)}}|\varepsilon\z(\vartheta_t\omega)|\|\A\Phi\|_{\H}\nonumber\\
		&=C\bigg(\int_{\mathbb{T}^2}\big(|\u_{\varepsilon}(x)|^{\frac{r+1}{2}}\big)^4\d x\bigg)^{\frac{r}{2(r+1)}}|\varepsilon\z(\vartheta_t\omega)|\|\A\Phi\|_{\H}\nonumber\\
		&=C\||\u_{\varepsilon}|^{\frac{r+1}{2}}\|^{\frac{2r}{r+1}}_{\wi \L^4}|\varepsilon\z(\vartheta_t\omega)|\|\A\Phi\|_{\H}\nonumber\\&\leq C\||\u_{\varepsilon}|^{\frac{r+1}{2}}\|^{\frac{2r}{r+1}}_{\V}|\varepsilon\z(\vartheta_t\omega)|\|\A\Phi\|_{\H}\nonumber\\
		&\leq C\bigg(\int_{\mathbb{T}^2}|\nabla|\u_{\varepsilon}(x)|^{\frac{r+1}{2}}|^2\d x\bigg)^{\frac{r}{(r+1)}}|\varepsilon\z(\vartheta_t\omega)|\|\A\Phi\|_{\H}\nonumber\\
		&\leq 2\left[\frac{r-1}{(r+1)^2}\right]\int_{\mathcal{O}}|\nabla|\u_{\varepsilon}(x)|^{\frac{r+1}{2}}|^2\d x + \varepsilon^{r+1}C|\z(\vartheta_t\omega)|^{r+1},\  \text{ for } \ r>1,\label{V6}\\	
		\big\langle\mathcal{C}\big(\u_{\varepsilon}\big),\varepsilon\z(\vartheta_t\omega)\A\Phi\big\rangle&\leq \frac{1}{2}\int_{\mathbb{T}^2} |\nabla\u_{\varepsilon}(x)|^2 \d x + \varepsilon^{2}C|\z(\vartheta_t\omega)|^{2}, \ \text{ for } \ r=1.\label{V7}
		\end{align}
		Combining \eqref{V2}-\eqref{V7} and using in \eqref{V1}, we deduce that 
		\begin{align*}
		&\frac{1}{2}\frac{\d}{\d t}\|\v_{\varepsilon}(t)\|^2_{\V}+\left(\frac{\mu}{2}+\frac{\mu}{4}\right)\|\A\v_{\varepsilon}(t)\|^2_{\H}+\frac{\beta}{2}\int_{\mathbb{T}^2}|\nabla\u_{\varepsilon}(t,x)|^2|\u_{\varepsilon}(t,x)|^{r-1}\d x\nonumber\\&\quad+2\beta\left[\frac{r-1}{(r+1)^2}\right]\int_{\mathbb{T}^2}|\nabla|\u_{\varepsilon}(t,x)|^{\frac{r+1}{2}}|^2\d x\nonumber\\ &\leq \frac{\varepsilon^2 K_1}{2}|\z(\vartheta_{t}\omega)|^2\|\v_{\varepsilon}(t)\|^2_{\V} +\frac{K_2}{2}\|\v_{\varepsilon}(t)\|^2_{\V}+C\|\f\|_{\H}^2+ \varepsilon^2C |\z(\vartheta_t\omega)|^2+\varepsilon^3C|\z(\vartheta_{t}\omega)|^3\nonumber\\&\qquad+\varepsilon^4C|\z(\vartheta_{t}\omega)|^4+ \varepsilon^{r+1}C|\z(\vartheta_t\omega)|^{r+1}+C,
		\end{align*}
		for a.e. $t\in[0,\tau]$. Using \eqref{poin_1} in above inequality, we have
		\begin{align*}
		&\frac{1}{2}\frac{\d}{\d t}\|\v_{\varepsilon}(t)\|^2_{\V}+\frac{\mu}{4}\|\A\v_{\varepsilon}(t)\|^2_{\H}+\left(\frac{\mu\lambda_1}{2}-\frac{K_2}{2}-\frac{\varepsilon^2 K_1}{2}|\z(\vartheta_{t}\omega)|^2\right)\|\v_{\varepsilon}(t)\|^2_{\V}\nonumber\\ &\leq  C\|\f\|_{\H}^2+  \varepsilon^{j}C|\z(\vartheta_t\omega)|^{j}+C,	
		\end{align*}
		From the definition of $K_2$ in \eqref{K2}, it is clear that 
		\begin{align}\label{V8}
		&\frac{\d}{\d t}\|\v_{\varepsilon}(t)\|^2_{\V}+\frac{\mu}{2}\|\A\v_{\varepsilon}(t)\|^2_{\H}+\left(\frac{\varrho}{8}-\varepsilon^2 K_1|\z(\vartheta_{t}\omega)|^2\right)\|\v_{\varepsilon}(t)\|^2_{\V}\nonumber\\ &\leq  C\|\f\|_{\H}^2+ \varepsilon^{j}C|\z(\vartheta_t\omega)|^{j}+C,
		\end{align}
		where $j=\max\{4,r+1\}$, for a.e. $t\in[0,\tau]$. Variation of constants formula in  \eqref{V8} gives 
		\begin{align*}
		&	\|\v_{\varepsilon}(t)\|^2_{\V} +\frac{\mu}{2}\int_{0}^{t} e^{\int_{t}^{\zeta}\left(\frac{\varrho}{8} - \varepsilon^2 K_1 |\z(\vartheta_{s}\omega)|^2\right)\d s}\|\A\v_{\varepsilon}(\zeta)\|^2_{\H}\d \zeta\nonumber\\& \leq e^{\int_{t}^{\uprho}\left(\frac{\varrho}{8} - \varepsilon^2 K_1 |\z(\vartheta_{s}\omega)|^2\right)\d s} \|\v_{\varepsilon}(\uprho)\|^2_{\V}+ \int_{0}^{t} Ce^{\int_{t}^{\zeta}\left(\frac{\varrho}{8} - \varepsilon^2 K_1 |\z(\vartheta_{s}\omega)|^2\right)\d s}\big(\|\f\|_{\H}^2+ \varepsilon^{j}|\z(\vartheta_{\zeta}\omega)|^{j}+ 1\big)\d \zeta,
		\end{align*}
		for any $t\geq\uprho\geq0$.	Hence, for any $t\geq\upeta>\uprho>0,$ we get 
		\begin{align*}
		\|\v_{\varepsilon}(\upeta,\vartheta_{-t}\omega, \v_{\varepsilon}(0))\|^2_{\V} & \leq e^{\int_{\upeta}^{\uprho}\left(\frac{\varrho}{8} - \varepsilon^2 K_1 |\z(\vartheta_{s-t}\omega)|^2\right)\d s} \|\v_{\varepsilon}(\uprho,\vartheta_{-t}\omega, \v_{\varepsilon}(0))\|^2_{\V}\nonumber\\&\quad+ \int_{0}^{\upeta} Ce^{\int_{\upeta}^{\xi}\left(\frac{\varrho}{8} - \varepsilon^2 K_1 |\z(\vartheta_{s-t}\omega)|^2\right)\d s}\big(\|\f\|_{\H}^2+ \varepsilon^{j}|\z(\vartheta_{\xi-t}\omega)|^{j}+ 1\big)\d \xi.
		\end{align*}
		For $t\geq\upeta>\uprho>\upeta-1\geq0$, integrating the  above inequality with respect to $\uprho$ over $(\upeta-1, \upeta)$, we get
		\begin{align}\label{V11}
		\|\v_{\varepsilon}(\upeta,\vartheta_{-t}\omega, \v_{\varepsilon}(0))\|^2_{\V} & \leq \int_{\upeta-1}^{\upeta}e^{\int_{\upeta}^{\uprho}\left(\frac{\varrho}{8} - \varepsilon^2 K_1 |\z(\vartheta_{s-t}\omega)|^2\right)\d s} \|\v_{\varepsilon}(\uprho,\vartheta_{-t}\omega, \v_{\varepsilon}(0))\|^2_{\V}\nonumber\\&\quad+ \int_{0}^{\upeta} Ce^{\int_{\upeta}^{\xi}\left(\frac{\varrho}{8} - \varepsilon^2 K_1 |\z(\vartheta_{s-t}\omega)|^2\right)\d s}\big(\|\f\|_{\H}^2+ \varepsilon^{j}|\z(\vartheta_{\xi-t}\omega)|^{j}+ 1\big)\d \xi.
		\end{align}
		Making use of \eqref{SH_ab} in \eqref{V11}, we obtain \eqref{SV_ab}, which completes the proof.
	\end{proof}
	\subsection{Perturbation radius of the singleton attractor under additive noise}
	Let us now consider the difference $\w_{\varepsilon}(\cdot)=\v_{\varepsilon}(\cdot)-\u(\cdot)$, where $\v_{\varepsilon}(\cdot)$ and $\u(\cdot)$ are the solutions of the systems \eqref{CCBF-add} and \eqref{D-CBF}, respectively. It is obvious that $\w_{\varepsilon}(\cdot)$ satisfies 
	\begin{equation}\label{Diff-CBF_add}
	\left\{
	\begin{aligned}
	\frac{\d\w_{\varepsilon}}{\d t}&=-\mu \A\w_{\varepsilon}-\B(\u_{\varepsilon})+\B(\u) -\beta\mathcal{C}(\u_{\varepsilon})+\beta\mathcal{C}(\u)+ \varepsilon\alpha\z(\vartheta_t\omega)\Phi-\varepsilon\mu\z(\vartheta_t\omega)\A \Phi , \\ 
	\w_{\varepsilon}(0)&=-\varepsilon\z(\omega)\Phi(x),
	\end{aligned}
	\right.
	\end{equation}
	in $\V'$. Due to technical difficulties, we restrict ourselves to $1\leq r\leq 2$. 
	\begin{lemma}\label{PertRad-add}
		For $n=2$ and any $1\leq r\leq 2$, let $\f\in \H$. Assume that $\v_{\varepsilon}(\cdot)$ and $\u(\cdot)$ are the solutions of \eqref{CCBF-add} and \eqref{D-CBF} corresponding to the initial data $\v_{\varepsilon}(0)\in D\in \mathfrak{D}$ and $\x\in\H$, respectively. Then there exists a random variable $\upgamma(\omega)$ depending only on $D$ and $\|\x\|_{\H}$ such that the $\w_{\varepsilon}(\cdot)$ satisfies
		\begin{align}\label{Diff_rad}
		\limsup_{t\to \infty} \|\w_{\varepsilon}(t,\vartheta_{-t}\omega,\w_{\varepsilon}(0))\|^2_{\H}\ \leq\ \varepsilon^{\frac{r+1}{r}} \upgamma(\omega), \  \text{ for all }  \ \varepsilon\in(0,1].
		\end{align}
	\end{lemma}
	\begin{proof}
		Taking the inner product with $\w_{\varepsilon}(\cdot)$ to the first equation in \eqref{Diff-CBF_add}, we find 
		\begin{align}\label{DA1}
		\frac{1}{2}\frac{\d}{\d t} \|\w_{\varepsilon}(t)\|^2_{\H} &=- \mu\|\w_{\varepsilon}(t)\|^2_{\V}-b(\u_{\varepsilon}(t),\u_{\varepsilon}(t),\w_{\varepsilon}(t))+b(\u(t),\u(t),\w_{\varepsilon}(t))\nonumber\\&\quad-\beta\left\langle\mathcal{C}(\u_{\varepsilon}(t))-\mathcal{C}(\u(t)),\w_{\varepsilon}(t)\right\rangle+\left\langle\varepsilon\alpha\z(\vartheta_t\omega)\Phi-\varepsilon\mu\z(\vartheta_t\omega)\A \Phi,\w_{\varepsilon}(t)\right\rangle\nonumber\\&=- \mu\|\w_{\varepsilon}(t)\|^2_{\V}-b(\u_{\varepsilon}(t),\u_{\varepsilon}(t),\w_{\varepsilon}(t))+b(\u(t),\u(t),\w_{\varepsilon}(t))\nonumber\\&\quad-\beta\left\langle\mathcal{C}(\u_{\varepsilon}(t))-\mathcal{C}(\u(t)),\u_{\varepsilon}(t)-\u(t)\right\rangle+\beta\left\langle\mathcal{C}(\u_{\varepsilon}(t))-\mathcal{C}(\u(t)),\varepsilon\z(\vartheta_{t}\omega)\Phi\right\rangle\nonumber\\&\quad+\left\langle\varepsilon\alpha\z(\vartheta_t\omega)\Phi-\varepsilon\mu\z(\vartheta_t\omega)\A \Phi,\w_{\varepsilon}(t)\right\rangle,
		\end{align}
		for a.e. $t\in[0,\tau]$. Applying H\"older's and Young's inequalities, and \eqref{poin}, we have
		\begin{align}\label{DA2}
		|\left\langle\varepsilon\alpha\z(\vartheta_t\omega)\Phi-\varepsilon\mu\z(\vartheta_t\omega)\A \Phi,\w_{\varepsilon}\right\rangle|&\leq \varepsilon^2C|\z(\vartheta_{t}\omega)|^2+\frac{\varrho}{24}\|\w_{\varepsilon}\|^2_{\H}\nonumber\\&\leq \varepsilon^2C|\z(\vartheta_{t}\omega)|^2+\frac{\varrho}{24\lambda_1}\|\w_{\varepsilon}\|^2_{\V}.
		\end{align}
		Making use of \eqref{b1} and then estimating the terms carefully by Young's inequality, we obtain 
		\begin{align}\label{DA3}
		&	|b(\u_{\varepsilon},\u_{\varepsilon},\w_{\varepsilon})-b(\u,\u,\w_{\varepsilon})|\nonumber\\&=|b(\u_{\varepsilon},\varepsilon\z(\vartheta_{t}\omega)\Phi,\w_{\varepsilon})+b(\w_{\varepsilon},\u,\w_{\varepsilon})+b(\varepsilon\z(\vartheta_{t}\omega)\Phi,\u,\w_{\varepsilon})|\nonumber\\&\leq \varepsilon c_1|\z(\vartheta_{t}\omega)|\|\u_{\varepsilon}\|^{1/2}_{\H}\|\u_{\varepsilon}\|^{1/2}_{\V}\|\Phi\|_{\V}\|\w_{\varepsilon}\|^{1/2}_{\H}\|\w_{\varepsilon}\|^{1/2}_{\V}+c_1\|\w_{\varepsilon}\|_{\H}\|\w_{\varepsilon}\|_{\V}\|\u\|_{\V}\nonumber\\&\quad+\varepsilon c_1|\z(\vartheta_{t}\omega)|\|\Phi\|^{1/2}_{\H}\|\Phi\|^{1/2}_{\V}\|\u\|_{\V}\|\w_{\varepsilon}\|^{1/2}_{\H}\|\w_{\varepsilon}\|^{1/2}_{\V}\nonumber\\& \leq \left(\varepsilon^2C|\z(\vartheta_{t}\omega)|^2\|\u_{\varepsilon}\|^2_{\V} + \frac{\varrho}{24\lambda_1}\|\w_{\varepsilon}\|^2_{\V}\right)+ \left(\frac{c_1^2}{2\mu}\|\w_{\varepsilon}\|^2_{\H}\|\u\|^2_{\V}+\frac{\mu}{2}\|\w_{\varepsilon}\|^2_{\V}\right)\nonumber\\&\quad+\left(\varepsilon^2C|\z(\vartheta_{t}\omega)|^2\|\u\|^2_{\V} + \frac{\varrho}{24\lambda_1}\|\w_{\varepsilon}\|^2_{\V}\right)\nonumber\\&=\varepsilon^2C|\z(\vartheta_{t}\omega)|^2\left(\|\u_{\varepsilon}\|^2_{\V}+\|\u\|^2_{\V}\right) +\frac{c_1^2}{2\mu}\|\w_{\varepsilon}\|^2_{\H}\|\u\|^2_{\V}+\left(\frac{\mu}{2}+\frac{\varrho}{12\lambda_1}\right)\|\w_{\varepsilon}\|^2_{\V}.
		\end{align}
		By \eqref{MO_c} and \eqref{a215}, one can easily deduce that 
		\begin{align}\label{DA4}
		\frac{\beta}{2^{r-1}}\|\u_{\varepsilon}-\u\|_{\wi\L^{r+1}}^{r+1}\leq 	\beta\left\langle\mathcal{C}(\u_{\varepsilon})-\mathcal{C}(\u),\u_{\varepsilon}-\u\right\rangle.
		\end{align} 
		For $r\in[1,2]$, using H\"older's and Young's inequalities, and Sobolev embedding, we have
		\begin{align}\label{DA5}
		&\beta|\left\langle\mathcal{C}(\u_{\varepsilon})-\mathcal{C}(\u),\varepsilon\z(\vartheta_{t}\omega)\Phi\right\rangle|\nonumber\\&\leq \beta r |\varepsilon\z(\vartheta_{t}\omega)|\left(\|\u_{\varepsilon}\|_{\wi \L^{r+1}}+\|\u\|_{\wi \L^{r+1}}\right)^{r-1}\|\u_{\varepsilon}-\u\|_{\wi\L^{r+1}}\|\Phi\|_{\wi \L^{r+1}}\nonumber\\
		&\leq \frac{\beta}{2^r}\|\u_{\varepsilon}-\u\|_{\wi\L^{r+1}}^{r+1}+\varepsilon^{\frac{r+1}{r} }C |\z(\vartheta_{t}\omega)|^{\frac{r+1}{r} }\left\{\|\u_{\varepsilon}\|^{\frac{(r+1)(r-1)}{r} }_{\V}+\|\u\|^{\frac{(r+1)(r-1)}{r} }_{\V}\right\}\nonumber\\
		&\leq \frac{\beta}{2^r}\|\u_{\varepsilon}-\u\|_{\wi\L^{r+1}}^{r+1}+\varepsilon^{\frac{r+1}{r} }C |\z(\vartheta_{t}\omega)|^{\frac{r+1}{r} }\left\{\|\u_{\varepsilon}\|^{2 }_{\V}+\|\u\|^{2 }_{\V}+1\right\}.
		\end{align}
		Combining \eqref{DA2}-\eqref{DA5} and using in \eqref{DA1}, we deduce that
		\begin{align}\label{DA6}
		&\frac{1}{2} \frac{\d}{\d t} \|\w_{\varepsilon}(t)\|^2_{\H} + \left(\frac{\mu}{2}-\frac{\varrho}{8\lambda_1}\right)\|\w_{\varepsilon}(t)\|^2_{\V}-\frac{c_1^2}{2\mu}\|\w_{\varepsilon}(t)\|^2_{\H}\|\u(t)\|^2_{\V}+\frac{\beta}{2^r}\|\u_{\varepsilon}(t)-\u(t)\|_{\wi\L^{r+1}}^{r+1}	\nonumber\\&\leq\left(\varepsilon^2C|\z(\vartheta_{t}\omega)|^2+\varepsilon^{\frac{r+1}{r} }C |\z(\vartheta_{t}\omega)|^{\frac{r+1}{r} }\right)\left(\|\u_{\varepsilon}(t)\|^2_{\V}+\|\u(t)\|^2_{\V}+1\right),
		\end{align}
		for a.e. $t\in[0,\tau]$. Since \eqref{C_2} holds, using \eqref{poin} and \eqref{DA6}, we get
		\begin{align}\label{DA7}
		&\frac{\d}{\d t} \|\w_{\varepsilon}(t)\|^2_{\H} + \left(\mu\lambda_1-\frac{\varrho}{4}-\frac{c_1^2}{\mu}\|\u(t)\|^2_{\V}\right)\|\w_{\varepsilon}(t)\|^2_{\H}\nonumber\\&\leq \frac{\d}{\d t} \|\w_{\varepsilon}(t)\|^2_{\H} + \left(\mu-\frac{\varrho}{4\lambda_1}\right)\|\w_{\varepsilon}(t)\|^2_{\V}-\frac{c_1^2}{\mu}\|\w_{\varepsilon}(t)\|^2_{\H}\|\u(t)\|^2_{\V}\nonumber\\&\leq\left(\varepsilon^2C|\z(\vartheta_{t}\omega)|^2++\varepsilon^{\frac{r+1}{r} }C |\z(\vartheta_{t}\omega)|^{\frac{r+1}{r} }\right)\left(\|\u_{\varepsilon}(t)\|^2_{\V}+\|\u(t)\|^2_{\V}+1\right),
		\end{align} 
		for a.e. $t\in[0,\tau]$. Now, consider big lapsing time $T=T_{B_{\H},\frac{\varrho\mu}{2c_1^2}}>0$ such that \eqref{uni5} holds, that is,
		\begin{align*}
		\mu\lambda_1-\frac{c_1^2}{\mu}\|\u(t)\|^2_{\V}\geq\frac{\varrho}{2}, \ \text{ for all }\ t\geq T.
		\end{align*}
		It follows from \eqref{DA7} that 
		\begin{align}\label{DA8}
		\frac{\d}{\d t} \|\w_{\varepsilon}(t)\|^2_{\H} + \frac{\varrho}{4}\|\w_{\varepsilon}(t)\|^2_{\H}&\leq\left(\varepsilon^2C|\z(\vartheta_{t}\omega)|^2+\varepsilon^{\frac{r+1}{r} }C |\z(\vartheta_{t}\omega)|^{\frac{r+1}{r} }\right)\left(\|\u_{\varepsilon}(t)\|^2_{\V}+1\right),
		\end{align}
		for a.e. $t\in[T,\tau]$. Variation of constants formula leads to  
		\begin{align*}
		\|\w_{\varepsilon}(t,\omega,\w_{\varepsilon}(0))\|^2_{\H}&\leq e^{-\frac{\varrho}{4}(t-T)}\|\w_{\varepsilon}(T,\omega,\w_{\varepsilon}(0))\|^2_{\H}\nonumber\\&\quad+\varepsilon^2C\int_{T}^{t}e^{\frac{\varrho}{4}(\upeta-t)}|\z(\vartheta_{\upeta}\omega)|^2\bigg(\|\u_{\varepsilon}(\upeta,\omega,\u_{\varepsilon}(0))\|^2_{\V}+1\bigg)\d\upeta\nonumber\\&\quad+\varepsilon^{\frac{r+1}{r}}C\int_{T}^{t}e^{\frac{\varrho}{4}(\upeta-t)}|\z(\vartheta_{\upeta}\omega)|^{\frac{r+1}{r}}\bigg(\|\u_{\varepsilon}(\upeta,\omega,\u_{\varepsilon}(0))\|^2_{\V}+1\bigg)\d\upeta, \ t>T.
		\end{align*}
		Therefore, replacing $\omega$ by $\vartheta_{-t}\omega$, we obtain
		\begin{align}\label{DA9}
		&	\|\w_{\varepsilon}(t,\vartheta_{-t}\omega,\w_{\varepsilon}(0))\|^2_{\H}\nonumber\\&\leq e^{-\frac{\varrho}{4}(t-T)}\|\w_{\varepsilon}(T,\vartheta_{-t}\omega,\w_{\varepsilon}(0))\|^2_{\H}\nonumber\\&\quad+\varepsilon^{\frac{r+1}{r}} C\int_{T}^{t}e^{\frac{\varrho}{4}(\upeta-t)}\left(|\z(\vartheta_{\upeta-t}\omega)|^2+|\z(\vartheta_{\upeta-t}\omega)|^{\frac{r+1}{r}}\right)\|\v_{\varepsilon}(\upeta,\vartheta_{-t}\omega,\v_{\varepsilon}(0))\|^2_{\V}\d\upeta\nonumber\\&\quad+\varepsilon^{\frac{r+1}{r}} C\int_{T-t}^{0}e^{\frac{\varrho\upeta}{4}}\left(|\z(\vartheta_{\upeta}\omega)|^{\frac{r+1}{r}}+|\z(\vartheta_{\upeta}\omega)|^2+|\z(\vartheta_{\upeta}\omega)|^{\frac{3r+1}{r}}+|\z(\vartheta_{\upeta}\omega)|^4\right)\d\upeta,
		\end{align}
		for all $t>T.$ We estimate the first term on the right hand side of \eqref{DA9} using \eqref{SH_ab} as 
		\begin{align}\label{DA10}
		& e^{-\frac{\varrho}{4}(t-T)}\|\w_{\varepsilon}(T,\vartheta_{-t}\omega,\w_{\varepsilon}(0))\|^2_{\H}\nonumber\\&\leq e^{-\frac{\varrho}{4}(t-T)}\left(\|\v_{\varepsilon}(T,\vartheta_{-t}\omega,\v_{\varepsilon}(0))\|^2_{\H} +\|\u(T,\x)\|^2_{\H}\right)\nonumber\\&\leq e^{-\frac{\varrho}{4}(t-T)}\bigg[e^{-\frac{\varrho T}{8}+\varepsilon^2K_1\int_{-t}^{T-t}|\z(\vartheta_{s}\omega)|^2\d s}\|\v_{\varepsilon}(0)\|^2_{\H}\nonumber\\&\quad+e^{\frac{\varrho}{8}(t-T)}\int_{-t}^{T-t}C e^{\frac{\varrho\tau}{8}+\varepsilon^2K_1\int_{\tau}^{T-t}|\z(\vartheta_{s}\omega)|^2\d s}\ \big(\|\f\|_{\H}^2 +  \varepsilon^{j} |\z(\vartheta_{\tau}\omega)|^{j}+1\big)\ \d\tau\bigg]\nonumber\\&\quad+ e^{-\frac{\varrho}{4}(t-T)}\|\u(T,\x)\|^2_{\H}\nonumber\\&\leq e^{-(t-T)\left(\frac{\varrho}{8}-\varepsilon^2K_1\frac{\int_{-t}^{T-t} |\z(\vartheta_{s}\omega)|^2\d s}{t-T}\right)}\bigg[e^{-\frac{\varrho t}{8}}\|\v_{\varepsilon}(0)\|^2_{\H}+\int_{-\infty}^{0}C e^{\frac{\varrho\tau}{8}}\ \big(\|\f\|_{\H}^2 + \varepsilon^{4} |\z(\vartheta_{\tau}\omega)|^{4}+1\big)\ \d\tau\bigg]\nonumber\\&\quad+ e^{-\frac{\varrho}{4}(t-T)}\|\u(T,\x)\|^2_{\H}.
		\end{align}
		Making use of Lemma \ref{z_esti}, tempered property of $\v_{\varepsilon}(0)$ and \eqref{Z5} in \eqref{DA10}, we get
		\begin{align}\label{DA11}
		e^{-\frac{\varrho}{4}(t-T)}\|\w_{\varepsilon}(T,\vartheta_{-t}\omega,\w_{\varepsilon}(0))\|^2_{\H}\to 0 \ \text{ as } \ t\to\infty.
		\end{align}
		We estimate the second term on the right hand side of \eqref{DA9} using \eqref{SV_ab} as follows
		\begin{align}\label{DA12}
		&\int_{T}^{t}e^{\frac{\varrho}{4}(\upeta-t)}\left(|\z(\vartheta_{\upeta-t}\omega)|^2+|\z(\vartheta_{\upeta-t}\omega)|^{\frac{r+1}{r}}\right)\|\v_{\varepsilon}(\upeta,\vartheta_{-t}\omega,\v_{\varepsilon}(0))\|^2_{\V}\d\upeta\nonumber\\&\leq\int_{T}^{t}e^{\frac{\varrho}{4}(\upeta-t)}\left(|\z(\vartheta_{\upeta-t}\omega)|^2+|\z(\vartheta_{\upeta-t}\omega)|^{\frac{r+1}{r}}\right)\bigg[e^{-\frac{\varrho \upeta}{8}+\varepsilon^2K_1\int_{-t}^{\upeta-t}|\z(\vartheta_{s}\omega)|^2\d s}\|\v_{\varepsilon}(0)\|^2_{\H}\nonumber\\&\quad+e^{\frac{\varrho}{8}(t-\upeta)}\int_{-t}^{\upeta-t}C e^{\frac{\varrho\xi}{8}+\varepsilon^2K_1\int_{\xi}^{\upeta-t}|\z(\vartheta_{s}\omega)|^2\d s}\ \big(\|\f\|_{\H}^2+ \varepsilon^{4}|\z(\vartheta_{\xi }\omega)|^{4}+ 1\big)\d \xi\bigg]\d \upeta\nonumber\\&=e^{-\frac{\varrho t}{8}} \int_{T-t}^{0}e^{\frac{\varrho\upeta}{8}+\varepsilon^2K_1\int_{-t}^{\upeta}|\z(\vartheta_{s}\omega)|^2\d s}\left(|\z(\vartheta_{\upeta}\omega)|^2+|\z(\vartheta_{\upeta}\omega)|^{\frac{r+1}{r}}\right)\|\v_{\varepsilon}(0)\|^2_{\H}\d\upeta\nonumber\\&\quad+\int_{T-t}^{0} e^{\frac{\varrho\upeta}{8}}\left(|\z(\vartheta_{\upeta}\omega)|^2+|\z(\vartheta_{\upeta}\omega)|^{\frac{r+1}{r}}\right)\int_{-t}^{\upeta}Ce^{\frac{\varrho\xi}{8}+\varepsilon^2K_1\int_{\xi}^{\upeta}|\z(\vartheta_{s}\omega)|^2\d s}\ \big(\|\f\|_{\H}^2\nonumber\\&\quad+ \varepsilon^{4}|\z(\vartheta_{\xi }\omega)|^{4}+ 1\big)\d \xi\d \upeta\nonumber\\&\leq  e^{-\frac{\varrho t}{8}+K_1\int_{-t}^{0}|\z(\vartheta_{s}\omega)|^2\d s}\int_{T-t}^{0}e^{\frac{\varrho\upeta}{8}}\left(|\z(\vartheta_{\upeta}\omega)|^2+|\z(\vartheta_{\upeta}\omega)|^{\frac{r+1}{r}}\right)\|\v_{\varepsilon}(0)\|^2_{\H}\d\upeta\nonumber\\&\quad+\int_{T-t}^{0} e^{\frac{\varrho\upeta}{8}}\left(|\z(\vartheta_{\upeta}\omega)|^2+|\z(\vartheta_{\upeta}\omega)|^{\frac{r+1}{r}}\right)\d \upeta\int_{-t}^{0}Ce^{\frac{\varrho\xi}{8}+K_1\int_{\xi}^{0}|\z(\vartheta_{s}\omega)|^2\d s}\ \big(\|\f\|_{\H}^2\nonumber\\&\quad+|\z(\vartheta_{\xi }\omega)|^{4}+ 1\big)\d \xi.
		\end{align}
		Now, we consider $\alpha$ large enough $\left( \alpha>\frac{8K_1\Gamma(\frac{3}{2})}{\varrho\sqrt{\pi}}\right)$ such that from \eqref{Z2}, we have
		\begin{align}\label{Z6}
		K_1\mathbb{E}(|\z(\cdot)|^2)<\frac{\varrho}{8} .
		\end{align}
		Then by ergodic theory and the temperedness of $s\mapsto|\z(\vartheta_{s}\omega)|^{k}$ as $s\to-\infty, \ k\in \mathbb{N},$ for the second term, we find 
		\begin{align}\label{DA13}
		\limsup_{t\to \infty} \int_{T}^{t} e^{\frac{\varrho}{4}(\upeta-t)}\left(|\z(\vartheta_{\upeta-t}\omega)|^2+|\z(\vartheta_{\upeta-t}\omega)|^{\frac{r+1}{r}}\right)\|\v_{\varepsilon}(\upeta,\vartheta_{-t}\omega, \v_{\varepsilon}(0))\|^2_{\V} \d\upeta \leq \kappa(\omega),
		\end{align}
		where $\kappa(\omega)$ is a random variable (independent of $\varepsilon$) given by
		\begin{align*}
		\kappa(\omega) &:= C\int_{-\infty}^{0} e^{\frac{\varrho\upeta}{8}}\left(|\z(\vartheta_{\upeta}\omega)|^2+|\z(\vartheta_{\upeta}\omega)|^{\frac{r+1}{r}}\right)\d \upeta\int_{-\infty}^{0}e^{\frac{\varrho\xi}{8}+K_1\int_{\xi}^{0}|\z(\vartheta_{s}\omega)|^2\d s}\ \big(\|\f\|_{\H}^2\nonumber\\&\qquad+|\z(\vartheta_{\xi }\omega)|^{4}+ 1\big)\d \xi.
		\end{align*}
		Hence by \eqref{DA9}, \eqref{DA11} and \eqref{DA13}, one can easily conclude the proof of lemma.
	\end{proof}
	\begin{theorem}\label{Conver-add}
		For $n=2 \text{ and } 1\leq r\leq2$, let $\f\in \H$. Then, there exists a random variable $R(\omega)$ (independent of $\varepsilon$) such that for every $\varepsilon\in(0,1],$ the additive noise random attractor $\mathcal{A}_{\varepsilon}$ and the deterministic attractor $\mathcal{A}$ satisfy
		\begin{align*}
		\emph{dist}_{\H}(\mathcal{A}_{\varepsilon}(\omega),\mathcal{A}) \leq \varepsilon^{\frac{r+1}{2r}} R(\omega), \ \text{ for all }\ \omega\in \Omega,
		\end{align*}
		where $\emph{dist}_{\H}(A,B)=\max\{\emph{dist}(A,B),\emph{dist}(B,A)\}$.
	\end{theorem}
	\begin{proof}
		Let $\emph{\textbf{a}}_{\varepsilon}(\omega)$ be an arbitrary element of $\mathcal{A}_{\varepsilon}(\omega)$. Then by the invariance property of random attractor, we have $\emph{\textbf{a}}_{\varepsilon}(\omega)=\v_{\varepsilon}(t,\vartheta_{-t}\omega,\emph{\textbf{a}}_{\varepsilon}(\vartheta_{-t}\omega)),$ for some $\emph{\textbf{a}}(\vartheta_{-t}\omega)\in\mathcal{A}_{\varepsilon}(\vartheta_{-t}\omega), t\geq 0.$ From Theorem \ref{D-SA}, we know that the deterministic attractor $\mathcal{A}=\{\emph{\textbf{a}}_{*}\}$ is a singleton. Therefore, we have 
		\begin{align*}
		\|\emph{\textbf{a}}_{\varepsilon}(\omega)-\emph{\textbf{a}}_{*}\|_{\H} = \|\v_{\varepsilon}(t,\vartheta_{-t}\omega,\emph{\textbf{a}}_{\varepsilon}(\vartheta_{-t}\omega))-\u(t,\emph{\textbf{a}}_{*})\|_{\H}, \ \text{ for all }\ t\geq 0.
		\end{align*}
		Taking $t\to\infty$ and applying Lemma \ref{PertRad-add}, one can conclude the proof.
	\end{proof}
	\begin{remark}
		Our main concern in this work is to prove the convergence of random attractor towards the deterministic singleton attractor. From Theorem \ref{Conver-add}, it is clear that we get the same order of convergence (order of $\e$) as stochastic NSE with additive noise for linear perturbation (that is, for $r=1$) only, see \cite{HCPEK}. For quadratic growth, the order of convergence becomes $\e^{\frac{3}{4}}$. 
	\end{remark}
	\section{2D Stochastic CBF equations with Multiplicative Noise} \label{sec5}\setcounter{equation}{0}
	Let us consider the 2D stochastic convective Brinkman-Forchheimer equations perturbed by multiplicative white noise as follows:
	\begin{equation}\label{SCBF-multi}
	\left\{
	\begin{aligned}
	\frac{\d\u_{\varepsilon}(t)}{\d t}+\mu \A\u_{\varepsilon}(t)+\B(\u_{\varepsilon}(t)) +\beta\mathcal{C}(\u_{\varepsilon}(t))&=\f +\varepsilon \u_{\varepsilon}\circ\frac{\d \W(t)}{\d t} , \quad t\geq 0, \\ 
	\u_{\varepsilon}(0)&=\x,
	\end{aligned}
	\right.
	\end{equation}
	with periodic boundary conditions defined in section \ref{sec2}, where $\f\in \H$ and $\circ$ means that the stochastic integral need to be understood in the  Stratonovich sense. 
	
	Let us define $\v_{\varepsilon}(t, \omega, \v_{\varepsilon}(0))=e^{-\varepsilon\z(\vartheta_{t}\omega)}\u_{\varepsilon}(t, \omega, \u_{\varepsilon}(0))$ with $\v_{\varepsilon}(0)=e^{-\varepsilon\z(\omega)}\u_{\varepsilon}(0)$ where $\z$ is defined by \eqref{OU1} and satisfies \eqref{OU2}, and $\u_{\varepsilon}(\cdot)$ is the solution of \eqref{SCBF-multi}. Then $\v_{\varepsilon}(\cdot)$ satisfies the following:
	\begin{equation}\label{CCBF-multi}
	\left\{
	\begin{aligned}
	\frac{\d\v_{\varepsilon}(t)}{\d t}+\mu \A\v_{\varepsilon}(t)&+e^{\varepsilon\z(\vartheta_{t}\omega)}\B\big(\v_{\varepsilon}(t)\big) +\beta e^{\varepsilon(r-1)\z(\vartheta_{t}\omega)}\mathcal{C}\big(\v_{\varepsilon}(t)\big)\\&=\f e^{-\varepsilon\z(\vartheta_{t}\omega)} + \varepsilon\alpha\z(\vartheta_t\omega)\v_{\varepsilon} , \quad t\geq 0, \\ 
	\v_{\varepsilon}(0)&=e^{-\varepsilon\z(\omega)}\u_{\varepsilon}(0).
	\end{aligned}
	\right.
	\end{equation}
	\subsection{Uniform estimates of solutions}
	The following lemma is useful for our main result of this section.
	\begin{lemma}
		For $n=2 \text{ and } r\geq 1$, let $\f\in \H$. Then for any $\varepsilon\in (0,1]$, the solution $\v_{\varepsilon}$ of \eqref{CCBF-multi} with initial data $\v_{\varepsilon}(0)\in \H$ satisfies 
		\begin{align}\label{M-uni}
		&\|\v_{\varepsilon}(s,\vartheta_{-t}\omega,\v_{\varepsilon}(0))\|^2_{\V}\nonumber\\&\leq Ce^{-\frac{\varrho}{8}s+\int_{-t}^{s-t}2\varepsilon\alpha\z(\vartheta_{\upeta}\omega)\d \upeta}\|\v_{\varepsilon}(0)\|^2_{\H} + Ce^{\frac{\varrho}{8}(t-s)} \int_{-t}^{s-t}e^{-2\varepsilon\z(\vartheta_{\upeta}\omega)+\frac{\varrho}{8}\upeta-\int_{s-t}^{\upeta}2\varepsilon\alpha\z(\vartheta_{\zeta}\omega)\d \zeta}\d \upeta,  
		\end{align}
		for $t\geq s>1.$
	\end{lemma}
	\begin{proof}
		Taking the inner product of the first equation in \eqref{CCBF-multi} with $\v_{\varepsilon}(\cdot)$ and using the fact that $b(\v_{\varepsilon},\v_{\varepsilon},\v_{\varepsilon})=0$, we get
		\begin{align}\label{M-uni1}
		&\frac{1}{2}\frac{\d}{\d t}\|\v_{\varepsilon}(t)\|^2_{\H} + \mu\|\v_{\varepsilon}(t)\|^2_{\V} +\beta e^{\varepsilon(r-1)\z(\vartheta_{t}\omega)} \|\v_{\varepsilon}(t)\|^{r+1}_{\wi \L^{r+1}}\nonumber\\&= e^{-\varepsilon\z(\vartheta_{t}\omega)}(\f,\v_{\varepsilon}(t)) + \varepsilon\alpha\z(\vartheta_{t}\omega)\|\v_{\varepsilon}(t)\|^2_{\H}\nonumber\\&\leq \frac{e^{-2\varepsilon\z(\vartheta_{t}\omega)}\|\f\|^2_{\H}}{\mu\lambda_1} + \varepsilon\alpha\z(\vartheta_{t}\omega)\|\v_{\varepsilon}(t)\|^2_{\H} +\frac{\mu\lambda_1}{4}\|\v_{\varepsilon}(t)\|^2_{\H},
		\end{align}
		for a.e. $t\in[0,\tau]$. 
		Using \eqref{poin} in \eqref{M-uni1}, we deduce that 
		\begin{align*}
		&\frac{\d}{\d t}\|\v_{\varepsilon}(t)\|^2_{\H} + \frac{\mu}{2}\|\v_{\varepsilon}(t)\|^2_{\V}+(\mu\lambda_1-2\varepsilon\alpha\z(\vartheta_{t}\omega))\|\v_{\varepsilon}(t)\|^2_{\H} +\beta e^{\varepsilon(r-1)\z(\vartheta_{t}\omega)} \|\v_{\varepsilon}(t)\|^{r+1}_{\wi \L^{r+1}}\nonumber\\&\leq \frac{2e^{-2\varepsilon\z(\vartheta_{t}\omega)}\|\f\|^2_{\H}}{\mu\lambda_1},
		\end{align*}
		for a.e. $t\in[0,\tau]$. From \eqref{C_2}, it is clear that $\frac{\varrho}{8}<\mu\lambda_1$, hence we find
		\begin{align}\label{M-uni3}
		&\frac{\d}{\d t}\|\v_{\varepsilon}(t)\|^2_{\H} + \frac{\mu}{2}\|\v_{\varepsilon}(t)\|^2_{\V}+\left(\frac{\varrho}{8}-2\varepsilon\alpha\z(\vartheta_{t}\omega)\right)\|\v_{\varepsilon}(t)\|^2_{\H} +\beta e^{\varepsilon(r-1)\z(\vartheta_{t}\omega)} \|\v_{\varepsilon}(t)\|^{r+1}_{\wi \L^{r+1}}\nonumber\\&\leq \frac{2e^{-2\varepsilon\z(\vartheta_{t}\omega)}\|\f\|^2_{\H}}{\mu\lambda_1},
		\end{align}
		for a.e. $t\in[0,\tau]$. Variation of constants formula yields 
		\begin{align}\label{M-uni4}
		&\|\v_{\varepsilon}(s,\omega,\v_{\varepsilon}(0))\|^2_{\H}+\frac{\mu}{2}\int_{0}^{s}e^{\int_{s}^{\upeta}\left(\frac{\varrho}{8}-2\varepsilon\alpha\z(\vartheta_{\zeta}\omega)\right)\d \zeta}\|\v_{\varepsilon}(\upeta,\omega,\v_{\varepsilon}(0))\|^2_{\V}\d\upeta\nonumber\\&\quad+\beta\int_{0}^{s}e^{\varepsilon(r-1)\z(\vartheta_{\upeta}\omega)+\int_{s}^{\upeta}\left(\frac{\varrho}{8}-2\varepsilon\alpha\z(\vartheta_{\zeta}\omega)\right)\d \zeta}\|\v_{\varepsilon}(\upeta,\omega,\v_{\varepsilon}(0))\|^{r+1}_{\wi \L^{r+1}}\d\upeta\nonumber\\&\leq e^{-\int_{0}^{s}\left(\frac{\varrho}{8}-2\varepsilon\alpha\z(\vartheta_{\upeta}\omega)\right)\d \upeta}\|\v_{\varepsilon}(0)\|^2_{\H} + \frac{2\|\f\|^2_{\H}}{\mu\lambda_1} \int_{0}^{s}e^{-2\varepsilon\z(\vartheta_{\upeta}\omega)+\int_{s}^{\upeta}\left(\frac{\varrho}{8}-2\varepsilon\alpha\z(\vartheta_{\zeta}\omega)\right)\d \zeta}\d \upeta,
		\end{align}
		for all $s\geq 0$. Hence, for any $t\geq s\geq0$, we have 
		\begin{align}\label{M-uni5}
		&\|\v_{\varepsilon}(s,\vartheta_{-t}\omega,\v_{\varepsilon}(0))\|^2_{\H}+\frac{\mu}{2}\int_{0}^{s}e^{\int_{s}^{\upeta}\left(\frac{\varrho}{8}-2\varepsilon\alpha\z(\vartheta_{\zeta-t}\omega)\right)\d \zeta}\|\v_{\varepsilon}(\upeta,\vartheta_{-t}\omega,\v_{\varepsilon}(0))\|^2_{\V}\d\upeta\nonumber\\&\quad+\beta\int_{0}^{s}e^{\varepsilon(r-1)\z(\vartheta_{\upeta-t}\omega)+\int_{s}^{\upeta}\left(\frac{\varrho}{8}-2\varepsilon\alpha\z(\vartheta_{\zeta-t}\omega)\right)\d \zeta}\|\v_{\varepsilon}(\upeta,\vartheta_{-t}\omega,\v_{\varepsilon}(0))\|^{r+1}_{\wi \L^{r+1}}\d\upeta\nonumber\\&\leq e^{-\int_{0}^{s}\left(\frac{\varrho}{8}-2\varepsilon\alpha\z(\vartheta_{\upeta-t}\omega)\right)\d \upeta}\|\v_{\varepsilon}(0)\|^2_{\H} + \frac{2\|\f\|^2_{\H}}{\mu\lambda_1} \int_{0}^{s}e^{-2\varepsilon\z(\vartheta_{\upeta-t}\omega)+\int_{s}^{\upeta}\left(\frac{\varrho}{8}-2\varepsilon\alpha\z(\vartheta_{\zeta-t}\omega)\right)\d \zeta}\d \upeta\nonumber\\&=e^{-\frac{\varrho}{8}s+\int_{-t}^{s-t}2\varepsilon\alpha\z(\vartheta_{\upeta}\omega)\d \upeta}\|\v_{\varepsilon}(0)\|^2_{\H} + \frac{2\|\f\|^2_{\H}}{\mu\lambda_1}e^{\frac{\varrho}{8}(t-s)} \int_{-t}^{s-t}e^{-2\varepsilon\z(\vartheta_{\upeta}\omega)+\frac{\varrho}{8}\upeta-\int_{s-t}^{\upeta}2\varepsilon\alpha\z(\vartheta_{\zeta}\omega)\d \zeta}\d \upeta.
		\end{align}
		Taking the inner product of the first equation in \eqref{CCBF-multi} with $\A\v_{\varepsilon}(\cdot)$, we find 
		\begin{align}\label{M-uni6}
		&\frac{1}{2}\frac{\d}{\d t}\|\v_{\varepsilon}(t)\|^2_{\V}+\mu\|\A\v_{\varepsilon}(t)\|^2_{\H} +\beta e^{\varepsilon(r-1)\z(\vartheta_{t}\omega)}\int_{\mathbb{T}^2}|\nabla\v_{\varepsilon}(t,x)|^2|\v_{\varepsilon}(t,x)|^{r-1}\d x\nonumber\\&\quad +4\beta e^{\varepsilon(r-1)\z(\vartheta_{t}\omega)}\left[\frac{r-1}{(r+1)^2}\right]\int_{\mathbb{T}^2}|\nabla|\v_{\varepsilon}(t,x)|^{\frac{r+1}{2}}|^2\d x\nonumber\\&= e^{-\varepsilon\z(\vartheta_{t}\omega)}(\f,\A\v_{\varepsilon}(t)) + \varepsilon\alpha\z(\vartheta_{t}\omega)\|\v_{\varepsilon}(t)\|^2_{\V},
		\end{align}
		where we used the fact that
		\begin{align*}
		b(\v_{\varepsilon},\v_{\varepsilon},\A\v_{\varepsilon})=0,
		\end{align*}
		and
		\begin{align*}
		\big\langle\mathcal{C}\big(\v_{\varepsilon}\big),\A\v_{\varepsilon}\big\rangle&=\int_{\mathbb{T}^2}|\nabla\v_{\varepsilon}(x)|^2|\v_{\varepsilon}(x)|^{r-1}\d x+4\left[\frac{r-1}{(r+1)^2}\right]\int_{\mathbb{T}^2}|\nabla|\v_{\varepsilon}(x)|^{\frac{r+1}{2}}|^2\d x.
		\end{align*}
		in  2D periodic domains. Applying Cauchy-Schwarz and Young's inequalities in the left hand side of \eqref{M-uni6}, we obtain
		\begin{align}\label{M-uni7}
		\frac{1}{2}\frac{\d}{\d t}\|\v_{\varepsilon}(t)\|^2_{\V}+\mu\|\A\v_{\varepsilon}(t)\|^2_{\H}\leq \frac{e^{-2\varepsilon\z(\vartheta_{t}\omega)}\|\f\|^2_{\H}}{2\mu} + \varepsilon\alpha\z(\vartheta_{t}\omega)\|\v_{\varepsilon}(t)\|^2_{\V} +\frac{\mu}{2}\|\A\v_{\varepsilon}(t)\|^2_{\H},
		\end{align}
		for a.e. $t\in[0,\tau]$. Also, using \eqref{poin_1} in  \eqref{M-uni7}, we deduce that 
		\begin{align*}
		\frac{\d}{\d t}\|\v_{\varepsilon}(t)\|^2_{\V}+\left(\mu\lambda_1-2\varepsilon\alpha\z(\vartheta_{t}\omega)\right)\|\v_{\varepsilon}(t)\|^2_{\V}\leq \frac{e^{-2\varepsilon\z(\vartheta_{t}\omega)}\|\f\|^2_{\H}}{\mu},
		\end{align*}
		for a.e. $t\in[0,\tau]$. From \eqref{C_2}, it is clear that $\frac{\varrho}{8}<\mu\lambda_1$, and thus we get 
		\begin{align*}
		\frac{\d}{\d t}\|\v_{\varepsilon}(t)\|^2_{\V}+\left(\frac{\varrho}{8}-2\varepsilon\alpha\z(\vartheta_{t}\omega)\right)\|\v_{\varepsilon}(t)\|^2_{\V}\leq \frac{e^{-2\varepsilon\z(\vartheta_{t}\omega)}\|\f\|^2_{\H}}{\mu}, 
		\end{align*}
		for a.e. $t\in[0,\tau]$. Applying the same argument as of \eqref{M-uni5}, for $t\geq s\geq \uprho>0,$ we arrive at
		\begin{align}\label{M-uni9}
		&\|\v_{\varepsilon}(s,\vartheta_{-t}\omega,\v_{\varepsilon}(0))\|^2_{\V}\nonumber\\&\leq e^{-\int_{\uprho}^{s}\left(\frac{\varrho}{8}-2\varepsilon\alpha\z(\vartheta_{\upeta-t}\omega)\right)\d \upeta}\|\v_{\varepsilon}(\uprho,\vartheta_{-t}\omega,\v_{\varepsilon}(0))\|^2_{\V} + \frac{\|\f\|^2_{\H}}{\mu} \int_{\uprho}^{s}e^{-2\varepsilon\z(\vartheta_{\upeta-t}\omega)+\int_{s}^{\upeta}\left(\frac{\varrho}{8}-2\varepsilon\alpha\z(\vartheta_{\zeta-t}\omega)\right)\d \zeta}\d \upeta\nonumber\\&\leq e^{-\int_{\uprho}^{s}\left(\frac{\varrho}{8}-2\varepsilon\alpha\z(\vartheta_{\upeta-t}\omega)\right)\d \upeta}\|\v_{\varepsilon}(\uprho,\vartheta_{-t}\omega,\v_{\varepsilon}(0))\|^2_{\V} + \frac{\|\f\|^2_{\H}}{\mu} \int_{0}^{s}e^{-2\varepsilon\z(\vartheta_{\upeta-t}\omega)+\int_{s}^{\upeta}\left(\frac{\varrho}{8}-2\varepsilon\alpha\z(\vartheta_{\zeta-t}\omega)\right)\d \zeta}\d \upeta.
		\end{align}
		Integrating \eqref{M-uni9} with respect to $\uprho$ over $(s-1,s),$ for $s>1$ and using \eqref{M-uni5}, we obtain
		\begin{align*}
		\|\v_{\varepsilon}(s,\vartheta_{-t}\omega,\v_{\varepsilon}(0))\|^2_{\V}&\leq \int_{s-1}^{s}e^{-\int_{\uprho}^{s}\left(\frac{\varrho}{8}-2\varepsilon\alpha\z(\vartheta_{\upeta-t}\omega)\right)\d \upeta}\|\v_{\varepsilon}(\uprho,\vartheta_{-t}\omega,\v_{\varepsilon}(0))\|^2_{\V}\d \uprho \nonumber\\&\quad+ \frac{\|\f\|^2_{\H}}{\mu} \int_{0}^{s}e^{-2\varepsilon\z(\vartheta_{\upeta-t}\omega)+\int_{s}^{\upeta}\left(\frac{\varrho}{8}-2\varepsilon\alpha\z(\vartheta_{\zeta-t}\omega)\right)\d \zeta}\d \upeta\nonumber\\&\leq Ce^{-\int_{0}^{s}\left(\frac{\varrho}{8}-2\varepsilon\alpha\z(\vartheta_{\upeta-t}\omega)\right)\d \upeta}\|\v_{\varepsilon}(0)\|^2_{\H} \nonumber\\&\quad+ C \int_{0}^{s}e^{-2\varepsilon\z(\vartheta_{\upeta-t}\omega)+\int_{s}^{\upeta}\left(\frac{\varrho}{8}-2\varepsilon\alpha\z(\vartheta_{\zeta-t}\omega)\right)\d \zeta}\d \upeta\nonumber\\&= Ce^{-\frac{\varrho}{8}s+\int_{-t}^{s-t}2\varepsilon\alpha\z(\vartheta_{\upeta}\omega)\d \upeta}\|\v_{\varepsilon}(0)\|^2_{\H} \nonumber\\&\quad+ Ce^{\frac{\varrho}{8}(t-s)} \int_{-t}^{s-t}e^{-2\varepsilon\z(\vartheta_{\upeta}\omega)+\frac{\varrho}{8}\upeta-\int_{s-t}^{\upeta}2\varepsilon\alpha\z(\vartheta_{\zeta}\omega)\d \zeta}\d \upeta,
		\end{align*}
		for $t\geq s>1$. 
	\end{proof}
	\subsection{Perturbation radius of the singleton attractor under multiplicative noise}
	Let us now consider the difference $\w_{\varepsilon}(\cdot)=\v_{\varepsilon}(\cdot)-\u(\cdot)$, where $\v_{\varepsilon}(\cdot)$ and $\u(\cdot)$ are the solutions of \eqref{CCBF-multi} and \eqref{D-CBF}, respectively. It is obvious that $\w_{\varepsilon}(\cdot)$ satisfies 
	\begin{equation}\label{Diff-CBF_multi}
	\left\{
	\begin{aligned}
	\frac{\d\w_{\varepsilon}}{\d t}&=-\mu \A\w_{\varepsilon}-e^{\varepsilon\z(\vartheta_{t}\omega)}\B(\v_{\varepsilon})+\B(\u) -\beta e^{\varepsilon(r-1)\z(\vartheta_{t}\omega)}\mathcal{C}(\v_{\varepsilon})+\beta\mathcal{C}(\u)\\&\quad+ \left(e^{-\varepsilon\z(\vartheta_{t}\omega)}-1\right)\f+ \varepsilon\alpha\z(\vartheta_t\omega)\v_{\varepsilon} , \\ 
	\w_{\varepsilon}(0)&=(e^{-\varepsilon\z(\vartheta_{t}\omega)}-1)\x,
	\end{aligned}
	\right.
	\end{equation}
	in $\V'$. We point out that in the case of multiplicative noise, we consider the absorption constant $1\leq r<\infty$. 
	\begin{lemma}\label{PertRad-multi}
		For $n=2$ and $r\geq1$, let $\f\in \H$. Then for any $\varepsilon\in(0,1]$, there exists a random variable $\gamma_{\varepsilon}(\omega)$ such that 
		\begin{align*}
		\limsup_{t\to \infty} \|\w_{\varepsilon}(t,\vartheta_{-t}\omega,\w_{\varepsilon}(0))\|^2_{\H}\leq \gamma_{\varepsilon}(\omega),\quad \varepsilon\in(0,1], \ \omega\in \Omega.
		\end{align*}
		Moreover, as a mapping of $\varepsilon$, $\gamma_{\varepsilon}(\omega)\sim\varepsilon^2$ as $\varepsilon\to0$, for every $\omega\in\Omega$.
	\end{lemma}
	\begin{proof}
		Taking inner product of first equation of \eqref{Diff-CBF_multi} with $\w_{\varepsilon}(\cdot)$, we get
		\begin{align}\label{W-multi1}
		\frac{1}{2}\frac{\d}{\d t}\|\w_{\varepsilon}(t)\|^2_{\H} &=- \mu \|\w_{\varepsilon}(t)\|^2_{\V} -\left[e^{\varepsilon\z(\vartheta_{t}\omega)}b(\v_{\varepsilon}(t),\v_{\varepsilon}(t),\w_{\varepsilon}(t))-b(\u(t),\u(t),\w_{\varepsilon}(t))\right]\nonumber\\&\quad-\beta\left\langle e^{\varepsilon(r-1)\z(\vartheta_{t}\omega)}\mathcal{C}(\v_{\varepsilon}(t))-\mathcal{C}(\u(t)),\w_{\varepsilon}(t)\right\rangle+\left(e^{-\varepsilon\z(\vartheta_{t}\omega)}-1\right)(\f,\w_{\varepsilon}(t))\nonumber\\&\quad+ (\varepsilon\alpha\z(\vartheta_t\omega)\v_{\varepsilon}(t),\w_{\varepsilon}(t))\nonumber\\&=- \mu \|\w_{\varepsilon}(t)\|^2_{\V} -\left[(e^{\varepsilon\z(\vartheta_{t}\omega)}-1)b(\v_{\varepsilon}(t),\v_{\varepsilon}(t),\w_{\varepsilon}(t))+b(\w_{\varepsilon}(t),\v_{\varepsilon}(t),\w_{\varepsilon}(t))\right]\nonumber\\&\quad-\beta e^{\varepsilon(r-1)\z(\vartheta_{t}\omega)}\left\langle\mathcal{C}(\v_{\varepsilon}(t))-\mathcal{C}(\u(t)),\v_{\varepsilon}(t)-\u(t)\right\rangle\nonumber\\&\quad-\beta(e^{\varepsilon(r-1)\z(\vartheta_{t}\omega)}-1)\left\langle\mathcal{C}(\u(t)),\w_{\varepsilon}(t)\right\rangle+\left(e^{-\varepsilon\z(\vartheta_{t}\omega)}-1\right)(\f,\w_{\varepsilon}(t))\nonumber\\&\quad+\varepsilon\alpha\z(\vartheta_t\omega)\|\w_{\varepsilon}(t)\|^2_{\H}+ (\varepsilon\alpha\z(\vartheta_t\omega)\u(t),\w_{\varepsilon}(t)),
		\end{align}
		for a.e. $t\in[0,\tau]$. Now, we estimate each term from the right hand side of \eqref{W-multi1} carefully to get our required estimate. By \eqref{b1} and Young's inequality, we have 
		\begin{align}\label{W-multi2}
		&|(e^{\varepsilon\z(\vartheta_{t}\omega)}-1)b(\v_{\varepsilon},\v_{\varepsilon},\w_{\varepsilon})+b(\w_{\varepsilon},\v_{\varepsilon},\w_{\varepsilon})|\nonumber\\&=|(e^{\varepsilon\z(\vartheta_{t}\omega)}-1)b(\v_{\varepsilon},\u,\w_{\varepsilon})+b(\w_{\varepsilon},\u,\w_{\varepsilon})|\nonumber\\&\leq C|e^{\varepsilon\z(\vartheta_{t}\omega)}-1|\|\v_{\varepsilon}\|_{\V}\|\u\|_{\V}\|\w_{\varepsilon}\|_{\V}+c_1\|\u\|_{\V}\|\w_{\varepsilon}\|_{\H}\|\w_{\varepsilon}\|_{\V}\nonumber\\&\leq C|e^{\varepsilon\z(\vartheta_{t}\omega)}-1|^2\|\v_{\varepsilon}\|^2_{\V}\|\u\|^2_{\V}+\frac{\varrho}{16\lambda_1}\|\w_{\varepsilon}\|^2_{\V} +\frac{c_1^2}{2\mu}\|\u\|^2_{\V}\|\w_{\varepsilon}\|^2_{\H} +\frac{\mu}{2}\|\w_{\varepsilon}\|^2_{\V}.
		\end{align}
		From \eqref{MO_c}, we have 
		\begin{align}\label{W-multi3}
		-\beta e^{\varepsilon(r-1)\z(\vartheta_{t}\omega)}\left\langle\mathcal{C}(\v_{\varepsilon})-\mathcal{C}(\u),\v_{\varepsilon}-\u\right\rangle\leq 0.
		\end{align}
		Applying H\"older's and Young's inequalities, and Sobolev embedding, we obtain
		\begin{align}\label{W-multi4}
		\left|(e^{\varepsilon(r-1)\z(\vartheta_{t}\omega)}-1)\left\langle\mathcal{C}(\u),\w_{\varepsilon}\right\rangle\right|&\leq\left|e^{\varepsilon(r-1)\z(\vartheta_{t}\omega)}-1\right|\|\u\|^{r}_{\wi\L^{2r}}\|\w_{\varepsilon}\|_{\H}\nonumber\\&\leq C\left|e^{\varepsilon(r-1)\z(\vartheta_{t}\omega)}-1\right|^2\|\u\|^{2r}_{\wi\L^{2r}}+\frac{\varrho}{32}\|\w_{\varepsilon}\|^2_{\H}\nonumber\\&\leq C\left|e^{\varepsilon(r-1)\z(\vartheta_{t}\omega)}-1\right|^2\|\u\|^{2r}_{\V}+\frac{\varrho}{32}\|\w_{\varepsilon}\|^2_{\H},
		\end{align}
		and 
		\begin{align}\label{W-multi5}
		&\left(e^{-\varepsilon\z(\vartheta_{t}\omega)}-1\right)(\f,\w_{\varepsilon})+\varepsilon\alpha\z(\vartheta_t\omega)\|\w_{\varepsilon}\|^2_{\H}+ (\varepsilon\alpha\z(\vartheta_t\omega)\u,\w_{\varepsilon})\nonumber\\&\leq C\left|e^{-\varepsilon\z(\vartheta_{t}\omega)}-1\right|^2+\left(\frac{\varrho}{32}+\varepsilon\alpha\z(\vartheta_{t}\omega)\right)\|\w_{\varepsilon}\|^2_{\H}+\varepsilon^2C\left|\z(\vartheta_{t}\omega)\right|^2\|\u\|^2_{\H}.
		\end{align}
		Combining \eqref{W-multi2}-\eqref{W-multi5} and using it in \eqref{W-multi1}, we get
		\begin{align*}
		&\frac{1}{2}\frac{\d}{\d t}\|\w_{\varepsilon}(t)\|^2_{\H}+\left(\frac{\mu}{2}-\frac{\varrho}{16\lambda_1}\right)\|\w_{\varepsilon}(t)\|^2_{\V}-\left(\frac{\varrho}{16}+\varepsilon\alpha\z(\vartheta_{t}\omega)+\frac{c_1^2}{2\mu}\|\u(t)\|^2_{\V}\right)\|\w_{\varepsilon}(t)\|^2_{\H}\nonumber\\&\leq\varepsilon^2C\left|\z(\vartheta_{t}\omega)\right|^2\|\u(t)\|^2_{\H}+	C|e^{\varepsilon\z(\vartheta_{t}\omega)}-1|^2\|\v_{\varepsilon}(t)\|^2_{\V}\|\u(t)\|^2_{\V}  +C\left|e^{\varepsilon(r-1)\z(\vartheta_{t}\omega)}-1\right|^2\|\u(t)\|^{2r}_{\V}\nonumber\\&\quad+C\left|e^{-\varepsilon\z(\vartheta_{t}\omega)}-1\right|^2,
		\end{align*}
		for a.e. $t\in[0,\tau]$. By \eqref{poin}, we have
		\begin{align}\label{W-multi6}
		&\frac{\d}{\d t}\|\w_{\varepsilon}(t)\|^2_{\H}+\left(\mu\lambda_1-\frac{c_1^2}{\mu}\|\u(t)\|^2_{\V}-\frac{\varrho}{4}-2\varepsilon\alpha\z(\vartheta_{t}\omega)\right)\|\w_{\varepsilon}(t)\|^2_{\H}\nonumber\\&\leq\varepsilon^2C\left|\z(\vartheta_{t}\omega)\right|^2\|\u(t)\|^2_{\H}+	C|e^{\varepsilon\z(\vartheta_{t}\omega)}-1|^2\|\v_{\varepsilon}(t)\|^2_{\V}\|\u(t)\|^2_{\V}  +C\left|e^{\varepsilon(r-1)\z(\vartheta_{t}\omega)}-1\right|^2\|\u(t)\|^{2r}_{\V}\nonumber\\&\quad+C\left|e^{-\varepsilon\z(\vartheta_{t}\omega)}-1\right|^2, 
		\end{align}
		for a.e. $t\in[0,\tau]$. Now, we consider again a time $T=T_{B_{\H},\frac{\varrho\mu}{2c_1^2}}>0$ such that \eqref{uni5} holds, that is,
		\begin{align*}
		\mu\lambda_1-\frac{c_1^2}{\mu}\|\u(t)\|^2_{\V}\geq\frac{\varrho}{2}, \ \text{ for all }\ t\geq T.
		\end{align*}
		It follows from \eqref{W-multi6} that 
		\begin{align}\label{W-multi7}
		&\frac{\d}{\d t}\|\w_{\varepsilon}(t)\|^2_{\H}+\left(\frac{\varrho}{4}-2\varepsilon\alpha\z(\vartheta_{t}\omega)\right)\|\w_{\varepsilon}(t)\|^2_{\H}\nonumber\\&\leq\varepsilon^2C\left|\z(\vartheta_{t}\omega)\right|^2+	C|e^{\varepsilon\z(\vartheta_{t}\omega)}-1|^2\|\v_{\varepsilon}(t)\|^2_{\V} +C\left|e^{\varepsilon(r-1)\z(\vartheta_{t}\omega)}-1\right|^2+C\left|e^{-\varepsilon\z(\vartheta_{t}\omega)}-1\right|^2, 
		\end{align}
		for a.e. $t\in[T,\tau]$. Variation of constants formula gives 
		\begin{align*}
		\|\w_{\varepsilon}(t,\omega,\w_{\varepsilon}(0))\|^2_{\H}&\leq e^{\int_{s}^{t}-\left(\frac{\varrho}{4}-2\varepsilon\alpha\z(\vartheta_{\upeta}\omega)\right)\d \upeta}\|\w_{\varepsilon}(s,\omega,\w_{\varepsilon}(0))\|^2_{\H}+C\int_{s}^{t}e^{\int_{\xi}^{t}-\left(\frac{\varrho}{4}-2\varepsilon\alpha\z(\vartheta_{\upeta}\omega)\right)\d \upeta}\nonumber\\&\quad\times\bigg(\varepsilon^2C\left|\z(\vartheta_{\xi}\omega)\right|^2+	C|e^{\varepsilon\z(\vartheta_{\xi}\omega)}-1|^2\|\v_{\varepsilon}(\xi)\|^2_{\V} +C\left|e^{\varepsilon(r-1)\z(\vartheta_{\xi}\omega)}-1\right|^2\nonumber\\&\quad+C\left|e^{-\varepsilon\z(\vartheta_{\xi}\omega)}-1\right|^2\bigg)\d\xi,
		\end{align*}
		for $T\leq s\leq t.$ Replacing $\omega$ with $\vartheta_{-t}\omega$, for $T\leq s\leq t$, we have 
		\begin{align}
		&\|\w_{\varepsilon}(t,\vartheta_{-t}\omega,\w_{\varepsilon}(0))\|^2_{\H}\nonumber\\&\leq e^{-\frac{\varrho}{4}(t-s)+\int_{s-t}^{0}2\varepsilon\alpha\z(\vartheta_{\upeta}\omega)\d\upeta}\|\w_{\varepsilon}(s,\vartheta_{-t}\omega,\w_{\varepsilon}(0))\|^2_{\H}+\varepsilon^2C\int_{s-t}^{0}e^{\frac{\varrho}{4}\xi+\int_{\xi}^{0}2\varepsilon\alpha\z(\vartheta_{\upeta}\omega)\d\upeta}\left|\z(\vartheta_{\xi}\omega)\right|^2\d\xi\nonumber\\&\quad+C\int_{s}^{t}e^{-\frac{\varrho}{4}(t-\xi)+\int_{\xi-t}^{0}2\varepsilon\alpha\z(\vartheta_{\upeta}\omega)\d\upeta}\bigg(	|e^{\varepsilon\z(\vartheta_{\xi-t}\omega)}-1|^2\|\v_{\varepsilon}(\xi,\vartheta_{-t}\omega,\v_{\varepsilon}(0))\|^2_{\V} \nonumber\\&\quad+\left|e^{\varepsilon(r-1)\z(\vartheta_{\xi-t}\omega)}-1\right|^2+\left|e^{-\varepsilon\z(\vartheta_{\xi-t}\omega)}-1\right|^2\bigg)\d\xi\label{W-multi8}\\&=:J_1(\varepsilon,t)+J_2(\varepsilon,t)+J_3(\varepsilon,t),\nonumber
		\end{align}
		where $J_1(\varepsilon,t), J_2(\varepsilon,t) \text{ and } J_3(\varepsilon,t)$ are the three terms appearing in the  right hand side of \eqref{W-multi8}. 
		Taking $\alpha>0$ large enough such that, by \eqref{Z2}, we get 
		\begin{align*}
		\mathbb{E}\left(2|\z(\cdot)|\right)<\frac{\varrho}{8}.
		\end{align*}
		By \eqref{M-uni}, we obtain 
		\begin{align}\label{W-multi9}
		J_1(\varepsilon,t)&\leq e^{-\frac{\varrho}{4}(t-s)+\int_{s-t}^{0}2\varepsilon\alpha\z(\vartheta_{\upeta}\omega)\d\upeta}\bigg(\|\v_{\varepsilon}(s,\vartheta_{-t}\omega,\v_{\varepsilon}(0))\|^2_{\H}+\|\u(s,\x)\|^2_{\H}\bigg) \nonumber\\&\leq Ce^{-\frac{\varrho}{4}(t-s)+\int_{s-t}^{0}2\varepsilon\alpha\z(\vartheta_{\upeta}\omega)\d\upeta} \bigg(e^{-\frac{\varrho}{8}s+\int_{-t}^{s-t}2\varepsilon\alpha\z(\vartheta_{\upeta}\omega)\d \upeta}\|\v_{\varepsilon}(0)\|^2_{\H} \nonumber\\&\quad+ e^{\frac{\varrho}{8}(t-s)} \int_{-t}^{s-t}e^{-2\varepsilon\z(\vartheta_{\upeta}\omega)+\frac{\varrho}{8}\upeta-\int_{s-t}^{\upeta}2\varepsilon\alpha\z(\vartheta_{\zeta}\omega)\d \zeta}\d \upeta+1\bigg) \nonumber\\&\leq Ce^{-\frac{\varrho}{8}(t-s)-\frac{\varrho}{8}t+\int_{-t}^{0}2\varepsilon\alpha\z(\vartheta_{\upeta}\omega)\d \upeta}\|\v_{\varepsilon}(0)\|^2_{\H} \nonumber\\&\quad+C e^{-\frac{\varrho}{8}(t-s)} \int_{-t}^{s-t}e^{-2\varepsilon\z(\vartheta_{\upeta}\omega)+\frac{\varrho}{8}\upeta-\int_{\upeta}^{0}2\varepsilon\alpha\z(\vartheta_{\zeta}\omega)\d \zeta}\d \upeta+Ce^{-\frac{\varrho}{4}(t-s)+\int_{s-t}^{0}2\varepsilon\alpha\z(\vartheta_{\upeta}\omega)\d\upeta}\nonumber\\&\to 0 \ \text{ as }\ t\to\infty.
		\end{align}
		For the second term, we find 
		\begin{align}\label{W-multi10}
		\limsup_{t\to \infty}J_2(\varepsilon,t)\leq\varepsilon^2C\int_{-\infty}^{0}	e^{\frac{\varrho}{4}\xi+\int_{\xi}^{0}2\varepsilon\alpha\z(\vartheta_{\upeta}\omega)\d\upeta}\left|\z(\vartheta_{\xi}\omega)\right|^2\d\xi.
		\end{align}
		For the third term, using \eqref{M-uni}, we deduce that 
		\begin{align}\label{W-multi11}
		& J_3(\varepsilon,t)\nonumber\\&=C\int_{s}^{t}e^{-\frac{\varrho}{4}(t-\xi)+\int_{\xi-t}^{0}2\varepsilon\alpha\z(\vartheta_{\upeta}\omega)\d\upeta}\bigg(	|e^{\varepsilon\z(\vartheta_{\xi-t}\omega)}-1|^2\|\v_{\varepsilon}(\xi,\vartheta_{-t}\omega,\v_{\varepsilon}(0))\|^2_{\V} \nonumber\\&\quad+\left|e^{\varepsilon(r-1)\z(\vartheta_{\xi-t}\omega)}-1\right|^2+\left|e^{-\varepsilon\z(\vartheta_{\xi-t}\omega)}-1\right|^2\bigg)\d\xi\nonumber\\&\leq C\int_{s}^{t}e^{-\frac{\varrho}{4}(t-\xi)+\int_{\xi-t}^{0}2\varepsilon\alpha\z(\vartheta_{\upeta}\omega)\d\upeta}\bigg(	e^{-\frac{\varrho}{8}\xi+\int_{-t}^{\xi-t}2\varepsilon\alpha\z(\vartheta_{\upeta}\omega)\d \upeta}|e^{\varepsilon\z(\vartheta_{\xi-t}\omega)}-1|^2\|\v_{\varepsilon}(0)\|^2_{\H} \nonumber\\&\quad+ e^{\frac{\varrho}{8}(t-\xi)} |e^{\varepsilon\z(\vartheta_{\xi-t}\omega)}-1|^2\int_{-t}^{\xi-t}e^{-2\varepsilon\z(\vartheta_{\upeta}\omega)+\frac{\varrho}{8}\upeta-\int_{\xi-t}^{\upeta}2\varepsilon\alpha\z(\vartheta_{\zeta}\omega)\d \zeta}\d \upeta +\left|e^{\varepsilon(r-1)\z(\vartheta_{\xi-t}\omega)}-1\right|^2\nonumber\\&\quad+\left|e^{-\varepsilon\z(\vartheta_{\xi-t}\omega)}-1\right|^2\bigg)\d\xi\nonumber\\&= C\int_{s}^{t}e^{-\frac{\varrho}{8}(t-\xi)-\frac{\varrho}{8}t+\int_{-t}^{0}2\varepsilon\alpha\z(\vartheta_{\upeta}\omega)\d\upeta}	|e^{\varepsilon\z(\vartheta_{\xi-t}\omega)}-1|^2\|\v_{\varepsilon}(0)\|^2_{\H}\d\xi \nonumber\\&\quad+C \int_{s}^{t}e^{-\frac{\varrho}{8}(t-\xi)} |e^{\varepsilon\z(\vartheta_{\xi-t}\omega)}-1|^2\int_{-t}^{\xi-t}e^{-2\varepsilon\z(\vartheta_{\upeta}\omega)+\frac{\varrho}{8}\upeta-\int_{\upeta}^{0}2\varepsilon\alpha\z(\vartheta_{\zeta}\omega)\d \zeta}\d \upeta\d\xi \nonumber\\&\quad+ C\int_{s}^{t}e^{-\frac{\varrho}{4}(t-\xi)+\int_{\xi-t}^{0}2\varepsilon\alpha\z(\vartheta_{\upeta}\omega)\d\upeta}\bigg(\left|e^{\varepsilon(r-1)\z(\vartheta_{\xi-t}\omega)}-1\right|^2+\left|e^{-\varepsilon\z(\vartheta_{\xi-t}\omega)}-1\right|^2\bigg)\d\xi\nonumber\\&\leq Ce^{-\frac{\varrho}{8}t+\int_{-t}^{0}2\varepsilon\alpha\z(\vartheta_{\upeta}\omega)\d\upeta}\int_{s-t}^{0}e^{\frac{\varrho}{8}\xi}	|e^{\varepsilon\z(\vartheta_{\xi}\omega)}-1|^2\|\v_{\varepsilon}(0)\|^2_{\H}\d\xi \nonumber\\&\quad+C \int_{s-t}^{0}e^{\frac{\varrho}{8}\xi} |e^{\varepsilon\z(\vartheta_{\xi}\omega)}-1|^2\d\xi\int_{-t}^{0}e^{-2\varepsilon\z(\vartheta_{\upeta}\omega)+\frac{\varrho}{8}\upeta-\int_{\upeta}^{0}2\varepsilon\alpha\z(\vartheta_{\zeta}\omega)\d \zeta}\d \upeta \nonumber\\&\quad+ C\int_{s-t}^{0}e^{\frac{\varrho}{4}\xi+\int_{\xi}^{0}2\varepsilon\alpha\z(\vartheta_{\upeta}\omega)\d\upeta}\bigg(\left|e^{\varepsilon(r-1)\z(\vartheta_{\xi}\omega)}-1\right|^2+\left|e^{-\varepsilon\z(\vartheta_{\xi}\omega)}-1\right|^2\bigg)\d\xi.
		\end{align}
		Therefore, \eqref{W-multi11} implies
		\begin{align}\label{W-multi12}
		\limsup_{t\to \infty}J_3(\varepsilon,t)&\leq C \int_{-\infty}^{0}e^{\frac{\varrho}{8}\xi} |e^{\varepsilon\z(\vartheta_{\xi}\omega)}-1|^2\d\xi\int_{-\infty}^{0}e^{-2\varepsilon\z(\vartheta_{\upeta}\omega)+\frac{\varrho}{8}\upeta-\int_{\upeta}^{0}2\varepsilon\alpha\z(\vartheta_{\zeta}\omega)\d \zeta}\d \upeta \nonumber\\&\quad+ C\int_{-\infty}^{0}e^{\frac{\varrho}{4}\xi+\int_{\xi}^{0}2\varepsilon\alpha\z(\vartheta_{\upeta}\omega)\d\upeta}\bigg(\left|e^{\varepsilon(r-1)\z(\vartheta_{\xi}\omega)}-1\right|^2+\left|e^{-\varepsilon\z(\vartheta_{\xi}\omega)}-1\right|^2\bigg)\d\xi\nonumber\\&=:\gamma^{(1)}_{\varepsilon}(\omega).
		\end{align}
		It is easy to see that $\gamma^{(1)}_{\varepsilon}(\omega)\sim\varepsilon^2$ as $\varepsilon\to0$. Hence, by \eqref{W-multi8}, \eqref{W-multi9}, \eqref{W-multi10} and \eqref{W-multi12}, one can easily conclude the proof. 
	\end{proof}
	\begin{theorem}\label{Conver-multi}
		For $n=2$ and $r\geq1$, let $\f\in\H$. Then, there exists a random variable $\gamma_{\varepsilon}(\omega)$ such that for every $\varepsilon\in(0,1]$, the multiplicative noise random attractor $\mathcal{A}_{\varepsilon}$ and the deterministic attractor $\mathcal{A}$ satisfy 
		\begin{align*}
		\emph{dist}_{\H}\left(\mathcal{A}_{\varepsilon}(\omega),\mathcal{A}\right)\leq\sqrt{\gamma_{\varepsilon}(\omega)},
		\end{align*}
		where $\emph{dist}_{\H}(A,B)=\max\{\emph{dist}(A,B),\emph{dist}(B,A)\}$. Moreover, as a mapping of $\varepsilon$, $\sqrt{\gamma_{\varepsilon}(\omega)}\sim\varepsilon$ as $\varepsilon\to0,$ for every $\omega\in\Omega$.
	\end{theorem}
	\begin{proof}
		Proof of this theorem is analogous to the proof of Theorem \ref{Conver-add}. Using Theorem \ref{PertRad-multi}, one can complete the proof.
	\end{proof}
	\begin{remark}
		It is interesting to note that in the case of multiplicative noise, the order of convergence of  random attractor towards the deterministic singleton attractor of the 2D CBF equations for $r\in[1,\infty)$ is $\e$, which is same as that of 2D NSE obtained in \cite{HCPEK}.
	\end{remark}

	\section{3D Stochastic CBF Equations with Multiplicative Noise} \label{sec6}\setcounter{equation}{0}
	Let us consider the 3D stochastic convective Brinkman-Forchheimer equations perturbed by multiplicative white noise as follows:
	\begin{equation}\label{3SCBF-multi}
	\left\{
	\begin{aligned}
	\frac{\d\u_{\varepsilon}(t)}{\d t}+\mu \A\u_{\varepsilon}(t)+\B(\u_{\varepsilon}(t)) +\beta\mathcal{C}(\u_{\varepsilon}(t))&=\f +\varepsilon \u_{\varepsilon}(t)\circ\frac{\d \W(t)}{\d t} , \quad t\geq 0, \\ 
	\u_{\varepsilon}(0)&=\x,
	\end{aligned}
	\right.
	\end{equation}
	with periodic boundary conditions defined in section \ref{sec2}, where $\f\in \H$. 
	
	Let us define $\v_{\varepsilon}(t, \omega, \v_{\varepsilon}(0))=e^{-\varepsilon\z(\vartheta_{t}\omega)}\u_{\varepsilon}(t, \omega, \u_{\varepsilon}(0))$ with $\v_{\varepsilon}(0)=e^{-\varepsilon\z(\omega)}\u_{\varepsilon}(0)$ where $\z$ is defined in \eqref{OU1} and satisfies \eqref{OU2}, and $\u_{\varepsilon}(\cdot)$ is the solution of \eqref{3SCBF-multi}. Then $\v_{\varepsilon}$ will satisfy the following:
	\begin{equation}\label{3CCBF-multi}
	\left\{
	\begin{aligned}
	\frac{\d\v_{\varepsilon}(t)}{\d t}+\mu \A\v_{\varepsilon}(t)&+e^{\varepsilon\z(\vartheta_{t}\omega)}\B\big(\v_{\varepsilon}(t)\big) +\beta e^{\varepsilon(r-1)\z(\vartheta_{t}\omega)}\mathcal{C}\big(\v_{\varepsilon}(t)\big)\\&=\f e^{-\varepsilon\z(\vartheta_{t}\omega)} + \varepsilon\alpha\z(\vartheta_t\omega)\v_{\varepsilon} , \ t\geq 0, \\ 
	\v_{\varepsilon}(0)&=e^{-\varepsilon\z(\omega)}\u_{\varepsilon}(0).
	\end{aligned}
	\right.
	\end{equation}
	\subsection{Uniform estimates of solutions}
	The following lemma is useful for our main result of this section.
	\begin{lemma}\label{3D-uniq}
		For $n=3 \text{ and } r\geq3  \ (\text{for } r>3 \text{ with any } \beta,\mu>0\text{ and for } r=3 \text{ with } 2\beta\mu\geq1)$, let $\f\in \H$. Then for any $\varepsilon\in (0,1]$, the solution $\v_{\varepsilon}$ of \eqref{CCBF-multi} with initial data $\v_{\varepsilon}(0)\in \H$ satisfies:
		\begin{itemize}
			\item [(i)] For $r>3$ with any $\beta,\mu>0$, 	\begin{align}\label{3M-uni}
			&\|\v_{\varepsilon}(s,\vartheta_{-t}\omega,\v_{\varepsilon}(0))\|^2_{\V}\nonumber\\&\leq Ce^{-\frac{\varrho_1}{8}s+\int_{-t}^{s-t}2\varepsilon\alpha\z(\vartheta_{\upeta}\omega)\d \upeta}\|\v_{\varepsilon}(0)\|^2_{\H} + Ce^{\frac{\varrho_1}{8}(t-s)} \int_{-t}^{s-t}e^{-2\varepsilon\z(\vartheta_{\upeta}\omega)+\frac{\varrho_1}{8}\upeta-\int_{s-t}^{\upeta}2\varepsilon\alpha\z(\vartheta_{\zeta}\omega)\d \zeta}\d \upeta,  
			\end{align}
			for $t\geq s>1,$ where $\varrho_1$ is the constant given by \eqref{3D-C_2}. 
			\item [(ii)] For $r=3$ with $2\beta\mu\geq1$, 	\begin{align}\label{3M-uni*}
			&\|\v_{\varepsilon}(s,\vartheta_{-t}\omega,\v_{\varepsilon}(0))\|^2_{\V}\nonumber\\&\leq Ce^{-\frac{\varrho_2}{8}s+\int_{-t}^{s-t}2\varepsilon\alpha\z(\vartheta_{\upeta}\omega)\d \upeta}\|\v_{\varepsilon}(0)\|^2_{\H} + Ce^{\frac{\varrho_2}{8}(t-s)} \int_{-t}^{s-t}e^{-2\varepsilon\z(\vartheta_{\upeta}\omega)+\frac{\varrho_2}{8}\upeta-\int_{s-t}^{\upeta}2\varepsilon\alpha\z(\vartheta_{\zeta}\omega)\d \zeta}\d \upeta,  
			\end{align}
			for $t\geq s>1,$	where $\varrho_2$ is the constant given by \eqref{3D-C_4}.
		\end{itemize} 
		
	\end{lemma}
	\begin{proof}
		From \eqref{3D-C_2} and \eqref{3D-C_4}, it is clear that $\frac{\varrho_i}{8}<\mu\lambda_1,$ for $i\in\{1,2\}$. Therefore calculation  similar to \eqref{M-uni5} gives 
		\begin{align}\label{3M-uni5}
		&\|\v_{\varepsilon}(s,\vartheta_{-t}\omega,\v_{\varepsilon}(0))\|^2_{\H}+\frac{\mu}{2}\int_{0}^{s}e^{\int_{s}^{\upeta}\left(\frac{\varrho_i}{8}-2\varepsilon\alpha\z(\vartheta_{\zeta-t}\omega)\right)\d \zeta}\|\v_{\varepsilon}(\upeta,\vartheta_{-t}\omega,\v_{\varepsilon}(0))\|^2_{\V}\d\upeta\nonumber\\&+\beta\int_{0}^{s}e^{\varepsilon(r-1)\z(\vartheta_{\upeta-t}\omega)+\int_{s}^{\upeta}\left(\frac{\varrho_i}{8}-2\varepsilon\alpha\z(\vartheta_{\zeta-t}\omega)\right)\d \zeta}\|\v_{\varepsilon}(\upeta,\vartheta_{-t}\omega,\v_{\varepsilon}(0))\|^{r+1}_{\wi \L^{r+1}}\d\upeta\nonumber\\&\leq e^{-\frac{\varrho_i}{8}s+\int_{-t}^{s-t}2\varepsilon\alpha\z(\vartheta_{\upeta}\omega)\d \upeta}\|\v_{\varepsilon}(0)\|^2_{\H} + \frac{2\|\f\|^2_{\H}}{\mu\lambda_1}e^{\frac{\varrho_i}{8}(t-s)} \int_{-t}^{s-t}e^{-2\varepsilon\z(\vartheta_{\upeta}\omega)+\frac{\varrho_i}{8}\upeta-\int_{s-t}^{\upeta}2\varepsilon\alpha\z(\vartheta_{\zeta}\omega)\d \zeta}\d \upeta,
		\end{align} for any $t\geq s\geq0$.
		Taking the inner product of the first equation in \eqref{3CCBF-multi} with $\A\v_{\varepsilon}(\cdot)$, we get 
		\begin{align}\label{3M-uni6}
		&\frac{1}{2}\frac{\d}{\d t}\|\v_{\varepsilon}(t)\|^2_{\V}+\mu\|\A\v_{\varepsilon}(t)\|^2_{\H} +\beta e^{\varepsilon(r-1)\z(\vartheta_{t}\omega)}\int_{\mathbb{T}^3}|\nabla\v_{\varepsilon}(t,x)|^2|\v_{\varepsilon}(t,x)|^{r-1}\d x\nonumber\\&\quad+4\beta e^{\varepsilon(r-1)\z(\vartheta_{t}\omega)}\left[\frac{r-1}{(r+1)^2}\right]\int_{\mathbb{T}^3}|\nabla|\v_{\varepsilon}(t,x)|^{\frac{r+1}{2}}|^2\d x\nonumber\\&=-e^{\varepsilon\z(\vartheta_{t}\omega)}\left(\B\big(\v_{\varepsilon}(t)\big), \A\v_{\varepsilon}(t)\right)+ e^{-\varepsilon\z(\vartheta_{t}\omega)}(\f,\A\v_{\varepsilon}(t)) + \varepsilon\alpha\z(\vartheta_{t}\omega)\|\v_{\varepsilon}(t)\|^2_{\V},
		\end{align}
		where we used 
		\begin{align*}
		\big(\mathcal{C}\big(\v_{\varepsilon}\big),\A\v_{\varepsilon}\big)&=\int_{\mathbb{T}^3}|\nabla\v_{\varepsilon}(x)|^2|\v_{\varepsilon}(x)|^{r-1}\d x+4\left[\frac{r-1}{(r+1)^2}\right]\int_{\mathbb{T}^3}|\nabla|\v_{\varepsilon}(x)|^{\frac{r+1}{2}}|^2\d x,
		\end{align*}
		for a.e. $t\in[0,\tau]$. Applying Cauchy-Schwarz and Young's inequalities, we obtain
		\begin{align}\label{3M-uni7}
		e^{-\varepsilon\z(\vartheta_{t}\omega)}(\f,\A\v_{\varepsilon}) \leq \frac{e^{-2\varepsilon\z(\vartheta_{t}\omega)}\|\f\|^2_{\H}}{\mu}  +\frac{\mu}{4}\|\A\v_{\varepsilon}\|^2_{\H}.
		\end{align}	
		\vskip 0.2 cm
		\noindent
		\textbf{Case I:}  $r>3$ with any $\beta,\mu>0$.
		Using similar calculations as in \eqref{3d-ab14}, we find 
		\begin{align}\label{3M-uni8}
		|e^{\varepsilon\z(\vartheta_{t}\omega)}(\B(\v_{\varepsilon}),\A\v_{\varepsilon})|&\leq\frac{\mu}{4}\|\A\v_{\varepsilon}\|_{\H}^2+\frac{\beta}{2}e^{\varepsilon(r-1)\z(\vartheta_{t}\omega)}\||\nabla\v_{\varepsilon}||\v_{\varepsilon}|^{\frac{r-1}{2}}\|^2_{\H}+\eta_3\|\v_{\varepsilon}\|^2_{\V}.
		\end{align}
		Using \eqref{3M-uni7} and \eqref{3M-uni8} in \eqref{3M-uni6}, we get 
		\begin{align}\label{3M-uni9}
		\frac{1}{2}\frac{\d}{\d t}\|\v_{\varepsilon}(t)\|^2_{\V}+\frac{\mu}{2}\|\A\v_{\varepsilon}(t)\|^2_{\H}\leq \frac{e^{-2\varepsilon\z(\vartheta_{t}\omega)}\|\f\|^2_{\H}}{\mu} + \varepsilon\alpha\z(\vartheta_{t}\omega)\|\v_{\varepsilon}(t)\|^2_{\V} +\eta_3\|\v_{\varepsilon}(t)\|^2_{\V},
		\end{align}
		for a.e. $t\in[0,\tau]$. Also, using \eqref{poin_1} in \eqref{3M-uni9}, we have 
		\begin{align*}
		\frac{\d}{\d t}\|\v_{\varepsilon}(t)\|^2_{\V}+\mu\lambda_1\|\v_{\varepsilon}(t)\|^2_{\V}&\leq \frac{2e^{-2\varepsilon\z(\vartheta_{t}\omega)}\|\f\|^2_{\H}}{\mu} + 2\varepsilon\alpha\z(\vartheta_{t}\omega)\|\v_{\varepsilon}(t)\|^2_{\V}+2\eta_3\|\v_{\varepsilon}(t)\|^2_{\V}.
		\end{align*}
		Thus, it is immediate that 
		\begin{align*}
		\frac{\d}{\d t}\|\v_{\varepsilon}(t)\|^2_{\V}+\left(\mu\lambda_1-2\varepsilon\alpha\z(\vartheta_{t}\omega)\right)\|\v_{\varepsilon}(t)\|^2_{\V}\leq \frac{2e^{-2\varepsilon\z(\vartheta_{t}\omega)}\|\f\|^2_{\H}}{\mu}+2\eta_3\|\v_{\varepsilon}(t)\|^2_{\V}.
		\end{align*}
		From \eqref{3D-C_2}, it is clear that $\frac{\varrho_1}{8}<\mu\lambda_1$,  and hence 
		\begin{align*}
		\frac{\d}{\d t}\|\v_{\varepsilon}(t)\|^2_{\V}+\left(\frac{\varrho_1}{8}-2\varepsilon\alpha\z(\vartheta_{t}\omega)\right)\|\v_{\varepsilon}(t)\|^2_{\V}\leq \frac{2e^{-2\varepsilon\z(\vartheta_{t}\omega)}\|\f\|^2_{\H}}{\mu}+2\eta_3\|\v_{\varepsilon}(t)\|^2_{\V},
		\end{align*}
		for a.e. $t\in[0,\tau]$. Using the same arguments that we used to obtain \eqref{M-uni5}, for $t\geq s\geq \uprho>0,$ we have 
		\begin{align}\label{3M-uni11}
		\|\v_{\varepsilon}(s,\vartheta_{-t}\omega,\v_{\varepsilon}(0))\|^2_{\V}&\leq e^{-\int_{\uprho}^{s}\left(\frac{\varrho_1}{8}-2\varepsilon\alpha\z(\vartheta_{\upeta-t}\omega)\right)\d \upeta}\|\v_{\varepsilon}(\uprho,\vartheta_{-t}\omega,\v_{\varepsilon}(0))\|^2_{\V}\nonumber\\&\quad + \frac{2\|\f\|^2_{\H}}{\mu} \int_{\uprho}^{s}e^{-2\varepsilon\z(\vartheta_{\upeta-t}\omega)+\int_{s}^{\upeta}\left(\frac{\varrho_1}{8}-2\varepsilon\alpha\z(\vartheta_{\zeta-t}\omega)\right)\d \zeta}\d \upeta\nonumber\\&\quad+2\eta_3\int_{\uprho}^{s}e^{\int_{s}^{\upeta}\left(\frac{\varrho_1}{8}-2\varepsilon\alpha\z(\vartheta_{\zeta-t}\omega)\right)\d \zeta}\|\v_{\varepsilon}(\upeta,\vartheta_{-t}\omega,\v_{\varepsilon}(0))\|^2_{\V}\d \upeta\nonumber\\&\leq e^{-\int_{\uprho}^{s}\left(\frac{\varrho_1}{8}-2\varepsilon\alpha\z(\vartheta_{\upeta-t}\omega)\right)\d \upeta}\|\v_{\varepsilon}(\uprho,\vartheta_{-t}\omega,\v_{\varepsilon}(0))\|^2_{\V} \nonumber\\&\quad+ \frac{2\|\f\|^2_{\H}}{\mu} \int_{0}^{s}e^{-2\varepsilon\z(\vartheta_{\upeta-t}\omega)+\int_{s}^{\upeta}\left(\frac{\varrho_1}{8}-2\varepsilon\alpha\z(\vartheta_{\zeta-t}\omega)\right)\d \zeta}\d \upeta\nonumber\\&\quad+2\eta_3\int_{0}^{s}e^{\int_{s}^{\upeta}\left(\frac{\varrho_1}{8}-2\varepsilon\alpha\z(\vartheta_{\zeta-t}\omega)\right)\d \zeta}\|\v_{\varepsilon}(\upeta,\vartheta_{-t}\omega,\v_{\varepsilon}(0))\|^2_{\V}\d \upeta.
		\end{align}
		Integrating \eqref{3M-uni11} with respect to $\uprho$ over $(s-1,s),$ for $s>1$ and using \eqref{3M-uni5}, we get 
		\begin{align}\label{3M-uni12}
		\|\v_{\varepsilon}(s,\vartheta_{-t}\omega,\v_{\varepsilon}(0))\|^2_{\V}&\leq \int_{s-1}^{s}e^{-\int_{\uprho}^{s}\left(\frac{\varrho_1}{8}-2\varepsilon\alpha\z(\vartheta_{\upeta-t}\omega)\right)\d \upeta}\|\v_{\varepsilon}(\uprho,\vartheta_{-t}\omega,\v_{\varepsilon}(0))\|^2_{\V}\d \uprho \nonumber\\&\quad+ \frac{\|\f\|^2_{\H}}{\mu} \int_{0}^{s}e^{-2\varepsilon\z(\vartheta_{\upeta-t}\omega)+\int_{s}^{\upeta}\left(\frac{\varrho_1}{8}-2\varepsilon\alpha\z(\vartheta_{\zeta-t}\omega)\right)\d \zeta}\d \upeta\nonumber\\&\quad+2\eta_3\int_{0}^{s}e^{\int_{s}^{\upeta}\left(\frac{\varrho_1}{8}-2\varepsilon\alpha\z(\vartheta_{\zeta-t}\omega)\right)\d \zeta}\|\v_{\varepsilon}(\upeta,\vartheta_{-t}\omega,\v_{\varepsilon}(0))\|^2_{\V}\d \upeta\nonumber\\&\leq Ce^{-\int_{0}^{s}\left(\frac{\varrho_1}{8}-2\varepsilon\alpha\z(\vartheta_{\upeta-t}\omega)\right)\d \upeta}\|\v_{\varepsilon}(0)\|^2_{\H} \nonumber\\&\quad+ C \int_{0}^{s}e^{-2\varepsilon\z(\vartheta_{\upeta-t}\omega)+\int_{s}^{\upeta}\left(\frac{\varrho_1}{8}-2\varepsilon\alpha\z(\vartheta_{\zeta-t}\omega)\right)\d \zeta}\d \upeta\nonumber\\&= Ce^{-\frac{\varrho_1}{8}s+\int_{-t}^{s-t}2\varepsilon\alpha\z(\vartheta_{\upeta}\omega)\d \upeta}\|\v_{\varepsilon}(0)\|^2_{\H} \nonumber\\&\quad+ Ce^{\frac{\varrho_1}{8}(t-s)} \int_{-t}^{s-t}e^{-2\varepsilon\z(\vartheta_{\upeta}\omega)+\frac{\varrho_1}{8}\upeta-\int_{s-t}^{\upeta}2\varepsilon\alpha\z(\vartheta_{\zeta}\omega)\d \zeta}\d \upeta,
		\end{align}
		for $t\geq s>1$, which completes the proof. 
		\vskip 0.2 cm
		\noindent
		\textbf{Case II:}  $r=3$ with $2\beta\mu\geq1$. Applying the Cauchy-Schwarz and Young inequalities, we find 
		\begin{align}
		|(\B(\v_{\varepsilon}),\A\v_{\varepsilon})|&\leq\||\v_{\varepsilon}||\nabla\v_{\varepsilon}|\|_{\H}\|\A\v_{\varepsilon}\|_{\H}\leq\frac{1}{4\beta}\|\A\v_{\varepsilon}\|_{\H}^2+\beta\||\v_{\varepsilon}||\nabla\v_{\varepsilon}|\|_{\H}^2.\label{3M-uni13}
		\end{align}
		Using \eqref{3M-uni7} and \eqref{3M-uni13} in \eqref{3M-uni6}, we obtain
		\begin{align}\label{3M-uni14}
		\frac{1}{2}&\frac{\d}{\d t}\|\v_{\varepsilon}(t)\|^2_{\V}+\frac{\mu}{4}\|\A\v_{\varepsilon}(t)\|^2_{\H}+\frac{1}{2}\left(\mu-\frac{1}{2\beta}\right)\|\A\v_{\varepsilon}(t)\|^2_{\H}\nonumber\\&\leq \frac{e^{-2\varepsilon\z(\vartheta_{t}\omega)}\|\f\|^2_{\H}}{\mu} + \varepsilon\alpha\z(\vartheta_{t}\omega)\|\v_{\varepsilon}(t)\|^2_{\V} ,
		\end{align}
		for a.e. $t\in[0,\tau]$. Also, using \eqref{poin_1} with \eqref{3M-uni14}, we have 
		\begin{align*}
		\frac{\d}{\d t}\|\v_{\varepsilon}(t)\|^2_{\V}+\frac{\mu\lambda_1}{2}\|\v_{\varepsilon}(t)\|^2_{\V}\leq \frac{2e^{-2\varepsilon\z(\vartheta_{t}\omega)}\|\f\|^2_{\H}}{\mu} + 2\varepsilon\alpha\z(\vartheta_{t}\omega)\|\v_{\varepsilon}(t)\|^2_{\V}.
		\end{align*}
		Therefore, we get 
		\begin{align*}
		\frac{\d}{\d t}\|\v_{\varepsilon}(t)\|^2_{\V}+\left(\frac{\mu\lambda_1}{2}-2\varepsilon\alpha\z(\vartheta_{t}\omega)\right)\|\v_{\varepsilon}(t)\|^2_{\V}\leq \frac{2e^{-2\varepsilon\z(\vartheta_{t}\omega)}\|\f\|^2_{\H}}{\mu}. 
		\end{align*}
		From \eqref{3D-C_4}, it is clear that $\frac{\varrho_2}{4}<\mu\lambda_1$, so that 
		\begin{align*}
		\frac{\d}{\d t}\|\v_{\varepsilon}(t)\|^2_{\V}+\left(\frac{\varrho_2}{8}-2\varepsilon\alpha\z(\vartheta_{t}\omega)\right)\|\v_{\varepsilon}(t)\|^2_{\V}\leq \frac{2e^{-2\varepsilon\z(\vartheta_{t}\omega)}\|\f\|^2_{\H}}{\mu},\ \text{ for a.e. }\ t\in[0,\tau] ,
		\end{align*}
		Now, using the same argument as in the case of $r>3$, one can conclude the proof.
	\end{proof}
	\subsection{Perturbation radius of the singleton attractor under multiplicative noise}
	Let us now consider the difference $\w_{\varepsilon}=\v_{\varepsilon}-\u$, where $\v_{\varepsilon}$ and $\u$ are the solutions of \eqref{3CCBF-multi} and \eqref{D-CBF}, respectively. It is obvious that $\w_{\varepsilon}$ satisfies 
	\begin{equation}\label{3Diff-CBF_multi}
	\left\{
	\begin{aligned}
	\frac{\d\w_{\varepsilon}}{\d t}&=-\mu \A\w_{\varepsilon}-e^{\varepsilon\z(\vartheta_{t}\omega)}\B(\v_{\varepsilon})+\B(\u) -\beta e^{\varepsilon(r-1)\z(\vartheta_{t}\omega)}\mathcal{C}(\v_{\varepsilon})+\beta\mathcal{C}(\u)\\&\quad+ \left(e^{-\varepsilon\z(\vartheta_{t}\omega)}-1\right)\f+ \varepsilon\alpha\z(\vartheta_t\omega)\v_{\varepsilon} , \\ 
	\w_{\varepsilon}(0)&=(e^{-\varepsilon\z(\vartheta_{t}\omega)}-1)\x,
	\end{aligned}
	\right.
	\end{equation}
	in $\V'$. Due to technical difficulties, in  the 3D case, we restrict ourselves to $3\leq r\leq5$. 
	\begin{lemma}\label{3PertRad-multi}
		For $n=3 \text{ and } 3\leq r\leq5  \ (\text{for } 3<r\leq5 \text{ with any } \beta,\mu>0\text{ and for } r=3 \text{ with } 2\beta\mu\geq1)$, let $\f\in \H$. Then for any $\varepsilon\in(0,1]$, there exists a random variable $\gamma^{*}_{\varepsilon}(\omega)$ such that 
		\begin{align*}
		\limsup_{t\to \infty} \|\w_{\varepsilon}(t,\vartheta_{-t}\omega,\w_{\varepsilon}(0))\|^2_{\H}\leq \gamma^{*}_{\varepsilon}(\omega),\quad \varepsilon\in(0,1],\  \omega\in \Omega.
		\end{align*}
		Moreover, as a mapping of $\varepsilon$, $\gamma^{*}_{\varepsilon}(\omega)\sim\varepsilon^2$ as $\varepsilon\to0$, for every $\omega\in\Omega$.
	\end{lemma}
	\begin{proof}
		Taking the inner product of first the  equation in \eqref{3Diff-CBF_multi} with $\w_{\varepsilon}(t)$, we get (see \eqref{W-multi1})
		\begin{align}\label{3W-multi1}
		&\frac{1}{2}\frac{\d}{\d t}\|\w_{\varepsilon}(t)\|^2_{\H} \nonumber\\&=- \mu \|\w_{\varepsilon}(t)\|^2_{\V} -\left[(e^{\varepsilon\z(\vartheta_{t}\omega)}-1)b(\v_{\varepsilon}(t),\v_{\varepsilon}(t),\w_{\varepsilon}(t))+b(\w_{\varepsilon}(t),\v_{\varepsilon}(t),\w_{\varepsilon}(t))\right]\nonumber\\&\quad-\beta e^{\varepsilon(r-1)\z(\vartheta_{t}\omega)}\left\langle\mathcal{C}(\v_{\varepsilon}(t))-\mathcal{C}(\u(t)),\v_{\varepsilon}(t)-\u(t)\right\rangle-\beta(e^{\varepsilon(r-1)\z(\vartheta_{t}\omega)}-1)\left\langle\mathcal{C}(\u(t)),\w_{\varepsilon}(t)\right\rangle\nonumber\\&\quad+\left(e^{-\varepsilon\z(\vartheta_{t}\omega)}-1\right)(\f,\w_{\varepsilon}(t))+\varepsilon\alpha\z(\vartheta_t\omega)\|\w_{\varepsilon}(t)\|^2_{\H}+ (\varepsilon\alpha\z(\vartheta_t\omega)\u(t),\w_{\varepsilon}(t)),
		\end{align}
		for a.e. $t\in[0,\tau]$. Now, we estimate the each term of the right side of the equality \eqref{3W-multi1} carefully to get our required estimate. Using \eqref{b3}, \eqref{poin} and Young's inequality, we have 
		\begin{align}\label{3W-multi2}
		&|(e^{\varepsilon\z(\vartheta_{t}\omega)}-1)b(\v_{\varepsilon},\v_{\varepsilon},\w_{\varepsilon})+b(\w_{\varepsilon},\v_{\varepsilon},\w_{\varepsilon})|\nonumber\\&=|(e^{\varepsilon\z(\vartheta_{t}\omega)}-1)b(\v_{\varepsilon},\u,\w_{\varepsilon})+b(\w_{\varepsilon},\u,\w_{\varepsilon})|\nonumber\\&\leq C|e^{\varepsilon\z(\vartheta_{t}\omega)}-1|\|\v_{\varepsilon}\|_{\V}\|\u\|_{\V}\|\w_{\varepsilon}\|_{\V}+c_3\|\u\|_{\V}\|\w_{\varepsilon}\|^{1/2}_{\H}\|\w_{\varepsilon}\|^{3/2}_{\V}\nonumber\\&\leq C|e^{\varepsilon\z(\vartheta_{t}\omega)}-1|^2\|\v_{\varepsilon}\|^2_{\V}\|\u\|^2_{\V}+\frac{\varrho_i}{32\lambda_1}\|\w_{\varepsilon}\|^2_{\V} +\frac{27c_3^4}{32\mu^3}\|\u\|^4_{\V}\|\w_{\varepsilon}\|^2_{\H} +\frac{\mu}{2}\|\w_{\varepsilon}\|^2_{\V},
		\end{align}
		where $\rho_1$ and $\rho_2$ are defined in \eqref{3D-C_2} (for $r>3$) and \eqref{3D-C_4} (for $r=3$), respectively.	From \eqref{MO_c}, we have 
		\begin{align}\label{3W-multi3}
		-\beta e^{\varepsilon(r-1)\z(\vartheta_{t}\omega)}\left\langle\mathcal{C}(\v_{\varepsilon})-\mathcal{C}(\u),\v_{\varepsilon}-\u\right\rangle\leq 0.
		\end{align}
		Applying H\"older's and Young's inequalities, and Sobolev embedding (since $3\leq r\leq 5$), we obtain
		\begin{align}\label{3W-multi4}
		\left|(e^{\varepsilon(r-1)\z(\vartheta_{t}\omega)}-1)\left\langle\mathcal{C}(\u),\w_{\varepsilon}\right\rangle\right|&\leq\left|e^{\varepsilon(r-1)\z(\vartheta_{t}\omega)}-1\right|\|\u\|^{r}_{\wi\L^{r+1}}\|\w_{\varepsilon}\|_{\wi \L^{r+1}}\nonumber\\&\leq C\left|e^{\varepsilon(r-1)\z(\vartheta_{t}\omega)}-1\right|\|\u\|^{r}_{\V}\|\w_{\varepsilon}\|_{\V}\nonumber\\&\leq C\left|e^{\varepsilon(r-1)\z(\vartheta_{t}\omega)}-1\right|^2\|\u\|^{2r}_{\V}+\frac{\varrho_i}{32\lambda_1}\|\w_{\varepsilon}\|^2_{\V},
		\end{align}
		and 
		\begin{align}\label{3W-multi5}
		&\left(e^{-\varepsilon\z(\vartheta_{t}\omega)}-1\right)(\f,\w_{\varepsilon})+\varepsilon\alpha\z(\vartheta_t\omega)\|\w_{\varepsilon}\|^2_{\H}+ (\varepsilon\alpha\z(\vartheta_t\omega)\u,\w_{\varepsilon})\nonumber\\&\leq C\left|e^{-\varepsilon\z(\vartheta_{t}\omega)}-1\right|^2+\left(\frac{\varrho_i}{16}+\varepsilon\alpha\z(\vartheta_{t}\omega)\right)\|\w_{\varepsilon}\|^2_{\H}+\varepsilon^2C\left|\z(\vartheta_{t}\omega)\right|^2\|\u\|^2_{\H}.
		\end{align}
		Combining \eqref{3W-multi2}-\eqref{3W-multi5} and using it in \eqref{3W-multi1}, we deduce that 
		\begin{align*}
		&\frac{1}{2}\frac{\d}{\d t}\|\w_{\varepsilon}(t)\|^2_{\H}+\left(\frac{\mu}{2}-\frac{\varrho_i}{16\lambda_1}\right)\|\w_{\varepsilon}(t)\|^2_{\V}-\left(\frac{\varrho_i}{16}+\varepsilon\alpha\z(\vartheta_{t}\omega)+\frac{27c_3^4}{32\mu^3}\|\u(t)\|^4_{\V}\right)\|\w_{\varepsilon}(t)\|^2_{\H}\nonumber\\&\leq\varepsilon^2C\left|\z(\vartheta_{t}\omega)\right|^2\|\u(t)\|^2_{\H}+	C|e^{\varepsilon\z(\vartheta_{t}\omega)}-1|^2\|\v_{\varepsilon}(t)\|^2_{\V}\|\u(t)\|^2_{\V}  +C\left|e^{\varepsilon(r-1)\z(\vartheta_{t}\omega)}-1\right|^2\|\u(t)\|^{2r}_{\V}\nonumber\\&\quad+C\left|e^{-\varepsilon\z(\vartheta_{t}\omega)}-1\right|^2,
		\end{align*}
		for a.e. $t\in[0,\tau]$. By \eqref{poin}, we have
		\begin{align}\label{3W-multi6}
		&\frac{\d}{\d t}\|\w_{\varepsilon}(t)\|^2_{\H}+\left(\mu\lambda_1-\frac{27c_3^4}{16\mu^3}\|\u(t)\|^4_{\V}-\frac{\varrho_i}{4}-2\varepsilon\alpha\z(\vartheta_{t}\omega)\right)\|\w_{\varepsilon}(t)\|^2_{\H}\nonumber\\&\leq\varepsilon^2C\left|\z(\vartheta_{t}\omega)\right|^2\|\u(t)\|^2_{\H}+	C|e^{\varepsilon\z(\vartheta_{t}\omega)}-1|^2\|\v_{\varepsilon}(t)\|^2_{\V}\|\u(t)\|^2_{\V}  +C\left|e^{\varepsilon(r-1)\z(\vartheta_{t}\omega)}-1\right|^2\|\u(t)\|^{2r}_{\V}\nonumber\\&\quad+C\left|e^{-\varepsilon\z(\vartheta_{t}\omega)}-1\right|^2.
		\end{align}
		Now, consider again a time $T_i=T_{B_{\H},\frac{8\varrho_i\mu^3}{27c_3^4}}>0$ such that \eqref{3d-uni5} and \eqref{3d-uni7} hold for $i=1$ and $i=2$ respectively, that is,
		\begin{align*}
		\mu\lambda_1-\frac{27c_3^4}{16\mu^3}\|\u\|^4_{\V}\geq\frac{\varrho_i}{2}, \ \text{ for all }\  t\geq T_i.
		\end{align*}
		It follows from \eqref{3W-multi6} that 
		\begin{align}\label{3W-multi7}
		&\frac{\d}{\d t}\|\w_{\varepsilon}\|^2_{\H}+\left(\frac{\varrho_i}{4}-2\varepsilon\alpha\z(\vartheta_{t}\omega)\right)\|\w_{\varepsilon}\|^2_{\H}\nonumber\\&\leq\varepsilon^2C\left|\z(\vartheta_{t}\omega)\right|^2+	C|e^{\varepsilon\z(\vartheta_{t}\omega)}-1|^2\|\v_{\varepsilon}\|^2_{\V} +C\left|e^{\varepsilon(r-1)\z(\vartheta_{t}\omega)}-1\right|^2+C\left|e^{-\varepsilon\z(\vartheta_{t}\omega)}-1\right|^2, 
		\end{align}
		for a.e. $t\in[T_i,\tau]$. Note that the situation of \eqref{3W-multi7} is same as that of \eqref{W-multi7}, therefore one can conclude the proof in a similar manner using Lemma \ref{3D-uniq}. 
	\end{proof}
	\begin{theorem}\label{3D-Conver-multi}
		For $n=3 \text{ and } 3\leq r\leq5  \ (\text{for } 3<r\leq5 \text{ with any } \beta,\mu>0\text{ and for } r=3 \text{ with } 2\beta\mu\geq1)$, let $\f\in\H$. Then, there exists a random variable $\gamma^{*}_{\varepsilon}(\omega)$ such that for every $\varepsilon\in(0,1],$ the multiplicative noise random attractor $\mathcal{A}_{\varepsilon}$ and the deterministic attractor $\mathcal{A}$ satisfy 
		\begin{align*}
		\emph{dist}_{\H}\left(\mathcal{A}_{\varepsilon}(\omega),\mathcal{A}\right)\leq\sqrt{\gamma^{*}_{\varepsilon}(\omega)},
		\end{align*}
		where $\emph{dist}_{\H}(A,B)=\max\{\emph{dist}(A,B),\emph{dist}(B,A)\}$. Moreover, as a mapping of $\varepsilon$, $\sqrt{\gamma^{*}_{\varepsilon}(\omega)}\sim\varepsilon$ as $\varepsilon\to0$ for every $\omega\in\Omega$.
	\end{theorem}
	\begin{proof}
		Proof of this theorem is analogous to the proof of Theorem \ref{Conver-add}.  Using Theorem \ref{3PertRad-multi}, proof can be completed. 
	\end{proof}

	\medskip\noindent
	{\bf Acknowledgments:}    The first author would like to thank the Council of Scientific $\&$ Industrial Research (CSIR), India for financial assistance (File No. 09/143(0938)/2019-EMR-I).  M. T. Mohan would  like to thank the Department of Science and Technology (DST), Govt of India for Innovation in Science Pursuit for Inspired Research (INSPIRE) Faculty Award (IFA17-MA110).

\end{document}